%% file: MAIN-FILE.tex
\documentclass[11pt,a4paper,dvips]{amsart}
\usepackage[T1]{fontenc}
\usepackage{amscd,amsfonts,amsmath,amssymb,amsthm,amstext,latexsym,amstext}
\usepackage{enumerate,pifont,graphicx}
\usepackage{epsfig}
\usepackage[english]{babel}
\newfont{\roc}{eusm10 at 12pt}

\def\a{{\alpha}}
\def\e{{\varepsilon}}
\def\g{{\gamma}}
\def\b{{\beta}}
\def\s{{\sigma}}

\newcommand{\cL}{{\mathcal L}}
\newcommand{\Z}{{\mathbb Z}}

\newcommand{\R}{{\mathbb R}}
\newcommand{\C}{{\mathbb C}}
\newcommand{\N}{{\mathbb N}}

\newcommand{\bL}{{\mathbb L}}

\def\esf{\mathbb{S}}

\def\t{{\theta}}
\def\l{{\lambda}}

\def\ve{{\varepsilon}}

\newtheorem{theorem}{Theorem}[section]
\newtheorem{lemma}[theorem]{Lemma}
\newtheorem{proposition}[theorem]{Proposition}
\newtheorem{remark}[theorem]{Remark}

\newtheorem{assertion}[theorem]{Assertion}
\newtheorem{definition}[theorem]{Definition}

\renewenvironment{proof}{\smallskip\noindent{\bf Proof.}\hskip \labelsep}
                      {\hfill\penalty10000\raisebox{-.09em}{$\Box$}\par\medskip}
\newcommand{\cal}{\mathcal}

\begin{document}
\title{An end-to-end construction for singly periodic minimal
  surfaces}
 
\author{L. Hauswirth}
  \address{Laurent Hauswirth \newline Universit\'e Paris-Est, 
Laboratoire d'Analyse et Math\'ematiques Appliqu\'ees,
5 blvd Descartes, 77454 Champs-sur-Marne, FRANCE}
\email{laurent.hauswirth@univ-mlv.fr}
\author{ F. Morabito}
   \address{Filippo Morabito \newline Universit\'e Paris-Est, 
Laboratoire d'Analyse et Math\'ematiques Appliqu\'ees,
5 blvd Descartes, 77454 Champs-sur-Marne, FRANCE \newline and \newline
Università Roma Tre, 
Dipartimento di Matematica, Largo S.L. Murialdo 1, 00146 Roma, ITALY}
\email{morabito@mat.uniroma3.it}
\author{ M. M. Rodr\'\i guez  }
\address{M. Magdalena Rodr\'\i guez
  \newline  
Universidad Complutense de Madrid,
Departamento de Álgebra, P.za de las Ciencias 3,  
28040 Madrid, SPAIN }
\email{magdalena@mat.ucm.es}

\thanks{Third author research activity
 is partially supported by grants from R\'egion
 Ile-de-France and a MEC/FEDER grant no. MTM2007-61775.}
 
\begin{abstract}
We show the existence of various families of properly
embedded singly periodic minimal surfaces in $\R^3$ with
finite arbitrary genus and Scherk type ends in the quotient. 
The proof of our results is based on the gluing of small perturbations
of pieces of already known minimal surfaces.
\end{abstract}
 
\maketitle

\section{Introduction}
Besides the plane and the helicoid, the first singly periodic minimal
surface was discovered by Scherk~\cite{S} in 1835. This surface, known
as {\it Scherk's second surface}, is a properly embedded minimal
surface in $\R^3$ invariant by one translation $T$ we can assume along
the $x_{2}$-axis, and can be seen as the desingularization of two
perpendicular planes $P_{1}$ and $P_{2}$ containing the $x_2$-axis. We
assume $P_1,P_2$ are symmetric with respect to the planes $\{x_1=0\}$
and $\{x_3=0\}$.  By changing the angle between $P_1,P_2$ we obtain a
1-parameter family of properly embedded singly periodic minimal
surfaces, we will refer to as {\it Scherk surfaces}. In the quotient
$\R^3/T$ by its shortest translation $T$, each Scherk surface has
genus zero and four ends asymptotic to flat annuli contained in
$P_1/T, P_2/T$. Such ends are called Scherk-type ends.\\

In 1982, C.~Costa~\cite{C1,C2} discovered a genus one minimal surface
with three embedded ends: one top catenoidal end, one middle planar
end and one bottom catenoidal end. D.~Hoffmann and
W.H.~Meeks~\cite{HM0,HM1,HM2} proved the global embeddedness for this
Costa example, and generalized it for higher genus.  For each $k\geqslant
1$, Costa-Hoffman-Meeks surface $M_k$ (we will abbreviate by saying
{\it CHM example}) is a properly embedded minimal surface of genus $k$ and
three ends: two catenoidal ones and one middle planar end.\\

F.~Martin and V.~Ramos Batista~\cite{mara} have recently constructed a
properly embedded singly periodic minimal example which has genus one
and six Scherk-type ends in the quotient $\R^3/T$, called {\it Scherk-Costa
surface}, based on Costa surface (from now on, $T$ will denote a
translation in the $x_2$-direction). Roughly speaking, they remove
each end of Costa surface (asymptotic to a catenoid or a plane) and
replace it by two Scherk-type ends.  In this paper we obtain surfaces
in the same spirit as Martin and Ramos Batista's one, but with a
completely different method.  We construct properly embedded singly
periodic minimal surfaces with genus $k\geqslant 1$ and six Scherk-type
ends in the quotient $\R^3/T$, by gluing (in an analytic way) a
compact piece of $M_{k}$ to two halves of a Scherk surface at the top
and bottom catenoidal ends, and one flat horizontal annulus $P/T$ with
a disk removed at the middle planar end.

\begin{theorem}\label{th1}
  Let $T$ denote a translation in the $x_2$-direction.  For each $k
  \geqslant 1$, there exists a 1-parameter family of properly embedded
  singly periodic minimal surfaces in $\R^3$ invariant by $T$ whose
  quotient in $\R^3/T$ has genus $k$ and six Scherk-type ends.
\end{theorem}

V. Ramos Batista~\cite{B} constructed a singly periodic Costa minimal
surface with two catenoidal ends and two Scherk-type middle end, which
has genus one in the quotient $\R^3/T$.  This example is not embedded
outside a slab in $\R^3/T$ which contains the topology of the surface.
We observe that the surface we obtain by gluing a compact piece of
$M_1$ (Costa surface) at its middle planar end to a flat horizontal
annulus with a disk removed has the same properties as Ramos Batista's
one.\\

In 1988, H. Karcher~\cite{K1,K2} defined a family of properly embedded
doubly periodic minimal surfaces, called {\it toroidal halfplane
  layers}, which has genus one and four horizontal Scherk-type ends in
the quotient.  In 1989, W.H. Meeks and H. Rosenberg \cite{MR1} developed a
general theory for doubly periodic minimal surfaces having finite
topology in the quotient, and used an approach of minimax type to
obtain the existence of a family of properly embedded doubly periodic
minimal surfaces, also with genus one and four horizontal Scherk-type
ends in the quotient.  These Karcher's and Meeks and Rosenberg's
surfaces have been generalized in~\cite{Ro}, constructing a
3-parameter family ${\cal K}=\{M_{\sigma,\a,\b}\}_{\sigma,\a,\b}$ of
surfaces, called {\it KMR examples} (sometimes, they are also referred in
the literature as toroidal halfplane layers). Such examples have been
classified by J. P\'erez, M.M. Rodr\'{\i}guez and M. Traizet~\cite{PRT} as the
only properly embedded doubly periodic minimal surfaces with genus one
and finitely many parallel (Scherk-type) ends in the quotient.  Each
$M_{\sigma,\a,\b}$ is invariant by a horizontal translation $T$ (by
the period vector at its ends) and a non horizontal one $\widetilde
T$.  We denote by $\widetilde M_{\sigma,\a,\b}$ the lifting of
$M_{\sigma,\a,\b}$ to $\R^3/T$, which has genus zero, infinitely many
horizontal Scherk-type ends, and two limit ends.\\

In 1992, F.S.~Wei~\cite{wei2} added a handle to a KMR example
$M_{\sigma,0,0}$ in a periodic way, obtaining a properly embedded
doubly periodic minimal surface invariant under reflection in three
orthogonal planes, which has genus two and four horizontal Scherk-type
ends in the quotient.  Some years later, W.~Rossman, E.C.~Thayer and
M.~Wolgemuth~\cite{rossthwo} added a different type of handle to a KMR
example $M_{\sigma,0,0}$, also in a periodic way, obtaining a
different minimal surfaces with the same properties as Wei's one. They
also added two handles to a KMR example, getting doubly periodic
examples of genus three in the quotient.  L.~Mazet and M.
Traizet~\cite{maztr} have recently added $N\geqslant 1$ handles to a
KMR example $M_{\s,0,0}$, obtaining a genus $N$ properly embedded
minimal surface in $\R^3/T$ with an infinite number of horizontal
Scherk-type ends and two limit ends. The idea of the construction
is to connect $N$  periods of the doubly periodic example of Wei with
two halves KMR example. However they only control the 
asymptotic behavior in their construction. 
 They have also constructed a
properly embedded minimal surface in $\R^3/T$ with infinite genus,
adding handles in a quasi-periodic way to a KMR example.\\

L. Hauswirth and F. Pacard~\cite{HP} have constructed higher genus
Riemann minimal examples in $\R^3$, by gluing two halves of a Riemann
minimal example with the intersection of a conveniently chosen CHM
surface with a slab. We follow their ideas to generalize Mazet and
Traizet's examples by constructing higher genus KMR examples: we
construct properly embedded singly periodic minimal examples whose
quotient in $\R^3/T$ has arbitrary finite genus, infinitely many
horizontal Scherk-type ends and two limit ends.  More precisely, we
glue a compact piece of a slightly deformed CHM example $M_{k}$ with
tilted catenoidal ends, to two halves of a KMR example $M_{\s,\a,0}$
or $M_{\s,0,\b}$ (see Figure~\ref{KMR+CHM}) and a periodic horizontal
flat annuli with a disk removed.

\begin{theorem}\label{th2}
  Let $T$ denote a translation in the $x_2$-direction.  For each $k
  \geqslant 1$, there exist two 1-parameter families ${\cal K}_1,{\cal
    K}_2$ of properly embedded singly periodic minimal surfaces in
  $\R^3$ whose quotient in $\R^3/T$ has genus $k$, infinitely many
  horizontal Scherk-type ends and two limit ends.  The surfaces in
  ${\cal K}_1$ have a vertical plane of symmetry orthogonal to the
  $x_1$-axis, and the surfaces in ${\cal K}_2$ have a vertical plane
  of symmetry orthogonal to the $x_2$-axis.
\end{theorem}

\begin{figure}\begin{center}
\epsfysize=4cm \epsffile{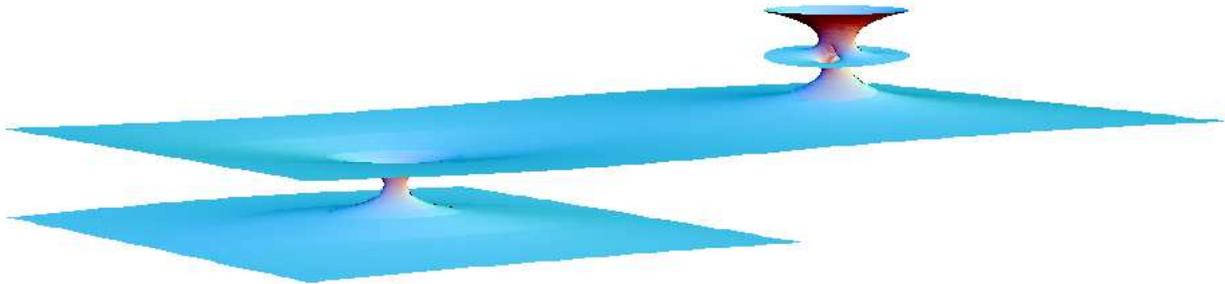}
\end{center}\caption{A sketch of half a KMR example $M_{\s,0,0}$ glued
  to a compact piece of Costa surface.}
\label{KMR+CHM}
\end{figure}

L.  Mazet, M. Traizet and M. Rodr\'\i guez \cite{mazrodtr} have recently
constructed saddle towers with infinitely many ends: they are
non-periodic properly embedded minimal surfaces in $\R^3/T$ with
infinitely many ends and one limit end. In the present paper, we
construct (non-periodic) properly embedded minimal surfaces in
$\R^3/T$ with arbitrary finite genus $k\geqslant 0$, infinitely many ends
and one limit end. With this aim, we glue half a Scherk example with
half a KMR example, in the case $k=0$; and, when $k\geqslant 1$, we glue a
compact piece of the CHM example $M_k$ to half a Scherk surface (at
the top catenoidal end of $M_k$), a periodic horizontal flat annuli
with a disk removed (at the middle planar end) and half a KMR example
(at the bottom catenoidal end), see Figure~\ref{sketch}.

\begin{theorem}\label{th3}
  Let $T$ denote a translation in the $x_2$-direction.  For each $k
  \geqslant 0$, there exists a 1-parameter family of properly embedded
  singly periodic minimal surfaces in $\R^3$ whose quotient in
  $\R^3/T$ has genus $k$, infinitely many horizontal Scherk-type ends
  and one limit end.
\end{theorem}

\begin{figure}\begin{center}
\epsfysize=7cm \epsffile{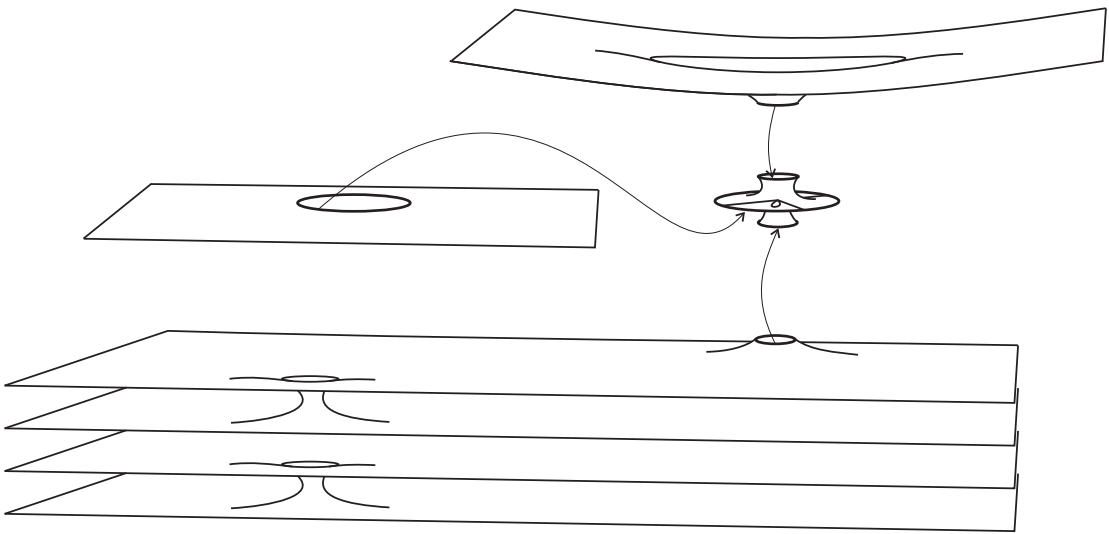}
\end{center}\caption{A sketch of a surface in the family of
  theorem~\ref{th3}.}\label{sketch}
\end{figure}

The family of KMR examples is a three parameter family which contains
two subfamilies whose surfaces have a vertical plane of symmetry. In
the construction of examples satisfying theorems~\ref{th2}
and~\ref{th3}, we need to have at least one vertical plane of symmetry
in order to control the kernel of the Jacobi operator on each glued
piece.  F.  Morabito \cite{M1} has recently proved there is a bounded
Jacobi field which does not come from isometries of $\R^3$ on the CHM
bent surface. For this reason, we are not able to produce a
3-parameter family of KMR examples with higher genus in
theorem~\ref{th2}.\\

The paper is organized as follows:
In section~\ref{costa} we briefly describe the CHM examples $M_k$
and obtain, for each genus $k$, a 1-parameter family of surfaces
$M_k(\xi)$ by bending the catenoidal ends of $M_k=M_k(0)$ keeping a
vertical plane of symmetry.  This is used to prescribe the flux of the
deformed CHM surface, which has to be the same as the corresponding
KMR example we want to glue (theorem~\ref{th2}).  To simplify the
construction of examples satisfying theorems~\ref{th1} and~\ref{th3},
we consider a ``no bent'' CHM example $M_k$.  In
section~\ref{family.costa.xi} we perturb $M_k(\xi)$ using the implicit
function theorem. We get an infinite dimensional family of minimal
surfaces that have three boundaries.\\

In section~\ref{strip}, we apply an implicit function theorem to solve
the Dirichlet problem for the minimal graph equation on a horizontal
flat periodic annuli with a disk $B$ removed, prescribing the boundary
data on $\partial B$ and the asymptotic direction of the Scherk-type
ends. We construct the flat annuli with a disk removed we will glue to
the CHM example at its middle planar end.  Varying the asymptotic
direction of the ends and the flux of the surface, we obtain the
pieces of Scherk example we will glue at the top and bottom catenoidal
ends of the CHM example (proving theorem~\ref{th1}) and to half a KMR
example (theorem~\ref{th3}).\\

In section~\ref{sec:KMR}, we study the KMR examples $M_{\s,\a,\b}$ and
describe a conformal parametrization of these examples on a cylinder.
We also obtain an expansion of pieces of the KMR examples as the flux
of $M_{\s,\a,\b}$ becomes horizontal (i.e. near the catenoidal limit).
Section~\ref{Jacobi.THL} is devoted to the study of the mapping
properties of the Jacobi operator about such $M_{\s,\a,\b}$ near the
catenoidal limit. And we apply in section~\ref{family.M.abt} the
implicit function theorem to perturb half a KMR example $M_{\s,\a,0}$
(resp. $M_{\s,0,\b}$), obtaining a family of minimal surfaces
asymptotic to half a $M_{\s,\a,0}$ (resp. $M_{\s,0,\b}$) and whose
boundary is a Jordan curve.  We prescribe the boundary data of such a
surface. We remark that sections \ref{sec:KMR}, \ref{Jacobi.THL}, \ref{family.M.abt} are of independent interest. They are devoted to the  
global analysis on KMR examples. \\

Finally, we do the end-to-end construction in section~\ref{gluing}: we
explain how the boundary data of the corresponding minimal surfaces
constructed in sections~\ref{strip},~\ref{family.costa.xi}
and~\ref{family.M.abt} can be chosen so that the union of these forms
smooth minimal surfaces satisfying theorems~\ref{th1},~\ref{th2}
and~\ref{th3}.

\input{costa.tex}

\input{scherk.tex}

\input{KMR.tex}

\input{KMRPeriode.tex}

\input{KMR-JACOBI.tex}

\input{KMR-FAMILY.tex}

\input{gluing}

\input{appendixA.tex}

\input{APPENDIXB.tex}

\input{APPENDIXC.tex}

\end{document}

%% file: costa.tex
\section{A Costa-Hoffman-Meeks type surface  with bent\\ catenoidal ends}
\label{costa}
In this section we recall the result shown in \cite{HP} about the
existence of a family of minimal surfaces $M_k(\xi)$ close to the
Costa-Hoffman-Meeks surface $M_k(0)=M_k$ of genus $k \geqslant 1$, with one planar
end and two slightly bent catenoidal ends by an angle $\xi$.

\subsection{Costa-Hoffman-Meeks surfaces}
We briefly present here the family of CHM surfaces $M_k$ studied
in~\cite{C1,C2,HM0,HM1,HM2}.  For each natural $k\geqslant 1$, $M_k$
is a properly embedded minimal surface of genus $k$ and three ends.
After suitable rotations and translations, we can assume its ends
are horizontal (in particular, they can be ordered by heights) and it
enjoys the following properties:
\begin{enumerate}
\item $M_k$ has one middle planar end $E_m$ asymptotic to the plane
  $\{x_3=0\}$, and two catenoidal ends: one top $E_t$ and one bottom
  $E_b$, respectively asymptotic to the upper and lower end of a
  catenoid having as axis of revolution the $x_3$-axis.

\item $M_k$ intersects the $\{x_3 =0\}$ plane in $k+1$ straight lines,
  which intersect at equal angles $\frac{\pi}{k+1}$ at the origin.
  The intersection of $M_k$ with any one of the remaining horizontal
  planes is a single Jordan curve. Thus the intersection of $M_k$ with
  the upper half-space $\{x_3 > 0\}$ (resp. the lower half-space
  $\{x_3 < 0\}$) is topologically an open annulus.

\item The symmetry group of $M_k$ is generated by $\pi$-rotations
  about the $k+1$ lines contained in the surface at height zero,
  together with the reflection symmetries in vertical planes that
  bisect those lines. Assume one of such planes of symmetry is the
  $\{x_2 =0\}$ plane. In particular, $M_k$ is invariant by the
  rotation of angle $\frac{2\pi}{k+1}$ about the $x_3$-axis and by the
  composition of a rotation by angle $\frac{\pi}{k+1}$ about the
  $x_3$-axis with the reflection symmetry across the $\{x_3 =0\}$
  plane.
  
\end{enumerate}
Now we give  a local description of the surfaces $M_{k}$ near its ends
and we introduce coordinates that we will use.\\

\noindent {\bf The planar end.} The planar end $E_m$ of $M_k$ can be
parametrized~\cite{HP} by
\[
X_m (x) = \left( \frac{x}{|x|^2}, u_m (x) \right) \in {\mathbb R}^3,
\quad x\in\bar B^*_{\rho_0}(0)
\]
where $\bar B^*_{\rho_0}(0)$ is the punctured closed disk in $\R^2$ of
radius $\rho_0>0$ small centered at the origin, and $u_m= {\cal
  O}_{C^{2,\a}_b}(|x|^{k+1})$ is a solution of
\begin{equation}
  |x|^4\, \mbox{div}\, \left(\frac{\nabla u}{(1+|x|^4\, |\nabla u|^2)^{1/2}}\right)=0.
\label{eq.min.piano.base2}
\end{equation}
Moreover, $u_m$ can be extended continuously to the puncture, using
Weierstrass representation (in fact, it can be extended as a
$C^{2,\a}$ function).  Here ${\cal O}_{C^{n,\a}_b}(g)$ denotes a
function that, together with its partial derivatives of order less
than or equal to $n+\a$, is bounded by a constant times $g$.

If we linearize in $u=0$ the nonlinear equation
\eqref{eq.min.piano.base2}, we obtain the expression of an operator
which is the Jacobi operator about the plane; i.e. ${\cal L}_{\R^2}=
|x|^4 \Delta_0$. To be more precise, the linearization of
(\ref{eq.min.piano.base2}) gives
\begin{equation}
\label{linearizzato.planar}
L_u\, v= |x|^4 {\rm div}\,
 \left(  \frac{\nabla v}{\sqrt{1+|x|^4|\nabla u|^2}}-|x|^4 \nabla u 
\frac{\nabla u \cdot \nabla  v}{\sqrt{(1+|x|^4|\nabla u|^2)^3}}\right).
\end{equation}
Equation~(\ref{eq.min.piano.base2}) means that the surface $\Sigma_u$
parametrized by $x\mapsto\left(\frac{x}{|x|^2},u(x)\right)$ is
minimal.  We will give the expression of the mean curvature $H_{u+v}$
of $\Sigma_{u+v}$ in terms of the mean curvature $H_u$ of $\Sigma_u$.

\begin{lemma}\label{lem:Hu+v.planar.end}
  There exists a function $Q_u$ satisfying $Q_u(0,0)= \nabla
  Q'_u(0,0)=0$ such that
\begin{equation}
\label{Hu+v.planar.end}
2H_{u+v}=2H_u+L_uv+ |x|^4 Q_u(\sqrt{|x|^4}\nabla v, \sqrt{|x|^4}
\nabla^2 v) .
\end{equation}
\end{lemma}
\begin{proof}
Define $f(t)=\frac{1}{\sqrt{1+|x|^4|\nabla(u+tv)|^2}}$ and 
apply Taylor expansion.
\end{proof}

Since $u$ satisfies (\ref{eq.min.piano.base2}), then $H_u=0$. 
The minimal surface equation that
we will use in the following sections is:
\begin{equation}
\label{operator.middle.end}
 {|x|^4} \left({\Delta_0}_{} v+ 
 \sqrt{1+|x|^4|\nabla u|^2} \left(\bar L_u v +  
Q_u({|x|^2}\nabla v, {|x|^2} \nabla^2 v)\right)\right)=0,
\end{equation}
where $\bar L_u $ is a second order linear operator whose coefficients
are in ${\mathcal O}_{C^{2,\a}}\left(|x|^{k+1} \right).$\\

\noindent
{\bf The catenoidal ends}. We denote by $X_c$ the parametrization 
of the standard catenoid $C$ 
whose axis of revolution is the $x_3$-axis. Its expression is 
$$ X_c (s, \theta) : = ( \cosh s \, \cos \theta, \cosh s \, \sin \theta, s)
 \in {\mathbb R}^3 $$
where $(s,\t) \in \R \times S^1$. 
The unit normal vector field about $C$ is
$$ n_c (s, \theta) : = \frac{1}{\cosh s} \, ( \cos \theta, 
 \sin \theta,  - \sinh s),\quad (s,\t) \in \R \times S^1. $$
Up to a dilation, we can assume that the two ends $E_t$ and $E_b$ of  $M_k$ are
 asymptotic to some translated copy of the catenoid parametrized by $X_c$ in the 
vertical direction. Therefore, $E_t$ and $E_b$ can be parametrized, 
respectively, by
$$ X_t : =  X_c  + w_t   \, n_c  + \sigma_t \, e_3\quad\mbox{in }
(s_0, \infty) \times S^1, $$
$$X_b : =  X_c  - w_b   \, n_c  - \sigma_b \, e_3\quad \mbox{in } (-
\infty, - s_0) \times S^1, $$ where $\sigma_t,\sigma_b \in {\mathbb
  R},$ and $w_t$ (resp. $w_{b}$) is a function defined in $(s_0,
\infty) \times S^1$ (resp. $(-\infty,-s_0) \times S^1$) which
tends exponentially fast to $0$ as $s$ goes to $+\infty$ (resp.
$-\infty$), reflecting the fact that the ends are asymptotic to a
catenoidal end.\\

We recall that the surface parametrized by $X : =  X_{c}  + w \, n_c$
is minimal if and only if the function $w$ satisfies the minimal
surface equation which, for normal graphs over a catenoid has the following form 
\begin{equation}
\label{eq.catenoide}
{\bL}_C w+
\frac{1}{  \cosh^2 s}\,\left(
 Q_{2}\left(\frac{w}{\cosh s}\right) + 
\cosh s \,Q_{3}\left(\frac{w}{\cosh s} \right) \right) =0,
\end{equation}
where ${\bL}_C$ is the Jacobi operator about the catenoid, i.e.
$${\bL}_C w=\frac{1}{\cosh^2 s}\left(
  \partial_{ss}^2 w +\partial_{\t\t}^2 w + \frac{2w}{\cosh^2s}
\right),$$ and $Q_{2}$, $Q_{3}$ are linear second order differential
operators which are bounded in ${\mathcal C}^k (\R \times S^1)$,
for every $k$, and  satisfy $Q_2(0)=Q_3(0)= 0$, $
\nabla Q_{2} (0)= \nabla Q_{3} (0) =0$,$\nabla^2 Q_{3} (0)=0$ and
then:
\begin{equation}
\| Q_{j} (v_2) - Q_{j} (v_1) \|_{{\mathcal C}^{0, \alpha} ([s,s+1]
\times S^1)} \leqslant c \, \left( \sup_{i=1,2} \| v_i \|_{{\mathcal
C}^{2, \alpha} ([s,s+1] \times S^1)} \right)^{j-1}  \,  \| v_2 -v_1
\|_{{\mathcal C}^{2, \alpha} ([s,s+1] \times S^1)}
\label{stima.Q.catenoide}
\end{equation}
for all $s \in {\mathbb R}$ and all $v_1, v_2$ such that $\| v_i
\|_{{\mathcal C}^{2, \alpha} ([s,s+1] \times S^1)} \leqslant 1$. The
 constant $c>0$ does not depend on $s$. \\

\subsection{ The family of Costa-Hoffman-Meeks surfaces with bent catenoidal ends.}
Denote by $R_\xi$ the rotation of angle $\xi$ about the $x_2$-axis
oriented by $e_2$.  Using an elaborate version of the implicit
function theorem and following \cite{J2} and \cite{KMP} it is possible
to prove the following

\begin{theorem}[\cite{HP}]
\label{existence.Mkx}
There exists $\xi _0 >0$ and a smooth 1-parameter family of minimal
surfaces $\{M_k (\xi)\ |\ \xi\in (- \xi_0, \xi_0)\}$ such that
$M_k(0)=M_k$ and each $M_k(\xi)$ is invariant by the reflection
symmetry across the $\{x_2=0\}$ plane, has one horizontal planar end
$E_m$ and has two catenoidal ends $E_t(\xi),E_b(\xi)$ asymptotic
respectively, up to a translation, to the upper and lower end of the
catenoid $R_\xi C$ (i.e. the standard catenoid whose axis of
revolution is directed by $R_\xi e_3$).  Moreover, $E_t(\xi),
E_b(\xi)$ can be parametrized respectively by
\begin{equation}
\label{parameterization.topend}
X_{t,\xi} = R_\xi  \, ( X_c + w_{t, \xi} \, n_c) + \sigma_{t ,\xi}
\, e_3  + \varsigma_{t,\xi} \, e_1 
\end{equation}
\begin{equation}
\label{parameterization.bottomend}
X_{b,\xi} = R_\xi  \, ( X_c - w_{b, \xi} \, n_c) - \sigma_{b ,\xi}
\, e_3 - \varsigma_{b,\xi}\, e_1
\end{equation}
where the functions $w_{t, \xi},w_{b, \xi}$ and the numbers $\sigma_{t
  ,\xi}, \varsigma_{t ,\xi} ,\sigma_{b ,\xi}, \varsigma_{b, \xi}\in
{\mathbb R}$ depend smoothly on $\xi$ and satisfy
$$|\sigma_{t,\xi}-\sigma_t| + |\sigma_{b,\xi}-\sigma_b| 
+ |\varsigma_{t,\xi}|+ |\varsigma_{b,\xi}|+
\|w_{t,\xi}-w_t\|_{C^{2,\a}_{-2}([s_0,+\infty)\times S^1)}+
\|w_{b,\xi}-w_b\|_{C^{2,\a}_{-2}((-\infty,-s_0]\times
  S^1)}\leqslant c|\xi|, $$
where
\[
\|w\|_{{\cal C}^{\ell,\a}_\delta([s_0,+\infty)\times S^1)}
=\sup_{s \geqslant s_0}\left( e^{-\delta s}\,\|w\|_{{\cal
      C}^{\ell,\a}([s,s+1] \times S^1)} \right) ,
\]
\[
\|w\|_{{\cal C}^{\ell,\a}_\delta((-\infty,-s_0])\times S^1)}
=\sup_{s \leqslant -s_0}\left( e^{-\delta s}\,\|w\|_{{\cal
      C}^{\ell,\a}([s-1,s] \times  S^1)} \right).
\]
\end{theorem}

\noindent
For all $s>s_0$ and $\rho <\rho_0$, we define
\begin{equation}
\label{definitionMk2}
M_k(\xi,s,\rho):= M_k(\xi)- \left(X_{t,\xi}([s,\infty)\times S^1)
  \cup X_m(B_\rho(0)) \cup X_{b,\xi}((-\infty,-s]\times S^1)\right).
\end{equation}
The parametrizations of the three ends of $M_k(\xi)$ induce a
decomposition of $M_k(\xi)$ into slightly overlapping components: a
compact piece $M_k(\xi,s_0+1, \rho_0/2)$ and three noncompact pieces
$X_{t,\xi} ((s_0 ,\infty) \times  S^1)$, $X_{b,\xi} ((-\infty, -s_0)
\times  S^1)$ and $X_m (\bar B_{\rho_0} (0))$.\\

We define the weighted space of functions on $M_{k}(\xi)$.

\begin{definition}
\label{definition.space.xi}
Given $\ell \in {\mathbb N}$, $\alpha \in (0,1)$ and $\delta \in
{\mathbb R}$, we define ${\mathcal C}^{\ell , \alpha}_{\delta }
(M_k(\xi))$ as the space of functions in ${\mathcal C}^{\ell ,
  \alpha}_{loc} (M_k(\xi))$ for which the following norm is finite
\[
\| w \|_{{\mathcal C}^{\ell , \alpha}_{\delta} (M_k(\xi))}  : = 
\| w \|_{{\mathcal C}^{\ell , \alpha} (M_k (\xi,s_0+1, \rho_0/2))} + \|
\, w \circ X_m \|_{{\mathcal C}^{\ell , \alpha}  (B_{\rho_0}(0))}
\]
\[
+ \|w\circ X_{t,\xi}\|_{{\cal C}^{\ell,\a}_\delta ([s_0,+\infty) \times
   S^1)}
+\|w\circ X_{b,\xi}\|_{{\cal C}^{\ell,\a}_\delta ((-\infty,-s_0] \times  S^1)}
\]
and which are invariant by the reflection symmetry across the
$\{x_2=0\}$ plane, i.e. $w(p) = w(\bar p)$ for all $p=(p_1,p_2,p_3)\in
M_k(\xi),$ where $\bar p := (p_1,-p_2, p_3)$.
\end{definition}

We remark that there is no weight on the planar end $E_{m}$ of
$M_{k}(\xi)$. In fact, we can compactify this end and consider a
weighted space of functions defined on a two-ended surface. In the
next section we will consider normal perturbations of $M_{k}(\xi)$ by
functions $u\in{\cal C}^{2,\a}_{\delta}$ and the planar end $E_{m}$
will be just vertically translated.\\

\noindent
{\bf The Jacobi operator.} The Jacobi operator about $M_k (\xi)$ is
$${\mathbb L}_{M_k(\xi)}:=\Delta_{M_k (\xi)}+ |A_{M_k (\xi)}|^2$$
where $|A_{M_k(\xi)}|$ is the norm of the second fundamental form on
$M_k(\xi)$. \\

In the parametrization of the ends of $M_k (\xi)$ introduced above,
the volume form $dvol_{M_k(\xi)}$ can be written as $\g_t\,ds\,d\t$ (
resp. $\g_b\,ds\,d\t,\ \gamma_m \,dx_1 \, dx_2$) near $E_t(\xi)$
(resp.  $E_b(\xi),\ E_m$). We define globally on $M_k(\xi)$ a smooth
function
\[
\gamma : M_k  (\xi) \longrightarrow [0, \infty)
\]
that equals $1$ on $M_k (\xi, s_0-1, 2 \rho_0)$ and equals $\gamma_t$
(resp. $\gamma_b$, $\gamma_m$) on the end $E_{t}(\xi)$ (resp.
$E_{b}(\xi)$, $E_m$). Observe that
\[
(\gamma \circ X_{t,\xi}) (s, \theta) \sim \cosh^2 s \quad \mbox{on }
X_{t,\xi}((s_0, \infty)\times  S^1) ,
\]
\[
(\gamma \circ X_{b,\xi}) (s, \theta) \sim \cosh^2 s \quad \mbox{on }
X_{b,\xi} ((-\infty,-s_{0})\times  S^1) ,
\]
\[
(\gamma \circ X_m)(x) \sim |x|^{-4} \quad \mbox{on } B_{\rho_0} .
\]

Granted the above defined spaces, one can check that:
\[
\begin{array}{rccccllll}
{\mathcal L}_{\xi,\delta} : &  {\mathcal C}^{2, \alpha}_{\delta } (M_k (\xi)) &
\longrightarrow & {\mathcal C}^{0, \alpha}_{\delta }
(M_k(\xi))\\[3mm]
& w & \longmapsto & \gamma  \, {\mathbb L}_{M_k (\xi)} \, (w)
\end{array}
\]
is a bounded linear operator. The subscript $\delta$ is meant to
keep track of the weighted space over which the Jacobi operator is
acting. Observe that, the function $\gamma$ is here to
counterbalance the effect of the conformal factor
$\frac{1}{\sqrt{|g_{M_k(\xi)}|}}$ in the expression of the Laplacian in
the coordinates we use to parametrize the ends of the surface
$M_k (\xi)$. This is precisely what is needed to have the operator defined
from the space ${\mathcal C}^{2, \alpha}_{\delta } (M_k(\xi))$ into the
target space ${\mathcal C}^{0, \alpha}_{\delta } (M_k(\xi))$.\\

To have a better grasp of what is going on, let us linearize the
nonlinear equation (\ref{eq.catenoide}) at $w = 0.$ We get the expression
of the Jacobi operator about the standard catenoid
\[
{\mathbb L}_C : = \frac{_1}{^{\cosh^2 s}} \, \left( \partial_s^2+
\partial_\theta^2 + \frac{_2}{^{\cosh^2 s}}\right).
\]
We can observe that the operator $\cosh^{2} s \, {\mathbb L}_C$ maps the
space $(\cosh s)^\delta \, {\mathcal C}^{2, \alpha} ((s_0,
+\infty)\times  S^1)$ into the space $(\cosh s)^\delta \, {\mathcal
C}^{0, \alpha} ((s_0, +\infty)\times  S^1)$.\\

Similarly, if we linearize the nonlinear equation \eqref{eq.min.piano.base2} at
$u = 0$, we obtain (see \eqref{linearizzato.planar} with $u=0$)
the expression of the Jacobi operator about the plane
\[
{\mathbb L}_{{\mathbb R}^2} : = |x|^4 \, \Delta_0.
\]
Again, the operator $|x|^{-4} \,{\mathbb L}_{{\mathbb R}^2}= \Delta_0$
clearly maps the space ${\mathcal C}^{2, \alpha}(\bar B_{\rho_0})$ into
the space ${\mathcal C}^{0, \alpha}(\bar B_{\rho_0})$. Now, the
function $\gamma$ plays, for the ends of the surface $M_k(\xi)$, the role
played by  the function $\cosh^2 s$ for the ends of the standard
catenoid and the role played by the function $|x|^{-4}$ for the
plane. Since the Jacobi operator about $M_k(\xi)$ is asymptotic to
${\mathbb L}_{{\mathbb R}^2}$ at $E_m$ and is asymptotic to
${\mathbb L}_C$ at $E_{t}(\xi)$ and $E_b(\xi)$, 
we conclude that the operator
${\mathcal L}_{\xi,\delta}$ maps ${\mathcal C}^{2, \alpha}_{\delta } (M_k(\xi))$ into
${\mathcal C}^{0, \alpha}_{\delta } (M_k(\xi))$.\\

We recall the notion of non degeneracy  introduced in \cite{HP}:

\begin{definition}
\label{nondegeneracy.xi}
The surface $M_k(\xi)$ is said to be non degenerate if $\cL_{\xi , \delta}$ 
is injective for all $\delta < -1.$
\end{definition}

It useful to observe that a duality argument in the  weighted Lebesgue spaces, implies 
$$ \left(\cL_{\xi,\delta} \quad \mbox{is injective} \right)  \quad
\Leftrightarrow \quad \left( \cL_{\xi,-\delta} \quad \mbox{is
    surjective}\right)$$
provided $\delta \notin {\mathbb Z}$.  See \cite{J2,Me} for more
details.\\

The non degeneracy of $M_k(\xi)$ follows from the study of the kernel
of $\cL_{\xi,\delta}$.\\

\noindent {\bf The Jacobi fields.} It is known that a smooth
1-parameter group of isometries containing the identity generates a
Jacobi field, that is a solution of the equation ${\mathbb
  L}_{M_k(\xi)}u=0$. The solutions which are invariant under the
action of the reflection symmetry across the $\{x_2=0\}$ plane, are
generated by dilations, vertical translations and horizontal
translations along the $x_1$-axis (we refer~\cite{HP} for details):

\begin{itemize}
\item The group of vertical translations generated by the Killing
  vector field $\Xi(p) = e_3$ gives rise to the Jacobi field $
  \Phi^{0,+}(p) : = n (p) \cdot e_3$.

\item The vector field $\Xi(p) =p$ associated to the 1-parameter group
  of dilation, generates the Jacobi fields $\Phi^{0,-} (p) := n(p)
  \cdot p$.
  
\item The Killing vector field $\Xi(p) = e_1$ that generates the group
  of translations along the $x_1$-axis is associated to a Jacobi field
  $\Phi^{1,+} (p) : = n(p) \cdot e_1$.
  
\item Finally, we denote by $\Phi^{1,-} (p) := n(p) \cdot (e_2 \times
  p)$ the Jacobi field associated to the Killing vector field $\Xi (p)
  = e_2\times p$ that generates the group of rotations about the
  $x_2$-axis.
\end{itemize}
There are other Jacobi fields we do not take into account because they
are not invariant by the reflection symmetry across
the $\{x_2=0\}$ plane.\\

With these notations, we define the deficiency space
\[
{\mathcal D} : =  \mbox{Span} \{ \chi_t \, \Phi^{j,\pm }, \chi_b \,
\Phi^{j, \pm} \,  : \, j=0, 1\}
\]
where $\chi_t$ is a cutoff function that is identically equal to $1$
on $X_{t,\xi} ((s_0+1,\infty)\times  S^1)$, identically equal to $0$
on $M_k(\xi) -X_{t,\xi} ((s_0,\infty)\times  S^1)$ and that is
invariant under the action of the symmetry with respect to the $\{x_2
=0\}$ plane; and
\[
\chi_b (\cdot) : = \chi_t (- \, \cdot).
\]
Clearly,
\[
\begin{array}{rcccllll}
\tilde {\mathcal L}_{\xi,\delta} :  & {\mathcal C}^{2, \alpha}_{\delta}
(M_k(\xi)) \oplus {\mathcal D} & \longrightarrow & {\mathcal C}^{0,
\alpha}_{\delta} (M_k(\xi)) \\[3mm]
& w &  \longmapsto & \gamma \, {\mathbb L}_{M_k (\xi)} \, (w)
\end{array}
\]
is a bounded linear operator, for $\delta <0$.\\

A result of S. Nayatani~\cite{N1, N2} extended by the second
author~\cite{M1} 
states that any bounded Jacobi field invariant by the reflection
symmetry across the $\{x_2=0\}$ plane, is linear combination of
$\Phi^{0,+}$ and $\Phi ^{1,+}$.  This fact together with an adaptation
to our setting of the linear decomposition lemma proved in~\cite{KMP}
for constant mean curvature surfaces (see also \cite{J2} for minimal
hypersurfaces), allows us to prove the following result.

\begin{proposition}
We fix $\delta \in (-2,-1).$
Then (reducing $\xi_0$ if this is necessary) the operator 
$\tilde {\mathcal L}_{\xi, \delta},$ for $|\xi| < \xi_0,$ 
is surjective and has a kernel of dimension $4$.
\end{proposition}

\noindent
From that we get the following one about the operator 
${\mathcal L}_{\xi , \delta}$

\begin{proposition}
We fix $\delta \in (1,2).$
Then (reducing $\xi_0$ if this is necessary) the operator 
${\mathcal L}_{\xi, \delta}$ is surjective and 
there exists $G_{\xi , \delta }$ 
a right inverse for ${\mathcal L}_{\xi , \delta}$ that depends
 smoothly on $\xi$ and in particular whose norm is bounded 
 uniformly as $|\xi| < \xi_0$.
\label{inverse.operator.costa}
\end{proposition}

\section{Infinite dimensional family of minimal surfaces close to $M_k
  (\xi)$}
\label{family.costa.xi}

In this section we consider a truncature of $M_{k}(\xi)$.  First we
write portions of the ends of $M_k(\xi)$ as vertical graphs over the
$\{x_3=0\}$ plane.

\begin{lemma}[\cite{HP}]
\label{annular.part.xi}
There exists $\e_0>0$ such that, for all $\e \in (0, \e_0)$ and all
$|\xi |\leqslant \e$, an annular part of the ends $E_t(\xi)$,
$E_b(\xi)$ and $E_m$ of $M_k (\xi)$ can be written, respectively, as
vertical graphs over the annulus $B_{2r_\e}-B_{r_\e/2}$ for the
functions
$$\begin{array}{l}
U _{t} (r,\theta )=\sigma_{t, \xi} +\ln (2 r) - \xi \, r \, \cos \theta +
{\mathcal O}_{{\mathcal C}^\infty_b} (\e),\\
U_{b} (r , \theta )  =  - \sigma_{b, \xi} - \ln (2r)  -
\xi \, r \, \cos \theta +  {\mathcal O}_{{\mathcal
C}^\infty_b} (\e) ,\\
 U_{m}  (r , \theta )  =   {\mathcal O}_{{\mathcal
C}^\infty_b} (r^{-(k+1)}).
\end{array}$$

Here $(r, \theta)$ are the polar coordinates in the $\{x_3=0\}$ plane.
The functions ${\mathcal O}(\e)$ are defined in the annulus $B_{2
  \,r_\e} - B_{r_\e / 2}$ and are bounded in the ${\mathcal
  C}^\infty_b$ topology by a constant (independent on $\e$) multiplied
by $\e$, where the partial derivatives are computed with respect to
the vector fields $r \, \partial_r$ and $\partial_\theta$.
\end{lemma}

In particular, a portion of the two catenoidal ends $E_{t}(\e/2)$ and
$E_{b}(\e/2)$ of $M_{k}(\e/2)$ are graphs over the annulus $B_{2
  \,r_\e} - B_{r_\e / 2}\subset\{x_{3}=0\}$ for functions $U _{t}$ and
$U_{b}$. We set
\[
s_\e=-\frac{1}{2}\ln \e,\quad \rho_\e=2 \e^{1/2} \quad \hbox{and}
\]
\[
M_k ^T(\e/2)=M_k(\e/2)-\left(X_{t,\e/2 }((s_\e, +\infty)
  \times  S^1)\cup X_{b,\e/2}((-\infty, - s_\e) \times  S^1)\cup
  X_{m}(B_{\rho_\e}(0))\right).
\]

We prove, following section 6 in~\cite{HP}, the existence of a family
of surfaces close to $M_k^T(\xi)$.  In a first step, we modify the
parametrization of the ends $E_t(\e/2), E_b(\e/2), E_m$, for
appropriates values of $s$, so that, when $r \in [3r_\e/4,3r_\e/2]$
the curves corresponding to the image of
\[
\theta \rightarrow (r\, \cos \theta, r\, \sin \theta, U_{t  } (r, \theta)),
\quad  \theta \rightarrow (r\, \cos \theta, r\, \sin \theta, U_{b  }
(r, \theta) ),
\]
\[
\theta \rightarrow (r\, \cos \theta, r\, \sin \theta, U_{m } (r,
\theta))
\]
correspond respectively to the curves $\{s = \ln (2 r)\},\{s = -\ln
(2 r)\}, \{\rho=\frac1r\}$.

The second step is the modification of the unit normal vector field on
$M_k(\e/2)$ to produce a transverse unit vector field $\tilde n_{\e/2}
$ in such a way that it coincides with the normal vector field $n
_{\e/2}$ on $M_k(\e/2) $, is equal to $e_3$ on the graph over $B_{3
  r_\e/2} - B_{3r_\e/4}$ of the functions $U_{t }$ and $U_{b }$ and
interpolate smoothly between the different definitions of $\tilde
n_{\e/2}$ in different subsets of $M_k^T (\e/2)$.

Finally we observe that close to $E_t(\e/2),$ we can give the following
estimate: 
\begin{equation}
\label{differenza1}
\left|{\cosh^2 s} \left( {\mathbb L}_{M_k(\e/2)}v-\cosh^{-2} s
\left(  \partial_{ss}^2 v +\partial_{\t\t}^2 v
 \right)  \right)\right|\leqslant c\left|\cosh^{-2}s\, v\right|.
 \end{equation}
 This follows easily from (\ref{eq.catenoide}) together with the fact
 that $w_{t,\xi}$ (see \eqref{parameterization.topend}) decays at
 least like $\cosh^{-2}s$ on $E_t(\e/2).$ Similar considerations hold
 close the bottom end $E_b(\e/2).$ Near the middle planar end $E_m,$
 we observe that the following estimate holds:
\begin{equation}
\label{differenza2}
\left||x|^{-4} \left( 
{\mathbb L}_{M_k(\e/2)}v-|x|^4 \Delta_0 v \right)\right|\leqslant c
\left||x|^{2k+3} \nabla v\right|.
\end{equation}
This follows easily from (\ref{linearizzato.planar}) together with the
fact that $u_m$ decays at least like $|x|^{k+1}$ on $E_m.$\\

The graph of a function $u$, using the vector field $\tilde n_{\e/2}$,
will be a minimal surface if and only if $u$ is a solution of a second
order nonlinear elliptic equation of the form
$$
{\mathbb L}_{M_k^T(\e/2) } \, u = \tilde L_{\e/2} \, u + Q_{ \e} \,(u)
$$
where ${\mathbb L}_{M_k^T(\e/2) }$ is the Jacobi operator about
$M_k^T(\e/2) $, $Q_{ \e}$ is a nonlinear second order differential
operator and $\tilde L_{\e/2}$ is a linear operator which takes into
account the change of the normal vector field (only for the top and
bottom ends) and the change of the parametrization.  This operator is
supported in a neighborhoods of $\{\pm s_\e\} \times  S^1$, where
its coefficients are uniformly bounded by a constant times $\e^2$, and
a neighborhood of $\{\rho_\e\} \times  S^1$ where its coefficients
are uniformly bounded by a constant times $\e^3$.\\

Now, we consider three even functions $\varphi_t, \varphi_b,\varphi_m
\in {\mathcal C}^{2, \alpha} ( S^1)$ such that $\varphi_t,
\varphi_b$ are $L^2$-orthogonal to $1$ and $\t\to\cos \t$ while
$\varphi_m$ is $L^2$-orthogonal to $1$. Assume that they satisfy
\begin{equation}
\label{stima.3.Phi.xi}
\|\varphi_t\|_{{\mathcal C}^{2, \alpha}} + \|\varphi_b\|_{{\mathcal
    C}^{2, \alpha}}+\|\varphi_m\|_{{\mathcal
    C}^{2, \alpha}} \leqslant \kappa \, \e.
\end{equation}
We set $\Phi : = (\varphi_t, \varphi_b, \varphi_m)$ and we define
$w_\Phi$ to be the function equal to:

\begin{enumerate}  
\item $\chi_+(s) \, H_{\varphi_t} (s_\e-s, \cdot)$ on the image of
  $X_{t,\e/2}$, where $\chi_+$ is a cut-off function equal to $0$ for
  $s \leqslant s_0+1$ and identically equal to $1$ for $s \in [s_0+2,
  s_\e]$,
\item $\chi_-(s) \, H _{\varphi_b} (s-s_\e, \cdot)$ on the image of
  $X_{b,\e/2}$, where $\chi_-$ is a cut-off function which equals $0$
  for $s \geqslant -s_0-1$
  and $1$ for $s \in [-s_\e, -s_0-2]$,

\item $\chi_m(r) \, \tilde H _{\rho_{\e},\varphi_m} ( \cdot , \cdot)$
  on the image of $X_{m},$ where $\chi_m$ is a cut-off function equal
  to $0$ for $r \geqslant \rho_{0} $ and to $1$ for $\rho \in
  [\rho_\e, \rho_0/2]$,

\item 0 in the remaining part of the surface $M_k^T (\e/2)$,
\end{enumerate}
where $\tilde H$ and $H$ are, respectively, harmonic extensions of the
operators introduced in Propositions~\ref{poisson.piano.interno.xi} and
\ref{poisson.catenoide.troncone.xi} of Appendix A.\\

We would like to prove that, under appropriates hypothesis, the graph
over $M_k^T (\e/2)$ of the function $u=w_\Phi +v$ is a minimal
surface. This is equivalent to solve the equation:
$$
{\mathbb L}_{M_k ^T(\e/2)} (w_\Phi + v)  = \tilde L_{\e/2} (w_\Phi + v) +
Q_{ \e} (w_\Phi + v ).
$$
The resolution of the previous equation is obtained thanks to 
the following fixed point problem:

\begin{equation}
\label{eq.fixed.point.costa.xi}
v= T(\Phi,v)
\end{equation}  
with
$$
T(\Phi,v) = {G}_{\e/2,  \delta} \circ {\mathcal E}_\e \,  \left( \gamma \left( \tilde
L_{\e/2} (w_\Phi  + v) - {\mathbb L}_{M_k^T (\e/2)} \, w_\Phi +
Q_{\e} (w_\Phi + v ) \right)\right) ,
$$
where $\delta \in (1,2), $  the operator ${G}_{\e/2,\delta}$ is 
the right inverse provided in Proposition  \ref{inverse.operator.costa} and ${\mathcal E}_\e$ 
is a linear extension operator 
$$
{\mathcal E}_\e : {\mathcal C}^{0, \alpha}_{\delta }(M_k^T (\e/2))
\longrightarrow {\mathcal C}^{0, \alpha}_{\delta } (M_k(\e/2) ) ,$$
(here ${\mathcal C}^{0, \alpha}_{\delta }( M_k^T (\e/2))$ denotes the
space of functions of ${\mathcal C}^{0, \alpha}_{\delta } (M_k
(\e/2))$ restricted to $M_k^T (\e/2)$) such that ${\mathcal E}_\e v=v$
in $M_k^T (\e/2),$ ${\mathcal E}_\e v=0$ in the image of
$[s_\e+1,+\infty) \times  S^1$ by $X_{t,\e/2},$ in the image of
$(-\infty,-s_\e-1) \times  S^1$ by $X_{b,\e/2}$ and in the image of
$B_{\rho_\e/2} \times  S^1$ by $X_{m}$, and ${\mathcal E}_\e v$ is
an interpolation of these values in the remaining part of
$M_{k}(\e/2)$:
$$({\mathcal E}_\e v)\circ X_{t,\e/2}(s,\t)=(1+s_\e-s)
(v\circ X_{t,\e/2}(s_\e,\t)),\quad \hbox{for}\quad (s,\t) \in
[s_\e,s_\e+1] \times  S^1$$
$$({\mathcal E}_\e v)\circ X_{b,\e/2}(s,\t)=(1+s_\e+s)(v\circ X_{b,\e/2}(s_\e,\t))
\quad \hbox{for} \quad (s,\t) \in [-s_\e-1,-s_\e] \times  S^1$$
$$({\mathcal E}_\e v)\circ X_{m}(\rho,\t)=
\left(\frac 2 {\rho_\e} \rho-1\right) (v\circ X_{m}(\rho_\e,\t))\quad
\hbox{for} \quad (\rho,\t) \in [\rho_\e/2, \rho_\e] \times  S^1.$$
\begin{remark}
As consequence of the properties  of ${\mathcal E}_\e,$
if $supp\, v \cap \left(B_{\rho_\e}-B_{{\rho_\e}/2 }\right)
\neq \emptyset$  then 
$$\| ({\mathcal E}_\e v) \circ X_m \|
_{{\mathcal C}^{0, \alpha}(\bar B_{\rho_0})} \leqslant c
\rho_\e^{-\a}\|v\circ X_m\| _{{\mathcal C}^{0,
    \alpha}(B_{\rho_0}-B_{\rho_\e})}.$$ This phenomenon of explosion of
the norm does not occur near the catenoidal type ends:
$$\|({\mathcal E}_\e v)\circ X_{t,\e/2}\|_{{\mathcal C}^{0, \alpha}([s_0,+\infty) \times  S^1)} \leqslant c  \| v\circ X_{t,\e/2}\|_{{\mathcal C}^{0, \alpha}([s_0,s_\e] \times  S^1)}.$$
A similar equation holds for the bottom end.\\

In the following we will assume $\a >0$ and close to zero.
\end{remark}

The existence of a solution
$v \in {\mathcal C}^{2, \alpha}_{\delta } (M_k^T (\e/2))$ for the
equation (\ref{eq.fixed.point.costa.xi}) is a consequence of the 
following result which proves that $T$ is a contracting mapping.
\begin{proposition}
\label{contrazione.costa.xi} There exist constants $c_\kappa  >0$ 
and $\e_\kappa >0$, such that
\begin{equation}
\|T(\Phi,0) \|_{{\mathcal C}^{2,\a}_{\delta}}(M_k^T (\e/2)) \leqslant c_\kappa \, \e^{3/ 2}
\label{estimate.T.xi}
\end{equation}
and, for all $\e \in (0, \e_\kappa)$ and $0<\a<\frac12$, 
\[
\|T (\Phi,v_2)-T(\Phi,v_1)\|_{{\mathcal C}^{2,\a}_{\delta }(M_k (\e/2))} 
\leqslant \frac{1}{2}
\, \| v_2-v_1\|_{{\mathcal C}^{2,\a}_{\delta }(M_k(\e/2))},
\]
\begin{equation}
\label{estimate.2fi}
\|T (\Phi_2,v)-T(\Phi_1,v)\|_{{\mathcal C}^{2,\a}_{\delta }(M_k (\e/2))} 
\leqslant c \e
\, \| \Phi_2-\Phi_1\|_{{\mathcal C}^{2,\a}(S^1)},
\end{equation}
where 
\[
\| \Phi_2-\Phi_1\|_{{\mathcal C}^{2,\a}(S^1)}=
\| \varphi_{t,2}-\varphi_{t,1}\|_{{\mathcal C}^{2,\a}(S^1)}
\]
\[
+\| \varphi_{b,2}-\varphi_{b,1}\|_{{\mathcal C}^{2,\a}(S^1)}+
\| \varphi_{m,2}-\varphi_{m,1}\|_{{\mathcal C}^{2,\a}(S^1)}
\]
for all $v,v_1,v_2 \in {\mathcal C}^{2,\a}_{\delta } (M_k^T (\e/2))$ and
satisfying $\|v\|_{{\mathcal C}^{2,\a}_{\delta }} \leqslant \, 2 \,
c_\kappa \, \e^{3/ 2}$ and for all boundary data 
$\Phi,\Phi_1,\Phi_2 \in [{\mathcal C}^{2, \alpha} ( S^1)]^3$ satisfying 
\eqref{stima.3.Phi.xi}.
\end{proposition}

\begin{proof}
We recall that the Jacobi operator associated to $M_k(\e/2),$
is asymptotic to the operator of the catenoid near the 
catenoidal ends, and it is asymptotic to the  laplacian  
near of the planar end. 
The  function $w_\Phi$ is identically zero far from the ends
where the explicit expression of ${\mathbb L}_{M_k(\e/2)}$ is not known: 
this is the  reason of our particular choice in the definition of $w_\Phi.$
Then from the definition of $w_\Phi,$ thanks to Proposition 
\ref{inverse.operator.costa},  to \eqref{differenza1} and
 \eqref{differenza2}, we obtain the estimate
$$
\|  {\mathcal E}_\e \left( \gamma {\mathbb L}_{M_k^T (\e/2)} \, w_\Phi \right) 
\|_{{\mathcal C}^{0,\a}_{\delta }(M_k (\e/2))}
 = $$
$$ \leqslant c 
\left|\left| \cosh^{-2} s \, (w_\Phi \circ X_{t,\e/2}) \right|\right|
_{{\mathcal C}^{0,\a}_{\delta }([s_0+1,s_\e] \times  S^1)}+
c \left|\left| \cosh^{-2} s \,( w_\Phi \circ X_{b,\e/2}) \right|\right|
_{{\mathcal C}^{0,\a}_{\delta }([-s_\e,-s_0-1] \times  S^1)}+
$$
$$c \e^{-\frac \a 2}
 \left|\left| \rho^{2k+3}\nabla(w_\Phi \circ X_{m})\right| \right|
_{{\mathcal C}^{0,\a}([\rho_\e,\rho_0] \times  S^1 )} 
\leqslant c_\kappa \e^{3/2}. $$


\noindent
Using the properties of $\tilde L_{\e/2}$, we obtain 
\[
\|  {\mathcal E}_\e \left( \gamma \tilde L_{\e/2} \, w_\Phi \right) 
\|_{{\mathcal C}^{0,\a}_{\delta }(M_k (\e/2))}
\leqslant c \e^{} \|w_\Phi \circ X_{t,\e/2}  \|_{{\mathcal
    C}^{0,\a}_{\delta}([s_0+1,s_\e] \times  S^1)}
\]
\[
\mbox{}\hspace{1cm}+c \e^{} \|w_\Phi \circ X_{b,\e/2}  \|
_{{\mathcal C}^{0,\a}_{\delta}([-s_\e,-s_0-1] \times  S^1)} 
+ c \e^{1-\a/2} \|w_\Phi \circ X_{m}  \|_{{\mathcal C}^{0,\a}
([\rho_\e,\rho_0] \times  S^1) } 
\leqslant   c_\kappa \e^{3/2}.
\]
As for the last term, we recall that the operator $Q_\e$ has two different
expressions if we consider the catenoidal type end and the planar end 
(see \eqref{eq.catenoide} and \eqref{operator.middle.end}).
It holds that 
$$\|  {\mathcal E}_\e \left( \gamma Q_\e \left( w _\Phi \right) \right)  \|_{{\mathcal C}^{0,\a}_{\delta }(M_k (\e/2))} \leqslant  c_k\e^{3/2}.$$
\noindent
Details are left to the reader.
 \end{proof}

\begin{theorem}
\label{fixedpoint.CHM}
Let be $B:=\{w \in {\cal C}^{2,\a}_{\delta}(M_{k,\e})\,
 | \, ||w||_{2,\a} \leqslant 2 c_\kappa {\e^{3/2}}\}.$
Then the nonlinear mapping $T$ defined above has a unique fixed point $v$ in $B.$
\end{theorem}

\noindent{\bf Proof.} The previous lemma shows that, if $\e$ is  chosen small 
enough, the nonlinear mapping  $T$
is a contraction mapping from the ball $B$ of radius 
$ 2 c_\kappa \e^{3/2}$ in 
${\mathcal C}^{2,\a}_{\delta}(M_{k}^T(\e/2))$ into itself. This value follows 
from the estimate
of the norm of $T(\Phi,0).$ Consequently thanks to Sch\"auder fixed point
theorem, $T$ has a unique fixed point $w$ in this ball.
\hfill \qed \\


\noindent
This argument provides a minimal surface $M_k^T (\e/2,
\Phi)$ which is close to $M_k^T (\e/2)$ and 
has three boundaries. This surface is, close to its
upper and lower boundary, a vertical graph over the annulus $B_{r_\e} -
B_{r_\e/2}$ whose parametrization is, respectively, given by
$$
 U_{t} (r, \theta ) = \sigma_{t,\e/2} + \ln (2r) +
\frac \e 2 r \cos \t +H_{\varphi_t} (s_\e-\ln 2r, \theta) + V_{t} (r, \theta ) ,
$$
$$
 U_b (r, \theta ) = - \sigma_{{b,\e/2} } - \ln (2r) +  
\frac \e 2 r \cos \t+H_{\varphi_b} (s_\e-\ln 2r, \theta) + V_{b} (r, \theta ) ,
$$
where $  s_\e  = - \frac{1}{2}\, \ln \e .$
The boundaries of the surface correspond to $r_\e = \frac{1}{2} \, \e^{-1/2}.$
Nearby the middle boundary the surface  is a vertical graph over the annulus 
$B_{r_\e } -B_{r_\e /2},$  Its parametrization is 
$$
 U_m (r, \theta ) = \tilde H_{\rho_\e,\varphi_m} (1/r ,\theta ) +
 V_{m} (r, \theta ) ,
$$
where $\rho_\e=2 \e^{1/2}.$
All the functions $V_i$ for $i=t,b,m$ depend non linearly on $\e,\varphi.$

\begin{lemma}
The functions $V_i(\e,\varphi_i),$ for $i=t,b,$ satisfy $ 
\|V_i(\e,\varphi_i)(r_\e \cdot,\cdot) \|
_{{\mathcal C}^{2,\a}  (\bar B_2- B_{1/2})}  
\leqslant c \e$ and 
\begin{equation}
\label{stima.contrazione.costa.xi.1}
 \|V_i(\e,\varphi_{i,2})(r_\e \cdot,\cdot)-V_i(\e,\varphi_{i,1})(r_\e \cdot,\cdot) \|_{{\mathcal C}^{2,\a} 
(\bar B_2- B_{1/2})} 
\leqslant c \e^{1-\delta/2} \| \varphi_{i,2} - \varphi_{i,1} \|
_{{\mathcal C}^{2,\a}}
\end{equation}
\noindent
The function $V_m(\e,\varphi)$ satisfies 
$ \|V_m(\e,\varphi)(\rho_\e \cdot,\cdot) \|
_{{\mathcal C}^{2,\a} (\bar B_2- B_{1/2})}  
\leqslant c \e$ and 
\begin{equation}
\label{stima.contrazione.costa.xi.2}
\|V_m(\e,\varphi_{m,2})(\rho_\e \cdot,\cdot)-V_m(\e,\varphi_{m,1})(\rho_\e \cdot,\cdot) \|_{{\mathcal C}^{2,\a}
(\bar B_2- B_{1/2})} 
\leqslant c \e^{} \| \varphi_{m,2} - \varphi_{m,1} \|_{{\mathcal C}^{2,\a}} 
\end{equation}
\end{lemma}
\begin{proof}
The wanted estimates follow from 
$$\|  V_i(\e,\varphi_2)(\cdot,\cdot)  - 
V_i (\e,\varphi_1)( \cdot,\cdot) \|_{{\mathcal C}^{2,\alpha}
( \bar B_{2r_\e} - B_{r_\e/2})}\leqslant
c e^{\delta s_\e} \| T(\Phi_2,V_i) - T(\Phi_1,V_i) 
\|_{{\mathcal C}^{2,\a}_{\delta}(E_i(\e/2))},
$$
\noindent
for $i=t,b$, together with
$$\|  V_m(\e,\varphi_2)(\cdot,\cdot)  - 
V_m(\e,\varphi_1)( \cdot,\cdot) \|_{{\mathcal C}^{2,\alpha}
( \bar B_{2\rho_\e} - B_{\rho_\e/2})}\leqslant
c \| T(\Phi_2,V_m) - T(\Phi_1,V_m) 
\|_{{\mathcal C}^{2,\a}(E_m)}
$$
and the estimate \eqref{estimate.2fi} of Proposition 
\ref{contrazione.costa.xi}.
\end{proof}

%% file: scherk.tex
\section{An infinite family of Scherk type minimal surfaces close 
to a horizontal periodic flat annuli}
\label{strip}
This section has two purposes. The first one is to find an infinite
family of minimal surfaces close to a horizontal periodic flat annuli
$\Sigma$ with a disk $D_s$ removed.  The surfaces of this family have
two horizontal Scherk-type ends $E_1,E_2$ and will be glued on the
middle planar end of a CHM surface. We will prescribe the boundary
data $\varphi$ on $\partial D_s$.  Assume the period $T$ of $\Sigma$
points to the $x_2$-direction. Then $E_1,E_2$ have the $x_{1}$-axis as
asymptotic direction.

The second and more general purpose of this section is to show the
existence of an infinite family of minimal graphs over $\Sigma-D_s$
whose ends have slightly modified asymptotic directions.  When the
asymptotic directions are not horizontal, these surfaces are close to
half a Scherk surface, seen as a graph over $\Sigma-D_s$ (see
Figure~\ref{sketch}, above). A piece of one such surface will be glued
to the catenoidal ends of a CHM example $M_k$ and to an end of a KMR
example $M_{\s,0,0}$. We will prescribe the boundary data on $\partial
D_s$. Since we need to prescribe the flux along $\partial D_{s}$, we
will modify the asymptotic direction of the ends and we will choose
$|T|$ large.

\subsection{Scherk type ends}
We parametrize conformally the annulus $\Sigma\subset\R^3/T$ on
$\C^*$, with the notation $(x_1,x_2,x_3)=(x_1+i x_2,x_3)$, by the
mapping
 \[
 A(w)=\left(-\frac{|T|}{2\pi}\ln (w),0\right),\quad w\in\C^*. 
 \]
 The horizontal Scherk-type end $E_1$ described above can be written
 as the graph of a function $h_1\in {\cal C}^{2,\a}(B^*_{r}(0))$,
 where $B^*_{r}(0)$ is the punctured disk $B_{r}(0)-\{0\}$
 of radius $r\in(0,1)$ centered at the origin. The function $h_1(w)$
 is bounded and extends to the puncture (see \cite{HT}). The end $E_1$
 can be parameterized by
 \[
 X_1(w)=A(w)+h_1(w) e_{3}
 =\left(-\frac{|T|}{2\pi}\ln
  (w),h_1(w)\right) \in \R^3/T,\quad w\in B^*_{r}(0), 
\]
in the orthonormal frame ${\cal F}=(e_{1},e_{2},e_{3})$.
The asymptotic direction of the end is $e_{1}$.

The horizontal Scherk-type end $E_{2}$ can be similarly
parameterized in $\C- B_{r^{-1}}(0)$. Via an inversion, we can
parametrize $E_2$  by
\[
X_{2}(w)
=\left(-\frac{|T|}{2\pi}\ln (w),h_2(w)\right) \in \R^3/T,\quad w\in
B^*_{r}(0),
\]
in the frame ${\cal F}^-=(-e_{1},-e_{2},e_{3})$, where $h_2\in {\cal
  C}^{2,\a}(B^*_{r}(0))$ is a bounded function which can be extended
to the puncture.  Now the asymptotic direction of the end is $-e_{1}$.

Let us now parametrize a general Scherk-type end, non necessarily
horizontal.
Let $R_{\theta}$ denote a rotation in $\R^3/T$ by angle $\t$ about the
$x_{2}$-axis (oriented by $e_{2}$). We can paremetrize a non
necessarily horizontal Scherk-type end $\widetilde E_1$ with
asymptotic direction $\cos\t_1\, e_1+ \sin\t_1\, e_3$ and limit  normal
vector $R_{\t_1}(e_3)$, with $\t_1\in[0,\pi/2)$, by
\[
\widetilde X_1(z)=
\left(-\frac{|T|}{2\pi}\ln (z),\widetilde
  h_1(z)\right), z \in B^*_{r}(0)
\]
in the frame ${\cal F}(\theta_1)=R_{\theta_1}{\cal F}$, where
$\widetilde h_1\in {\cal C}^{2,\a}(B^*_{r}(0))$ is a bounded function
which can be extended to the origin.

Finally, a Scherk-type end $\widetilde E_{2}$ with asymptotic
direction $-\cos\t_2\, e_1+ \sin\t_2\, e_3$ and limit normal vector
$R_{-\t_2}(e_3)$, with $\t_2\in[0,\pi/2)$, can be parameterized by
\[
\widetilde X_2(z)=
\left(-\frac{|T|}{2\pi}\ln (z),\widetilde
  h_2(z)\right), z \in B^*_{r}(0)
\]
in the frame ${\cal F}^-(\theta_2)=R_{-\theta_2}{\cal F}^-$, where
$\widetilde h_2\in {\cal C}^{2,\a}(B^*_{r}(0))$ is a bounded function
which can be extended to the origin.

\subsection{Construction of the infinite families}

Given $r\in(0,1)$ and $\Theta=(\theta_{1},\theta_{2}) \in [0,\theta
_{0}]^2$, with $\theta_0>0$ small, we denote by $A_\Theta:\C^*\to\R^3/T$
the immersion obtained as smooth interpolation of
\[
A_\Theta(z)=\left\{\begin{array}{lll}
    (R_{\theta_{1}}\circ A)(z) &  \mbox{if}\ |z|<r/2, \\
    A(z)  & \mbox{if}\ r<|z|<r^{-1},\\
    (R_{-\theta_{2}}\circ A)(z) &  \mbox{if}\ |z|>2r^{-1}.
  \end{array}\right.
\]
Let $N_{\Theta}$ be the vector field obtained as smooth interpolation
of $R_{\theta_1}(e_3)$ on $\{ |z|<r/2\}$, $e_{3}$ on $\{ r<|z|<r^{-1}\}$
and $R_{-\theta_2}(e_3)$ on $\{|z|>2r^{-1}\}$.  For any $h \in
C^{2,\a}(\bar \C)$, we define the immersion
\[
X_{\Theta,h}(z)=A_{\Theta}(z)+h(z)N_{\Theta}(z),\quad z\in\C^*.
\]
The immersion $X_{\Theta,h}$ has two Scherk-type ends $E_1,E_2$ with
asymptotic directions $\cos\t_1\, e_1+ \sin\t_1\, e_3$ and
$-\cos\t_2\, e_1+ \sin\t_2\, e_3$, respectively.\\

At the end $E_1$ (resp. $E_2$), $X_{\Theta,h}(z)=A(z)+h_1(z)e_3$ in the
orthogonal frame ${\cal F}(\t_1)$ ( resp. 
$X_{\Theta,h}(z)=A(z^{-1})+h_2(z)e_3$ in the frame ${\cal F}^-(\t_2)$), with
$z\in B^*_{r}(0)$, where $h_1(z)=h(z)$ and $h_2(z)=h(z^{-1})$.
In~\cite{HT} has been proven that the mean curvature of $X_{\Theta,h}$
at $E_i$ is given in terms of the $z$-coordinate by
\[
2H=\frac{4 \pi^2 |z|^2}{|T|^2}{\rm div}_0 ( P^{-1/2} \nabla_0 h_i),
\]
where $P=1+\frac{4 \pi^2 |z|^2}{|T|^2}\|\nabla_0 h_i\|_0^2$ and the
subindex $\cdot_0$ means that the corresponding object is computed
with respect to the flat metric of the $z$-plane.  We denote by
$\lambda$ the smooth function without zeroes defined by $\lambda
(z)=\frac{|T|^2}{4 \pi^2 |z|^2}$, for $z\in B^*_{r}(0)$.  Then at
$E_{i}$ we have
\[
2\lambda H=P^{-1/2} \Delta_0 h_i- \frac{1}{2}P^{-3/2} \langle \nabla_0
P ,\nabla_0 h_i \rangle_0.
\]
So the mean curvature at the end $E_i$ vanishes if $h_i$ satisfies the
equation
\begin{equation}
\label{traizet}
\Delta_0 h-\frac{1}{2 P} \langle \nabla_0 P ,\nabla_0 h \rangle_0=0.
\end{equation}
 


\begin{definition} Given $k\in \N$ and $\alpha \in (0,1)$ we define
  $C^{k,\a}(\bar \C)$ as the space of functions $u\in
  C^{k,\a}_{loc}(\bar \C)$ such that
$$||u||_{C^{k,\a}(\bar \C)}:=[u]_{k,\a,\bar \C}<+\infty,$$
where $[u]_{k,\a,\bar \C}$ denotes the usual ${C^{k,\a}}$ H\"older
norm on $\bar \C.$
\end{definition}

Let $B_{s}$ be the disk of radius $s$ in $\C^*$ such that
$D_s=A(B_s)\subset\Sigma=\{z \in \C\,|\, -|T|<2y\leqslant |T|\}$
 is a geodesic disk centered at the origin of
$\R^3/T$.  Denote by $C^{k,\a}(\bar \C - B_{s})$ the space of
functions in $C^{k,\a}(\bar \C)$ restricted to $\bar \C - B_{s}.$
We denote
by $H(\Theta,h)$ the mean curvature of $X_{\Theta,h}$, and $\bar
H(\Theta,h)=\l H(\Theta,h)$, where $ \lambda$ is the smooth function
defined  in a neighbourhood
at each puncture by $\lambda(z)=\frac{|T|^2}{4 \pi^2 |z|^2}$.
 Lemma 4.1 of \cite{HT} shows that
\[
\bar H :\R^2\times {\cal C}^{2,\alpha}(\bar \C -B_{s}) \longrightarrow
{\cal C}^{0,\alpha}(\bar \C -B_{s})
\]
is an analytical operator.
Denote by ${\cal L}_\Theta$ the Jacobi operator about $A_{\Theta}$.
We set $\bar {\cal L}_{\Theta}=\lambda {\cal L}_{\Theta}$. 
{\begin{remark}
\label{metrica.C}
The operators $H$ and ${\cal L}_{\Theta}$ are the mean curvature
operator and the Jacobi operator with respect to 
the metric $|dz|^2$ of  $\bar \C.$ Defining  
the operators $\bar H=\lambda H$ and $\bar  {\cal L}_{\Theta}=
\lambda {\cal L}_{\Theta}$
means to consider a different metric on $\bar \C.$ 
Actually $\bar H$ and $\bar  {\cal L}_{\Theta}$ 
 are the mean curvature operator and Jacobi operator
with respect to the metric $g_{\lambda}=|dz|^2/\lambda.$ From the definition
of $\lambda,$ it follows that the volume of $\bar \C$ with 
respect this metric is finite.
\end{remark}

$\bar {\cal L}_{\Theta}$ 
is a second order linear elliptic operator satisfying 
$|\bar {\cal L}_{\Theta} u-\Delta u| \leqslant c (|\theta_{1}|+|\theta_{2}|)|u|$
and the  coefficients of  $F_{\Theta}=\Delta -\bar {\cal L}_{\Theta}$ have compact support.\\


Now we fix $s_0>0$. Given $\ve>0$ and $|T|\in [4/\sqrt{\e},+\infty)$
large enough, we  choose $s\in (0,s_0)$ such that $D_s=A(B_s)$ is
the geodesic disk of radius $1/2\sqrt \ve$ centered at the origin.

\begin{proposition}
\label{mapdelta}
There exists $\e_{0}>0$ and $\eta_{0}>0$ such that for every $\e \in
(0,\e_{0})$ and every $|T| \in (\eta_{0},+\infty)$, there exists an
operator
\[
G_{\e,|T|} : {\cal C }^{0,\alpha} (\bar \C-B_{s}) \longrightarrow
{\cal C}^{2,\alpha}(\bar \C -B_{s})
\]
such that, given $f \in {\cal C }^{0,\alpha} (\bar \C-B_{s})$, $w=
G_{\e,|T|}f$ satisfies
\[
\left\{\begin{array}{lll}
 \Delta w= f & {\rm on} &\bar \C-B_{s}\\
 w \in {\rm Span}\{1\} & {\rm on} & \partial B_{s}
\end{array}\right.
\]
and $||w||_{{\cal C}^{2,\alpha}} \leq c ||f||_{{\cal C}^{0,\alpha} }$
for some constant $c>0$ which does not depend on $\e,|T|$.
\end{proposition}
\begin{proof}
  Let be $u$ a solution of $\Delta u = f$ on $\bar \C -B_{s}$ with 
$u=0$ on $\partial B_{s}$.
We set $w= u-\int_{\bar \C-B_{s}}u \,dvol_{g_\lambda}$. 
We  recall that the metric in use on $\bar \C$
is given by $g_\lambda=|dz|^2/\lambda.$ 
With respect to this metric $vol(\bar \C-B_{s}) < +\infty$ 
and $\int_{\bar \C-B_{s}}u \,dvol_{g_\lambda}< \infty .$ 
So the function $w$ is well defined and 
then $\int_{\bar\C-B_{s}} w\,dvol_{g_\lambda}=0$ and 
$w \in {\rm Span}\{1\}$ on $\partial B_s$. 
If the theorem is false, there is a sequence of functions $f_{n}$, of 
solutions $w_{n}$ and of real numbers $s_{n}$ such that

$$\sup_{\bar \C -B_{s_{n}}}|f_{n}|=1, \quad 
A_{n}:= \sup _{\bar \C -B_{s_{n}}}|w_{n}| \rightarrow +\infty \hbox{ as }
 n \rightarrow +\infty,$$ where $s_{n} \in [0,s_{0}].$
Now we set $\tilde w_{n}:=w_{n}/A_{n}$. Elliptic estimates imply that 
$s_n$ and $\tilde w_{n}$ converge up to a subsequence, respectively, 
to $s_\infty \in [0,s_{0}]$ and to $\tilde w_{\infty}$ on $\bar \C - B_{s_{\infty}}.$ This function satisfies  
$$\Delta \tilde w_{\infty}=0.$$
Then $\tilde w_{\infty}={\rm const}$
on $\bar \C - B_{s_{\infty}}$  and 
$\int_{\bar \C - B_{s_{\infty}}} \tilde w_{\infty}\,dvol_{g_\lambda}=0$, a 
contradiction with $\sup \tilde |w_{\infty}|=1$.  
\end{proof}

\noindent
Now we fix $|T|\geqslant 4/\sqrt{\e}$, $\Theta \in (0,\e)^2$, 
$s_{\varepsilon}=\frac{1}{2\sqrt \e}$ and let 
$\varphi \in {\cal C}^{2,\alpha}(S^1)$ be even (or odd)
$L^2$-orthogonal to the function $z\rightarrow 1$,  with 
$\|\varphi\|_{{\cal C}^{2,\alpha}(S^1)} \leqslant \kappa \e$. 
Let $w_{\varphi}=\tilde H_{s_\e,\varphi}$ (see proposition
\ref{poisson.piano.interno.xi})
be the unique bounded harmonic extension of $\varphi.$ We would
like to solve   the minimal 
surface equation $H(\Theta, v+w_{\varphi})=0$ with
fixed boundary data $\varphi$,  prescribed asymptotic direction 
$\Theta$ and with period $|T|$. Then we have to solve the equation:
$$\Delta v= F_{\Theta} (v+w_{\varphi}) + Q_{\Theta}(v+ w_{\varphi}) $$

\noindent
with $Q_{\Theta}$ a quadratic term such that
 $|Q_{\Theta}(v_{1})-Q_{\Theta}(v_{2})|\leqslant c |v_{1}-v_{2}|^2$. The resolution
of the previous equation is obtained by showing  the  existence 
of a fixed point
$$v=S({\Theta},\varphi, v)$$
where
$$S({ \Theta},\varphi, v)=G_{\e,|T|}( F_{\Theta} (v+w_{\varphi}) + 
Q_{\Theta}(v+ w_{\varphi})). $$
 
\begin{proposition}
There exists $c_{\kappa}>0$ and $\e_{\kappa}>0$ such that for all 
$|T| \geqslant 4/\sqrt{\e}$ we have

$$||S({ \Theta},\varphi, 0)||_{{\cal C }^{2,\alpha} }
\leqslant c_{\kappa} \e ^{2}$$
and for all $\e \in (0,\e_{\kappa})$,

$$|| S({ \Theta},\varphi, v_{1}) -S({\Theta},\varphi, v_{2})||_
{{\cal C }^{2,\alpha}} \leqslant \frac12 ||v_{2}-v_{1}||_{{\cal C }^{2,\alpha} }$$
$$|| S({\Theta},\varphi_1, v) -S({ \Theta},\varphi_2, v)||_
{{\cal C }^{2,\alpha}} \leqslant c \e ||\varphi_{2}-\varphi_{1}||
_{{\cal C }^{2,\alpha} }$$
for all $v,v_{1},v_{2} \in {\cal C }^{2,\alpha}(\bar \C -B_{s})$ and satisfying 
$|| v ||_{{\cal C }^{2,\alpha}} \leqslant 2c_{\kappa} \e^2$, for all 
boundary data $\varphi, \varphi_1, \varphi_2  \perp 1$, whose norm is bounded 
by a constant $\kappa$ times $\e$
and for all  $\Theta= (\t_1,\t_2)$ such 
that $|\theta_{1}|+|\theta_{2}| \leqslant \e.$
\end{proposition}
\noindent 
{\bf Proof.} 
Using the proposition \ref{mapdelta}, the inequality 
$|\bar {\cal L}u-\Delta u| \leqslant c (|\t_1|+|\t_2|)|u|$
and the quadratic behavior of $Q_{\Theta}$ we derive the estimate 
of the proposition. The details of the proof are left to 
the reader. \hfill \qed \\

\noindent
On the set 
$B_{2s_{\e}}-B_{s_{\e}},$
the function $U=v+\tilde H_{s_\e,\varphi}$ is the solution 
of the  equation 
(\ref{traizet}).
Using the vertical translation 
$c_{0}e_{3}$ we can fix the value $c_0+\varphi$ at the boundary:
$$U=c_{0}+\tilde H_{s_\e,\varphi}+v.$$
\noindent
Using Sch\"auder estimate for the equation on a fixed bounded domain 
$$||v_{}(\varphi_{1})-v_{}(\varphi_{2})||_{{\cal C}^{2,\alpha}
(\bar \C -B_{s})} \leqslant
 c_{\kappa}\e || \varphi_{1}-\varphi_{2}||_{{\cal C}^{2,\alpha}(S^1)}.$$

\noindent
This can be done uniformly in $(\theta_{1},\theta_{2})$. 
We observe that the function $U$ grows logarithmically close 
$\partial B_{s_{\e}}.$ The hypothesis of orthogonality of $\varphi$ 
to the constant function, implies that the function 
$w_{\varphi}$ enjoys the same property and is bounded.
 It is not the case of $v$ which  can be seen as the sum  
of a bounded function which is orthogonal to the constant 
and a  function of the form
 $c \ln (r/s_\e),$ where $c=c(|T|,\t_1,\t_2),$ defined in a neighbourhood 
 of $\partial B_{s_{\e}}.$ 
We are able to determine $c$ using flux formula.
Let $\gamma_1,\gamma_2$ be two closed curves in $\bar \Sigma/T$
chosen in such a way to correspond by the conformal mapping to the 
boundaries of two circular neighbourhoods $N_1,N_2$ of the punctures 
corresponding to the ends with linear growth.
Let ${\cal S} = \bar \C-(B_{s_\e} \cup N_1 \cup N_2).$
Now we observe that, being $X$ the parametrization 
of a minimal surface,  then the following equality
holds:
$$ \int_{\cal S} \Delta X=0.$$ 
Thanks to the divergence theorem, if $\Gamma=\partial {\cal S},$ 
 then 
$$\int_{\cal S} \Delta X=\int_{\Gamma} \frac {\partial X}{\partial \eta}ds= 
\int_{\gamma_1} \frac {\partial X}{\partial \eta}ds+
\int_{\gamma_2} \frac {\partial X}{\partial \eta}ds+
\int_{\partial B_{s_\e}} \frac {\partial X}{\partial \eta}ds=0, $$
where $\eta$ denotes the conormal along $\Gamma.$
This equality implies  
$$\int_{\partial B_{s_{\e}}} \frac {\partial U}{\partial \eta}ds=
-\sin \t_1 |T| -\sin \t_2 |T|.$$
By integration we can conclude that on
$ B_{2s_\e}-B_{s_\e},$ it holds that  

$$U=-\frac{|T|}{2\pi}(\sin \t_1 +\sin \t_2) \ln (r/s_\e)+c_{0} +
w_{\varphi}+ v^{\perp}.$$
with $v^{\perp} \perp 1$. Now we choose $|T|$ such that 
$\frac{|T|}{2\pi}(\sin \t_1 +\sin \t_2)=1$.\\

\noindent 
We observe that, if $\t_2=\t_1=0,$ we can state that there exists
an infinite family of minimal surfaces which are close to the
surface $\Sigma-D_{s_\e}$. We denote by $S_m(\varphi)$ one of such surfaces.
 It can be seen as the graph about  $B_{2s_\e}-B_{s_\e}$
 of the function 
$$\bar U_m(r,\t)=c_{0}+\tilde H_{s_\e,\varphi}(r,\t)+\bar V_m$$ 
 where $V_m={\mathcal O}_{C^{2,\a}_b}(\e)$ and it satisfies
 \begin{equation}
 \label{stima.contrazione.strip}
||\bar V_{m}(\varphi_{1})-\bar V_{m}(\varphi_{2})||_{{\cal C}^{2,\alpha}
(B_{2s_\e}-B_{s_\e})} \leqslant 
 c_{\kappa}\e || \varphi_{1}-\varphi_{2}||_{{\cal C}^{2,\alpha}(S^1)},
 \end{equation}
 for $\varphi_2,\varphi_1 \in C^{2,\a}(S^1).$\\ 

\noindent
If $(\t_2,\t_1)\neq 0,$ we can state that there exists
an infinite family of minimal surfaces which are close to the
periodic Scherk type example.
After a vertical translation, they can be seen as the graph about  $B_{2s_\e}-B_{s_\e}$ of the 
function 
$$\bar U_t(r,\t)=-\ln(2r)+ c_{0}+\tilde H_{s_\e,\varphi}(r,\t)+\bar V_t$$ 
 where $\bar V_t={\mathcal O}_{C^{2,\a}_b}(\e)$ and it satisfies
 \begin{equation}
 \label{stima.contrazione.scherk}
||\bar V_{t}(\varphi_{1})-\bar V_{t}(\varphi_{2})||_{{\cal C}^{2,\alpha}
(B_{2s_\e}-B_{s_\e})} \leqslant 
 c_{\kappa}\e || \varphi_{1}-\varphi_{2}||_{{\cal C}^{2,\alpha}(S^1)},
 \end{equation}
 for $\varphi_2,\varphi_1 \in C^{2,\a}(S^1).$ \\
 
\begin{remark}
\label{simmetrie.scherk}
If the boundary data $\varphi$ is an even function, it is clear 
 the surfaces we have just described are symmetric with respect to the vertical 
plane $x_2=0.$ Instead, if the boundary data $\varphi$ is an 
odd function and $\t_1=\t_2$ the surfaces are symmetric with respect 
to the vertical plane $x_1=0.$
\end{remark}

%% file: KMR.tex
\section{KMR examples}
\label{sec:KMR}
In this section, we briefly present the {\it KMR examples}
$M_{\s,\a,\b}$ studied in~\cite{K1,K2,MR1,Ro} (also called {\it
  toroidal halfplane layers}), which are the only properly embedded
doubly periodic minimal surfaces with genus one and finitely many
parallel (Scherk-type) ends in the quotient (see~\cite{PRT}).

For each $\s\in(0,\frac{\pi}{2})$, $\a\in[0,\frac{\pi}{2}]$ and $\b\in
[0,\frac{\pi}{2}]$ with $(\a,\b)\neq(0,\s)$, consider the rectangular
torus $\Sigma_\s=\left\{(z,w)\in\overline{\C }^2\ |\
  w^2=(z^2+\lambda^2)(z^2+\lambda^{-2})\right\}$, where
$\lambda=\lambda(\s)=\cot\frac{\s}{2}>1$. The KMR example $M_{\s,\a,\b}$ is
determined by its Gauss map $g$ and the differential of its height
function $h$, which are defined on $\Sigma_\s$ and given by:
\[ 
g(z,w)=\frac{az+b}{i(\overline{a}-\overline{b}z)} , \qquad
dh=\mu\,\frac{dz}{w},
\]
with
\begin{equation}
\begin{array}{l}
  a= a(\a,\b)=\cos\frac{\a+\b}{2}+i\cos\frac{\a-\b}{2};\\
  \\
  b= b(\a,\b)=\sin\frac{\a-\b}{2}+i\sin\frac{\a+\b}{2};\\
  \\
  \mu=\mu(\s) = \frac{\pi\csc\s}{\mbox{\roc K}(\sin^2\s)} ,\\
  \end{array}
\end{equation}
where $\mbox{\roc
  K}(m)=\int_0^{\frac{\pi}{2}}\frac{1}{\sqrt{1-m\sin^2u}}\, du$ ,
$0<m<1$, is the complete elliptic integral of first kind.  Such $\mu$
has been chosen so that the vertical part of the flux of
$M_{\s,\a,\b}$ along any horizontal level section equals $2\pi$.

\begin{remark}\label{rem1} The following statements give us a better
  understanding of the geometrical meaning of $a,b$ defined above:
\begin{itemize}
\item[(i)] $b\to 0$ if and only if $\a\to 0$ and $\b\to 0$, in which
  case $a\to 1+i$.
\item[(ii)] $|b|^2+|a|^2=2$.
\item[(iii)] When $\b=0$, then $a=(1+i) \cos\frac{\a}{2}$ and $b=(1+i)
  \sin\frac{\a}{2}$. In particular, $b={\cal O}(\a)$.
\item[(iv)] When $\a=0$, then $a=(1+i) \cos\frac{\b}{2}$ and
  $b=(-1+i) \sin\frac{\b}{2}$. In particular, $b={\cal O}(\b)$.
\item[(v)] In general,
  $\left|\frac{b}{a}\right|=\tan\frac{\varphi}{2}$, where $\varphi$ is
  the angle between the North Pole $(0,0,1)\in\esf^2$ and the pole of
  $g$ seen in $\esf^2$ via the inverse of the stereographic
  projection.
\end{itemize}
\end{remark}

The KMR example $M_{\s,\a,\b}$ can be parametrized on
$\Sigma_\s$ by the immersion $X = (X_1,X_2,X_3)=\Re\int{\cal W}$, where
${\cal W}$ is the Weierstrass form:
\[
{\cal W}=\left(\frac{1}{2}\left(\frac{1}{g}-g\right)dh,\frac{i}{2}\left(\frac{1}{g}+g\right)dh,dh\right).
\]
The ends of $M_{\s,\a,\b}$ correspond to the punctures
$\{A,A',A'',A'''\} =g^{-1}(\{ 0,\infty \})$, and the branch values of
$g$ are those with $w=0$, i.e.
\begin{equation}\label{eqbranchings1}
  \textstyle{ D=(-i\lambda,0),\ D'=(i\lambda,0),\ D''=(\frac{i}{\lambda},0),\ D'''=(-\frac{i}{\lambda},0). }
\end{equation}
Seen in $\esf^2$, these points form two pairs of antipodal points:
$D''=-D$ and $D'''=-D'$ and each KMR example can be given in terms of
the branch values of its Gauss map (see~\cite{Ro}).\\


In \cite{Ro}, it is proven that the above Weierstrass data define a
properly embedded minimal surface $M_{\s,\a,\b}$ invariant by two
independent translations: the translation by the period $T$ at its
ends, and the translation by the period $\widetilde T$ along a
homology class. Moreover, the group of isometries Iso$(M_{\s,\a,\b})$
of $M_{\s,\a,\b}$ always contains a subgroup isomorphic to $(\Z
/2\Z)^2$, with generators ${\cal D}$ (corresponding to the deck
transformation $(z,w)\mapsto(z,-w)$), which represents in $\R^3$ a
central symmetry about any of the four branch points of the Gauss map
of $M_{\s,\a,\b}$; and ${\cal F}$, which consists of a translation by
$\frac{1}{2}(T+\widetilde T)$.  In particular, the ends of
$M_{\s,\a,\b}$ are equally spaced.\\

\noindent
We are going to focus on two more symmetric subfamilies of KMR
examples: $\{M_{\s,\a,0}\ |\ 0<\s<\frac{\pi}{2},\
0\leqslant\a\leqslant\frac{\pi}{2}\}$ and $\{M_{\s,0,\b}\ |\
0<\s<\frac{\pi}{2},\ 0\leqslant\b<\s\}$ (these two families were originally
studied by Karcher~\cite{K1,K2}).
\begin{enumerate}
\item When $\a=\b=0$, $M_{\s,0,0}$ contains four straight lines
  parallel to the $x_1$-axis, and Iso$(M_{\s,0,0})$ is isomorphic to
  $(\Z /2\Z)^4$ with generators $S_1,S_2,S_3,R_D$: $S_1$ is a
  reflection symmetry in a vertical plane orthogonal to the
  $x_1$-axis; $S_2$ is a reflection symmetry across a plane orthogonal
  to the $x_2$-axis; $S_3$ is a reflection symmetry in a horizontal
  plane (these three planes can be chosen meeting at a common point
  which is not contained in the surface); and $R_D$ is the
  $\pi$-rotation around one of the four straight lines contained in
  the surface, see Figure~\ref{KMR} left. In this case,
  $T=(0,\pi\mu,0)$.

\item When $0<\a<\frac{\pi}{2}$, Iso$(M_{\s,\a,0})$ is isomorphic to
  $(\Z /2\Z)^3$, with generators ${\cal D}$, $S_2$ and $R_2$, where
  $S_2$ represents a reflection symmetry across a plane orthogonal
  to the $x_2$-axis, and $R_2$ is a $\pi$-rotation around a line
  parallel to the $x_2$-axis that cuts $M_{\s,\a,0}$ orthogonally, see
  Figure~\ref{KMR} right. Now $T=(0,\pi\mu t_\a,0)$, with
  $t_\a=\frac{\sin\s}{\sqrt{\sin^2\s \cos^2\a+\sin^2\a}}$. 

  \begin{figure}
    \begin{center}
      \epsfysize=4cm \epsffile{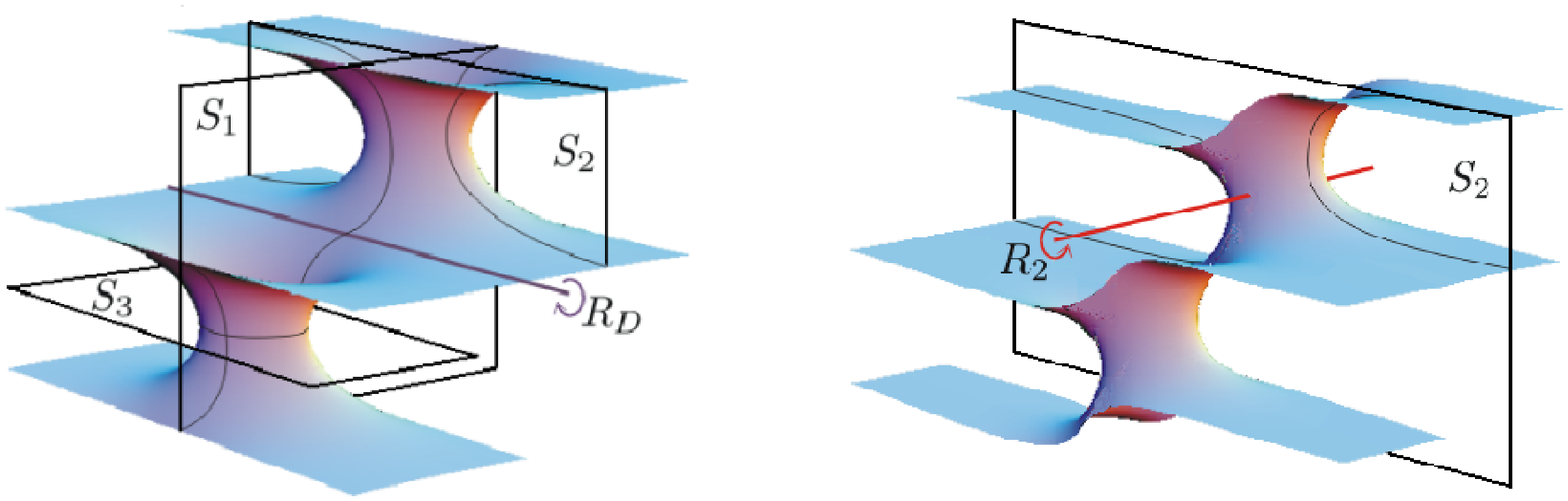}
    \end{center}\caption{Left: $M_{\s,0,0}$, with $\s=\frac{\pi}{4}$.
      Right: $M_{\s,\a,0}$ for $\s=\a=\frac{\pi}{4}$.}
    \label{KMR}
  \end{figure}

\item Suppose that $0<\b<\s$. Then $M_{\s,0,\b}$ contains four
  straight lines parallel to the $x_1$-axis, and Iso$(M_{\s,0,\b})$ is
  isomorphic to $(\Z /2\Z)^3$, with generators $S_1$, $R_1$ and $R_D$:
  $S_1$ represents a reflection symmetry across a plane orthogonal to
  the $x_1$-axis; $R_1$ corresponds to a $\pi$-rotation around a line
  parallel to the $x_1$-axis that cuts the surface orthogonally; and
  $R_D$ is the $\pi$-rotation around any one of the straight lines
  contained in the surface, see Figure~\ref{KMR1}. Moreover,
  $T=(0,\pi\mu t^\b,0)$, where
  $t^\b=\frac{\sin\s}{\sqrt{\sin^2\s-\sin^2\b}}$.
  \begin{figure}
    \begin{center}
      \epsfysize=4cm 
      \epsffile{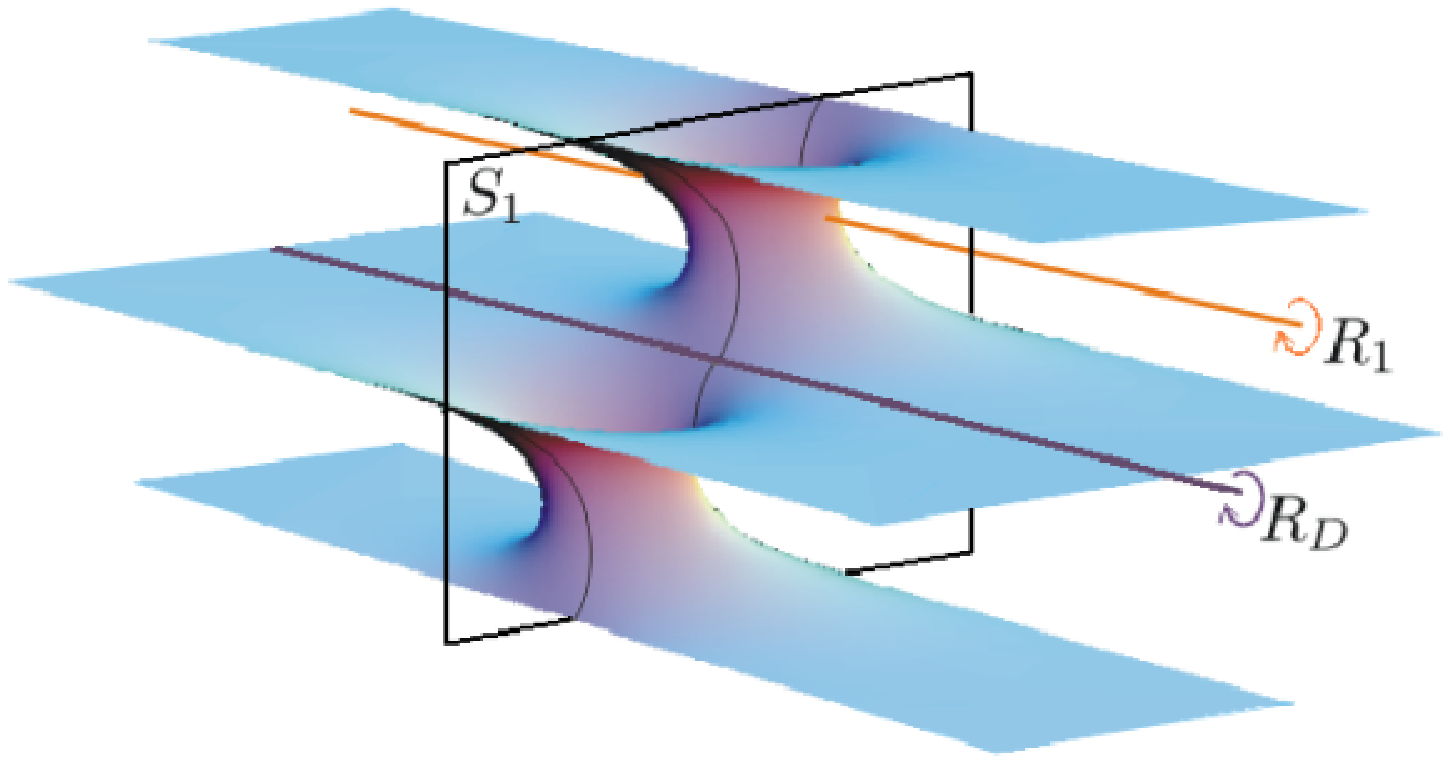}
    \end{center}
    \caption{$M_{\s,0,\b}$, where $\s=\frac{\pi}{4}$ and
      $\b=\frac{\pi}{8}$.}
\label{KMR1}
  \end{figure}

\end{enumerate}

Finally, it will be useful to see $\Sigma_\s$ as a branched
$2$-covering of $\overline\C$ through the map $(z,w)\mapsto z$. Thus
$\Sigma_\s$ can be seen as two copies $\overline{\C}_1,\overline\C_2$
of $\overline\C$ glued along two common cuts $\gamma_1,\gamma_2$,
which can be taken along the imaginary axis: $\gamma_1$ from $D$ to $D'$
and $\gamma_2$ from $D''$ to $D'''$.

\subsection{$M_{\s,\a,\b}$ as a graph over $\{x_3=0\}/T$}
\label{Msab.graph}
The KMR examples $M_{\s,\a,\b}$ converge as $(\s,\a,\b)\to(0,0,0)$ to
a vertical catenoid, since $\Sigma_\s$ converges to two pinched
spheres, $g(z)\to z$ and $dh\to\pm\frac{dz}{z}$ as
$(\s,\a,\b)\to(0,0,0)$. In fact, we can obtain two catenoids in the
limit, depending on the choice of branch for $w$ (for each copy of
$\overline\C$ in $\Sigma_\s$, we obtain one catenoid in the limit).
One of our aims along this paper consists of gluing KMR examples
$M_{\s,\a,0}$ or $M_{\s,0,\b}$ near this catenoidal limit, to a
convenient compact piece of a deformed CHM surface $M_k(\ve/2)$.  In
this subsection we express part of $M_{\s,\a,\b}$ as a vertical graph
over the $\{x_3=0\}$ plane when $\s,\a,\b$ are small.

Consider $M_{\s,\a,\b}$ near the catenoidal limit, i.e.  $\s,\a,\b$
close to zero. 
Without lost of generality, we can assume $dh\sim
\frac{dz}{z}$ in $\overline\C_1$. 
We are studying the surface in an annulus about one of its ends, say a
zero of its Gauss map.

\begin{lemma}
  \label{lemX} Consider $\a+\b+\s\leqslant\ve$ small. Up to
  translations, $M_{\s,\a,\b}$ can be parametrized in the annulus
  $\{(z,w)\in\Sigma_\s\ |\ z\in\overline\C_1,\
  |\frac{b}{a}|<|z|<\nu\}$ (for $\nu>|\frac{b}{a}|$ small) as:
  \[
  \left\{\begin{array}{l} X_1+i\, X_2=
      \frac{-1}{2}\left(z+\frac{\,1\,}{\overline{z}}\right)
      -\frac{(1+i)\overline{b}}{4\overline{z}^2}
      +{\mathcal O}(\ve z^{-1}+\ve^2 z^{-3})\\
      \\
      X_3= \ln|z|+{\mathcal O}(\ve^2 z^{-2}),
    \end{array}\right.
  \]
\end{lemma}
{\bf Proof.}
  Recall we have assumed $dh\sim\frac{dz}{z}$ in the annulus we are working
  on.  More precisely, we have
  \[\textstyle{ dh= \frac{\mu\, dz}{\sqrt{(z^2+\lambda^2)(z^2+\lambda^{-2})}}=
    \frac{\mu}{\lambda\, \sqrt{1+\lambda^{-2}z^2+\lambda^{-2}z^{-2}+\lambda^{-4}}}\, \frac{dz}{z} .
  }\]
  Since $\frac{\mu}{\lambda}= \frac{\pi}{(1+\cos(\s)) \mbox{\roc
      K}(\sin^2\s)}= 1+{\mathcal O}(\s^4)$,
  and $\lambda=\cot\frac{\s}{2}={\cal O}(\ve)$, we get
  \[
  dh=\frac{dz}{z}(1+{\cal O}(\ve^4))(1+{\cal O}(\ve^2 z^2+\ve^2
  z^{-2}+\ve^4)) .
  \]
  Since $|z|<\nu<1$, then 
  \[
  dh=\frac{dz}{z}(1+{\mathcal O}(\ve^2 z^{-2})) .
  \]
  Fix any point $z_0\in\C_1$,
  $z_0\not\in\left\{\frac{-b}{a},\frac{\overline a}{\overline
      b}\right\}$ (which correspond to two ends of the KMR example),
  and recall that $g=\frac{az+b}{i(\overline{a}-\overline{b}z)}$.
  Straightforward computations give us, for $|b/a|<|z|<1$,
  \begin{itemize}
  \item $\int_{z_0}^z\frac{dh}{g}
    =-\frac{i\,\overline{b}}{a}\,\ln z
    -\frac{2 i}{a^2 z}
    +\frac{2 i\,b}{a^3 z^2}
    +C_1+ {\mathcal O}(\ve^2 z^{-3})$,
  \item $\int_{z_0}^z g\, dh= -\frac{i\,b}{\overline{a}}\,\ln z-
    \frac{2i}{\overline{a}^2}\,z+ C_2+
    {\mathcal O}(\ve^2 z^{-1})$,
\end{itemize}
where $C_1,C_2\in\C$ verify $\overline{C_1}-C_2=$
$\frac{1}{2}\left(z_0+\frac{\, 1\, }{\overline{z_0}}\right)+{\mathcal
    O}(\ve)$. Taking into account that $a=(1+i)+{\cal O}(\ve)$, we
  obtain
\[
\begin{array}{ll}
  X_1+i\, X_2 &
  = \frac{1}{2}\left(\overline{\int_{z_0}^z\frac{dh}{g}}- \int_{z_0}^z
    g\,dh\right)\vspace{3mm}\\

  &  = \frac{i}{\overline{a}^2}\left(z+\frac{\,1\,}{\overline{z}}\right)
  + \frac{i\,b}{\overline{a}}\ln|z|
  -\frac{i\,\overline{b}}{\overline{a}^3\overline{z}^2}
  + \frac{1}{2}\left(z_0+\frac{\, 1\, }{\overline{z_0}}\right)
  + {\mathcal O}(\ve^2 z^{-3}) \vspace{3mm}\\

  &  = \frac{-1}{2}\left(z+\frac{\,1\,}{\overline{z}}\right)
  -\frac{(1+i)\overline{b}}{4\overline{z}^2}
  + \frac{1}{2}\left(z_0+\frac{\, 1\, }{\overline{z_0}}\right)
  + {\mathcal O}(\ve z^{-1}+\ve^2 z^{-3}) .
\end{array}
\]
Similarly, $\int_{z_0}^z dh= \ln z- \ln z_0+{\mathcal O}(\ve^2 z^{-2})$, hence
\[
X_3=\Re \int_{z_0}^z dh= \ln|z|- \ln|z_0|+{\mathcal O}(\ve^2 z^{-2}),
\]
which completes the proof of lemma \ref{lemX}. \hfill \qed

By suitably translating $M_{\s,\a,\b}$, we can assume its coordinate
functions are as in Lemma~\ref{lemX}.

\begin{lemma}
  \label{lemX_3} Let $(r,\t)$ denote the polar coordinates in the
  $\{x_3=0\}$ plane. If $\a+\b+\s\leqslant\ve$ small, then a piece of
  $M_{\s,\a,\b}$  can be written as a vertical graph of
  the function
  \[
 \widetilde U(r,\t)=-\ln(2r)+r\big(\kappa_1\cos\t-\kappa_2\sin\t \big)+
 {\cal O}(\ve),
 \]
 for
 $(r,\t)\in(\frac{1}{4\sqrt{\ve}},\frac{4}{\sqrt{\ve}})\times[0,2\pi)$,
 where $\kappa_1=\Re(b)+\Im(b)$ and $\kappa_2=\Re(b)-\Im(b)$.
 We denote by $M_{\s,\a,\b}(\g,\xi)$ the KMR example $M_{\s,\a,\b}$
 dilated by $1+\g$, for $\g\leqslant 0$ small, and translated by a
 vector $\xi=(\xi_1,\xi_2,\xi_3)$. Then,
 $M_{\s,\a,\b}(\g,\xi)$ can be written as a vertical graph of
 \[
 \widetilde{ U}_{\g,\xi}(r,\t)
 =-(1+\g)\ln\frac{2r}{1+\g}+r\left(\kappa_1\cos\t-\kappa_2\sin\t\right)-
 \frac{1+\g}{r}(\xi_1\cos\t+\xi_2\sin\t)+\xi_3+{\mathcal O}(\ve).
 \]
\end{lemma}

\begin{remark}
Recall that $b=\sin\frac{\a-\b}{2}+i\sin\frac{\a+\b}{2}$. In
particular:
\begin{itemize}
\item When $\b=0$, we have $\kappa_1=2\sin\frac{\a}{2}$ and $\kappa_2=0$, so
   \[
 \widetilde{U}_{\g,\xi}(r,\t)
 =-(1+\g)\ln\frac{2r}{1+\g}+r\,\kappa_1\,\cos\t-
 \frac{1+\g}{r}(\xi_1\cos\t+\xi_2\sin\t)+\xi_3+{\mathcal O}(\ve).
 \]
\item When $\a=0$, $\kappa_1=0$ and $\kappa_2=2\sin\frac{\b}{2}$, so
 \[
 \widetilde{U}_{\g,\xi}(r,\t)
 =-(1+\g)\ln\frac{2r}{1+\g}-r\,\kappa_2 \,\sin\t-
 \frac{1+\g}{r}(\xi_1\cos\t+\xi_2\sin\t)+\xi_3+{\mathcal O}(\ve).
 \]
  \item When $\a=0$ (resp. $\b=0$), we will consider $\xi_1=0$ (resp.
$\xi_2=0$) in Section~\ref{family.M.abt}.
\end{itemize}
\end{remark}

\noindent
{\bf Proof.}
  Suppose $|\frac{b}{a}|<|z|<\nu$, $\nu>|\frac{b}{a}|$ small.  From
  Lemma~\ref{lemX}, we know the coordinate functions $(X_1,X_2,X_3)$
  of the perturbed KMR example $M_{\s,\a,\b}(\g,\xi)$ are given by
\begin{equation}\label{eq:dilation}
\left\{\begin{array}{l}
    X_1+i\, X_2=
      -\frac{1+\g}{2}\left(z+\frac{\,1\,}{\overline{z}}\right)
      +A(z)\\
      \\
      X_3= (1+\g)\ln|z|+\xi_3+{\mathcal O}(\ve^2 z^{-2}),
    \end{array}\right.
  \end{equation}
  
  where
  \[
  A(z)= -\frac{(1+\g)(1+i)\overline{b}}{4\overline{z}^2} +(\xi_1+i\xi_2)
  +{\mathcal O}(\ve z^{-1}+\ve^2 z^{-3})
  \]
  \[
  =- \frac{(1+\g)(\kappa_1+i\kappa_2)}{4\overline{z}^2}
  +(\xi_1+i\xi_2) +{\mathcal O}(\ve z^{-1}+\ve^2 z^{-3})
  \]
  If we set $z=|z|e^{i\psi}$ and $X_1+i X_2= r e^{i\t}$, then
  $z+\frac{1}{\overline{z}}=\left(|z|+\frac{1}{|z|}\right)e^{i\psi}$
  and
  \[
  r\cos\t=-\frac{1+\g}{2}\left(|z|+\frac{1}{|z|}\right) \cos\psi+A_1 ,
  \]
    \[
  r\sin\t=-\frac{1+\g}{2}\left(|z|+\frac{1}{|z|}\right) \sin\psi+A_2 ,
  \]
  where  $A_1=\Re(A)$ and $A_2=\Im(A)$.
  Therefore,
  \begin{equation}\label{eq:r2}
\begin{array}{ll}
  r^2 = & \displaystyle{\frac{(1+\g)^2}{4}\left(|z|+\frac{1}{|z|}\right)^2\left( 1-\frac{4
      |z|}{(1+\g)(|z|^2+1)} \left(A_1 \cos\psi+A_2\sin\psi\right)\right.}\\
  & \\
  & \displaystyle{\left. +\frac{4 |z|^2}{(1+\g)^2(|z|^2+1)^2}(A_1^2+A_2^2) \right)} .
\end{array}
\end{equation}
  When $\frac{\sqrt{\ve}}{4}\leqslant|z|\leqslant 4\sqrt{\ve}$, the functions
  $A_i$ are bounded, and we get
  \begin{equation}\label{eq:r}
  r=\frac{1+\g}{2}\left(|z|+\frac{1}{|z|}\right)(1+{\mathcal
    O}(\sqrt{\ve})) =\frac{1+\g}{2|z|}+{\cal O}(\sqrt{\ve}) .
  \end{equation}
  In particular, $r={\mathcal O}(1/\sqrt{\ve})$.  Moreover, we get
  $\frac{r}{\frac{1+\g}{2}\left(|z|+\frac{1}{|z|}\right)} =1+{\mathcal
    O}(\sqrt{\ve})$, from where
 \[
  \frac{X_1+i
    X_2}{\frac{1+\g}{2}\left(|z|+\frac{1}{|z|}\right)}=e^{i\t}(1+{\cal O}(\sqrt{\ve})).
 \]
On the other hand,
 \[
 \frac{X_1+i X_2}{\frac{1+\g}{2}\left(|z|+\frac{1}{|z|}\right)}=
 -e^{i\psi}+\frac{2 |z| A}{(1+\g)(1+|z|^2)}=
 - e^{i\psi}+{\cal O}(\sqrt{\ve}) ,
  \]
  thus $e^{i\psi}=-e^{i\t}(1+ {\mathcal O}(\sqrt{\ve}))$.
  
  From (\ref{eq:r2}) and (\ref{eq:r}) we can deduce $\frac{(1+\g)^2(1+|z|^2)^2}{4|z|^2}
  =r^2\left(1+\frac{2}{r}\left(A_1 \cos\psi+A_2\sin\psi\right) +{\cal
      O}(\ve)\right)$, from where we obtain
  \begin{equation}\label{eq:z2}
    \begin{array}{ll}
    \displaystyle{\frac{1}{|z|^2}} & 
    \displaystyle{=\left(\frac{2r}{1+\g}\right)^2\left(1+\frac{2}{r}\left(A_1 \cos\psi+A_2\sin\psi\right)
      +{\cal O}(\ve)\right)(1+{\cal O}(\ve))}\\
    & \\
    & \displaystyle{=\left(\frac{2r}{1+\g}\right)^2\left(1+\frac{2}{r}\left(A_1 \cos\psi+A_2\sin\psi\right)
      +{\cal O}(\ve)\right)},
    \end{array}
  \end{equation}
  and then
\begin{equation}\label{eq:ln}
-\ln|z|=\ln\frac{2r}{1+\g}+\frac{1}{r}(A_1\cos\psi+A_2\sin\psi)+{\mathcal O}(\ve).
\end{equation}
Finally, it is not  difficult to prove that
\[
A_1=-\frac{1+\g}{4|z|^2}\left(\kappa_1
  \cos(2\psi)-\kappa_2\sin(2\psi)\right)+\xi_1+{\cal O}(\sqrt{\ve})
\]
\[\mbox{and}\qquad
A_2=-\frac{1+\g}{4|z|^2}\left(\kappa_1\sin(2\psi)+\kappa_2
  \cos(2\psi)\right)+\xi_2+{\cal O}(\sqrt{\ve}) .
\]
Therefore,
\[
A_1\cos\psi+A_2\sin\psi
=-\frac{1+\g}{4 |z|^2}\left(\kappa_1\cos\psi-\kappa_2\sin\psi\right)+\xi_1\cos\psi+\xi_2\sin\psi+{\cal O}(\sqrt{\ve})
\]
\[
=-\frac{r^2}{1+\g}\left(\kappa_1\cos\t-\kappa_2\sin\t\right)
+\xi_1\cos\t+\xi_2\sin\t+{\mathcal O}(\sqrt{\ve}).
\] 
From here, (\ref{eq:ln}) and (\ref{eq:dilation}), Lemma \ref{lemX_3} follows.
\hfill \qed

\subsection{Parametrization of the KMR example on the cylinder}
\label{subsecLame}


In this subsection we want to parametrize the KMR example
$M_{\s,\a,\b}$ on a cylinder.
Recall its conformal compactification $\Sigma_\s$  only depends on $\s$. The parameter $\s\in(0,\frac{\pi}{2})$
will remains fixed along this subsection, and we will  omit the dependence on~$\s$
of the functions we are introducing. 
Also recall that $\Sigma_\s$ can be seen as a branched $2$-covering of
$\overline\C$, by gluing $\overline{\C}_1,\overline\C_2$ along two
common cuts $\gamma_1$ and $\gamma_2$ along the imaginary axis joining
the branch points $D,D'$ and $D'',D'''$ respectively (see
(\ref{eqbranchings1})).

We introduce the
sphero-conal coordinates $(x,y)$ on the annulus $\esf^2-(\g_1\cup\g_2)$
(see~\cite{J}): for any $(x,y) \in S^1\times (0, \pi) \equiv
[0,2\pi) \times (0, \pi)$, we define
\[
F(x,y)=\left(\cos x\, \sin y, \sin x\, m(y),l(x)\,\cos y \right)\in
\esf^2-(\g_1\cup\g_2),
\]
where
\[
m(y)=\sqrt{1-\cos^2 \s \cos^2y} \quad\mbox{and}\quad l(x)=\sqrt{1-\sin^2\s \sin^2x}.
\]

To understand the geometrical meaning of these variables, remark that
$\{x=constant\}$ and $\{y=constant\}$ correspond to two closed curves
on $\esf^2$ which are two curves are the cross sections of the sphere
and two elliptic cones (one with horizontal axis, the other one with
vertical axis) having as vertex the center of the sphere.\\

If we compose $F(x,y)$ with the stereographical projection and enlarge
the domain of definition of the function, we obtain the differentiable
map ${\bf z}$ defined on the torus
$ S^1\times S^1\equiv[0,2\pi)\times [0, 2\pi) \to\overline\C$
and given by
\begin{equation}
\label{espressione.z}
{\bf z}(x,y)=\frac{\cos x\, \sin y+i\, \sin x\, m(y)}{1-l(x)\, \cos y},
\end{equation}
which is a branch $2$-covering of $\overline\C$ with branch values in
the four points whose sphero-conal coordinates are
$(x,y)\in\left\{\pm\frac{\pi}{2}\right\}\times\{0,\pi\}$, which
correspond to $D,D',D'',D'''$.  Moreover, ${\bf z}$ maps $ S^1
\times (0,\pi)$ on $\overline\C-(\gamma_1\cup\gamma_2)$. Hence we can
parametrize the KMR example by ${\bf z}$, via its Weierstrass
data
\[ 
g({\bf z})=\frac{a{\bf z}+b}{i(\overline{a}-\overline{b}{\bf z})} , \qquad
dh=\mu\,\frac{d{\bf z}}{\sqrt{({\bf z}^2+\l^2)({\bf z}^2+\l^{-2})}},
\]

We denote by $\widetilde M_{\s,\a,\b}$ the lifting of $M_{\s,\a,\b}$
to $\R^3/T$ by forgetting its non horizontal period (i.e. its period
in homology, $\widetilde T$).  We can then parametrize $\widetilde
M_{\s,\a,\b}$ on $S^1\times \R$ by extending ${\bf z}$ to
$[0,2\pi)\times\R$ periodically. But such a parametrization is not
conformal, since the sphero-conal coordinates $(x,y)\mapsto F(x,y)$ of
the sphere are not conformal.  As the stereographic projection is a
conformal map, it suffices to find new conformal coordinates $(u,v)$
of the sphere defined on the cylinder. In particular, we look for a
change of variables $(x,y)\mapsto(u,v)$ for which $|\widetilde
F_u|=|\widetilde F_v|$ and $\langle \widetilde F_u,\widetilde
F_v\rangle=0$, where $\widetilde F(u,v)=F(x(u,v),y(u,v))$.

We observe that
\[
|F_x|=\frac{\sqrt{k(x,y)}}{l(x)}\quad \hbox{and}\quad
|F_y|=\frac{\sqrt{k(x,y)}}{m(y)},
\]
with $k(x,y)=\sin^2\s\,\cos^2x+\cos^2\s\,\sin^2y$.  Then it is natural
to consider the change of variables $(x,y)\in [0,2\pi) \times\R
\mapsto (u,v) \in [0,U_\s) \times \R$ defined by
\begin{equation}
\label{change.coordinates}
u(x)=\int_0^x \frac{1}{l(t)}dt \quad \hbox{and} 
\quad v(y)=\int_{\frac \pi 2}^y \frac{1}{m(t)}dt,
\end{equation}
where
\begin{equation}
\label{U.sigma}
\textstyle{U_\s=u(2 \pi)=\int_0^{2 \pi} \frac{dt}{\sqrt{1-\sin^2\s
      \sin^2 t}}}.
\end{equation}
Note that $U_\s$ is a function on $\s$ that 
goes to $2\pi$ as $\s$ approaches to zero, and that the above change
of variables is well defined because $\s\in(0,\frac{\pi}{2})$.

In these variables $(u,v)$, ${\bf z}$ is $v$-periodic with period
\begin{equation}
\label{V.sigma}
\textstyle{V_\s=v(2\pi)-v(0)=\int_{0}^{2\pi} \frac{dt}{\sqrt{1-\cos^2\s
      \cos^2 t}}}.
\end{equation}
The period $V_\s$ goes to $+\infty$ as $\sigma$ goes to zero (see lemma
 \ref{remark.v.epsilon}),
 which was clear taking into account the limits of $M_{\s,\a,\b}$
as $\s$ tends to zero.

From all this, we can deduce that $\widetilde M_{\s,\a,\b}$ (resp.
$M_{\s,\a,\b}$) is conformally parametrized on $(u,v)\in I_\s\times\R$
(resp. $(u,v)\in
I_\s\times J_\s $), where $I_\s=[0,U_\s]$ and $J_\s=[v(0),v(2\pi)]$.

%% file: KMRPeriode.tex
In section \ref{Jacobi.THL}, devoted to the study of the mapping
properties of the Jacobi operator about $\widetilde M_{\s,\a,\b}$, we
will make exclusive use of the $(u,v)$ variables.
In the proof of Lemma~\ref{lemX_3} we have written as a vertical graph
an appropriate piece of $\widetilde M_{\s,\a,\b}$ corresponding to the
annulus $\frac{\sqrt{\ve}}{4}\leqslant|{\bf z}|\leqslant 4\sqrt{\ve}.$
The boundary curve along which we will glue the piece of $\widetilde
M_{\s,\a,\b}$ with the CHM surface corresponds to $|\bf z|=\sqrt \e.$

\begin{lemma}
\label{remark.v.epsilon}
If $\frac{\sqrt{\ve}}{4}\leqslant|{\bf z}|\leqslant 4\sqrt{\ve}$, then
$$ -\frac 1 2\ln \e + c_1 \leqslant v \leqslant 
-\frac 1 2\ln \e + c_2,$$
where $c_1$ and $c_2.$ Under the same assumptions 
 $V_\s=- 4\ln \e + {\mathcal
  O}(1)$. 
\end{lemma}
\begin{proof}
The proof is articulated in two steps.
Firstly, using equation \eqref{espressione.z}, 
 we will give the values of the $y$ variable
for which   $\frac{\sqrt{\ve}}{4}\leqslant|{\bf z}(x,y)|\leqslant 
4\sqrt{\ve}$ with $\s = {\mathcal O}(\e)$.
Secondly we will determine the values of the $v$ variable
corresponding to these values of $y.$
It is possible to show
that, if $|{\bf z}|$ varies in the interval specified above, 
then 
$\pi - a_1 \leqslant y \leqslant \pi -a_2,$
with $a_i=d_i \sqrt \e,$ where $d_1 > d_2>0$ are constants. 
Being $v$ an increasing function of $y,$ 
then $v(\pi-a_1)\leqslant v(y) \leqslant v(\pi-a_2).$
Now we will obtain the  values of $v(y)$ for $y=\pi-a_i,$ $i=1,2.$
To this aim we observe that
$$v(y)= \int_{\frac \pi 2}^{y} \frac{ds}{\sqrt{1-\cos^2 \s \cos^2 s}}.$$
It holds that
$$\int_{0}^{y}
\frac{ds}{\sqrt{1-\cos^2 \s \cos^2 s}} =
  \int_{0}^{y} \frac{ds}{\sqrt{1-\cos^2 \s+ \cos^2 \s\sin^2 s}} =$$
$$\frac{1}{\sin \s} \int_{0}^{y}
\frac{ds}{\sqrt{1+ \frac{\cos^2 \s}{\sin^2 \s}\sin^2 s}}=
\frac{1}{\sin \s} \mbox{\roc F}(y,m_\s),$$ where
 $m_\s=-\frac{\cos^2 \s}{\sin^2 \s}$ and
 $\mbox{\roc F}(y,m)=\int_{0}^y \frac{ds}{\sqrt{1-m\sin^2s}}$
 is the incomplete elliptic integral of the first kind.
 $\mbox{\roc F}(y,m)$ is an odd function with respect to $y$
  and if $k \in {\mathbb Z},$
 $$\mbox{\roc F}(y+k\pi,m)=\mbox{\roc F}(y,m)+2k\mbox{\roc K}(m), $$
 where $\mbox{\roc K}(m)=\mbox{\roc F}(\frac \pi 2 , m)$ is the
 complete elliptic integral of first kind.
  If $y=d \sqrt \e $ and $\s={\mathcal O}( \e),$ then
  $$\frac1{\sin \s}\mbox{\roc F}(y,m_\s)=-\frac 1 2 \ln \e +{\mathcal O}(1).$$
  It is known that if $m$ is sufficiently big then
 $$\mbox{\roc K}(m)=\frac{1}{\sqrt{-m}}
 \left( \ln 4+\frac 1 2 \ln(-m) \right)
 \left(1+ {\mathcal O}\left(\frac 1 m\right)\right). $$
 It follows that
$$\frac1{\sin \s} \mbox{\roc K}(m_\s)=  -  \ln  \e+{\mathcal O}(1).$$
Then, for $i=1,2,$ $$v(\pi - a_i)=\frac 1 {\sin \s} 
 \left(\mbox{\roc F}(\pi-a_i,m_\s)-\mbox{\roc K}(m_\s) \right)=$$
 $$\frac 1 {\sin \s} 
 \left(\mbox{\roc F}(-a_i,m_\s)+2\mbox{\roc K}(m_\s)-\mbox{\roc K}(m_\s) \right)=$$
$$ \frac 1 {\sin \s}\left(-\mbox{\roc F}(d_i \sqrt \e,m_\s)+
 \mbox{\roc K}(m_\s)\right)
= -\frac 1 2 \ln \e+c_i.$$
The result concerning $V_\s=v(2 \pi) - v(0)$
follows at once after having observed that 
$v(2\pi)= \frac{3\mbox{\roc K}(m_\s)}{\sin \s}$
and $v(0)=- \frac{\mbox{\roc K}(m_\s)}{\sin \s}.$
\end{proof}

From the lemma just proved it follows that
the value of the $v$ corresponding to $|\bf z|=\sqrt \e,$
is $v_\e=-\frac 1 2 \ln \e+c,$ where $c$ is a constant.   

%% file: KMR-JACOBI.tex
\section{The Jacobi operator about KMR examples}
\label{Jacobi.THL}
The Jacobi operator for $M_{\s,\a,\b}$ is given by ${\cal
  J}=\Delta_{ds^2} +|A |^2,$ where $|A|^2 $ is the squared norm of the
second fundamental form on $M_{\s,\a,\b}$ and $\Delta_{ds^2}$ is the
Laplace-Beltrami operator with respect to the metric $ds^2$ induced on
the surface by the immersion $X=(X_1,X_2,X_3)$ defined in section
\ref{Msab.graph}; i.e.
 \[
 ds^2=\frac{1}{4}\left(|g|+|g|^{-1}\right)^2|dh|^2.
 \]
 We consider the metric on the torus $\Sigma_\s$ obtained as pull-back
 of standard metric $ds_0^2$ of the sphere $\esf^2$ by the Gauss map
 $N$: $dN^*(ds_0^2)=-K\; ds^2$, where $K=-\frac 1 2 |A|^2$ denotes the
 Gauss curvature of $M_{\s,\a,\b}$. Hence $\Delta_{ds^2}= -K\,
 \Delta_{ds_0^2}$, and so
 \[
 {\cal J}=-K\; ( \Delta_{ds_0^2}+2) .
 \]

 From \cite{J} and taking into account the parametrization of
 $M_{\s,\a,\b}$ on the cylinder given in subsection~\ref{subsecLame},
 we can deduce that, in the $(x,y)$-variables,
\[
 \Delta_{ds_0^2}: =\frac{l(x)\,m(y)}{k(x,y)} \, \left[  \partial_{x}\left(\frac{l(x)}
{m(y)}\,\partial_{x}\right) + \partial_{y}\left(\frac{m(y)}
{l(x)}\,\partial_{y}\right)\right] .
\]
In the $(u,v)$-variables defined by (\ref{change.coordinates}), we
have ${\cal J}=\frac{-K}{k(x(u),y(v))}{\cal L}_\s,$ where
\begin{equation}  
{\cal L}_\s : = \partial^2_{uu} +\partial_{vv}^2 +
  2\sin^2\s\, \cos^2 (x(u))+ 2\cos^2\s\, \sin^2 (y(v))
\label{equationoperator}
\end{equation}
is the Lam\'e operator~\cite{J}.

\begin{remark}
  \label{operator-sigma=0} In Proposition \ref{inverse}, we will take
  $\s \rightarrow 0$.  For such a limit, the torus $\Sigma_\s$
  degenerates into a Riemann surface with nodes consisting of two
  spheres jointed by two common points $p_0,p_1$, and the
  corresponding Jacobi operator equals the Legendre operator on
  $\esf^2\setminus \{p_0, p_1\},$ given by ${\cal
    L}_0=\partial^2_{xx}+\sin y\, \partial_{y}\left(\sin y
    \,\partial_{y}\right)+ 2\sin^2 y$ in the $(x,y)$-variables.  Note
  that when $\s=0$, the change of variables $(x,y)\mapsto(u,v)$ given
  in (\ref{change.coordinates}) is not well-defined.
\end{remark}

\subsection{The mapping properties of the Jacobi operator}
From now on, we consider the conformal parametrization of $\widetilde
M_{\s,\a,\b}$ on the cylinder $\esf^1\times\R\equiv I_\s\times\R$
described in subsection~\ref{subsecLame}.  Our aim along this
subsection is to study the mapping properties of the operator ${\cal
  J}$. It is clear that it is sufficient to study the simpler operator
${\cal L}_\s$ defined by (\ref{equationoperator}), so we will study
the problem
\[
\left\{\begin{array}{ll}
    {\cal L}_\s w=f &\mbox{ in } I_\s\times[v_0,+\infty[ \\
    w=\varphi & \mbox{ on } I_\s\times\{v_0\}
  \end{array}\right.
\]
with $v_0\in\R$ and considering convenient normed functional spaces
for $w,f,\varphi$ so that the norm of $w$ is bounded by the one of
$f$.

We will work in two different functional spaces to solve the above
Dirichlet problem.  To explain the reason, it is convenient to recall
that the isometry group of $\widetilde M_{\s,\a,\b}$ depends on the
values of the three parameters $\s,\a,\b$. When $\a=0$ (resp. $\b=0$),
$\widetilde M_{\s,\a,\b}$ is invariant by the reflection symmetry
about the $\{x_1=0\}$ plane (resp. $\{x_2=0\}$ plane). We are
interested in showing the existence of families of minimal surfaces
close to $\widetilde M_{\s,0,\b}$ and $\widetilde M_{\s,\a,0}$,
keeping the same properties of symmetry.  Thus the surfaces in the
family about $\widetilde M_{\s,0,\b}$ (resp. $\widetilde M_{\s,\a,0}$)
will be defined as normal graphs of functions  defined in
$I_\s\times\R$ which are even (resp. odd) in the first variable.  We
will solve the above Dirichlet problem in the first case. The second
one follows similarly.

\begin{definition}
  \label{definition.space2} Given $\s\in(0,\pi/2)$, $\ell\in\N$,
  $\a\in(0,1)$, $\mu\in\R$ and an interval $I$, we define ${\cal
    C}^{\ell,\a}_{\mu} (I_\s\times I)$ as the space of functions $w=
  w(u,v)$ in ${\cal C}^{\ell,\a}_{loc}(I_\s\times I)$ which are even
  and $U_\s$-periodic in the variable $u$ and for which the following
  norm is finite:
\[
\|w \|_{{\cal C}^{\ell,\a}_\mu(I_\s\times I)}:= \sup_{v \in I} e^{-\mu
  v} \| w \|_{{\cal C}^{\ell,\a} (I_\s \times [v,v+1])}
\]
\end{definition}
It is possible to prove that $U_\s \to 2\pi$ as $\s \to 0$.
We observe that the Jacobi operator ${\cal L}_\s$ becomes a Fredholm
operator when restricted to the functional spaces defined in
\ref{definition.space2}. Moreover, ${\cal L}_\s$ has separated
variables. Then we consider the operator
\[
L_\s=\partial_{uu}^2+ 2\sin^2\s\cos^2(x(u)),
\]
defined on the space of $U_\s$-periodic and even functions in $I_\s$.
This operator $L_\s$ is uniformly elliptic and self-adjoint. In
particular, $L_\s$ has discrete spectrum $(\lambda_{\s,
  i})_{i\geqslant 0}$, that we assume arranged so that $\lambda_{\s,
  i} < \lambda_{\s,i+1}$ for every $i$. Each eigenvalue
$\lambda_{\s,i}$ is simple because we only consider even functions. We
denote by $e_{\s,i}$ the even eigenfunction associated to
$\lambda_{\s, i}$, normalized so that
\[
\int_0^{U_\s} (e_{\s,i}(u))^2 \, du =1.
\]

\begin{lemma}
\label{eigenvalue} For every $i\geqslant 0$, the eigenvalue
$\lambda_{\s,i}$ of the operator $L_\s$ and its associated
eigenfunctions $e_{\s,i}$ satisfy
\begin{equation}
  -2\, \sin^2\s\leqslant\lambda_{\s,i}- i^2 \leqslant 0 , \qquad
  |e_{\s,i}- e_{0,i}|_{{\mathcal C}^2}\leqslant c_i\, \sin^2\s ,
\label{estimateeigenvalue}
\end{equation}
where $e_{0, i}(u):=\cos(i x(u))$  
for every $u\in I_\s$, and the constant $c_i>0$ depends only on
$i$ (it does not depend on $\s$).
\end{lemma}
\noindent
\begin{proof}
  The bound for $\lambda_{\s,i}- i^2$ comes from the variational
  characterization of the eigenvalues
  \[
  \lambda_{\s, i} = \sup_{\hbox{\tiny codim} \, E = i} \, \inf_{\tiny \begin{array}{c}e \in E\\ 
    ||e||_{L^2}=1\end{array}} \int_0^{U_\s} \left( (\partial_u e)^2 - 2 \sin^2\s
    \cos^2 (x(u)) \, e^2 \right) \, du ,
  \]
  where $E$ is a subset of the space of $U_\s$-periodic and even
  functions in $L^2(I_\s)$, since it always holds $0\leqslant
  2\sin^2\s \cos^2(x(u))\leqslant 2\sin^2\s$.  The bound for the
  eigenfunctions follows from standard perturbation theory~\cite{K}.
\end{proof}

The Hilbert basis $\{e_{\s,i}\}_{i\in\N}$ of the space of
$U_\s$-periodic and even functions in $L^2(I_\s)$ introduced above
induces the following Fourier decomposition of functions $g= g(u,v)$
in $L^2(I_\s\times\R)$ which are $U_\s$-periodic and even in the
variable $u$,
\[
g(u,v)=\sum_{i\geqslant 0}g_i(v)\, e_{\s,i}(u).
\]
From this, we deduce that the operator ${\cal L}_\s$, can be
decomposed as ${{\cal L}_\s=\sum_{i\geqslant 0}L_{\s,i}}$, being
\[
L_{\s,i} = \partial_{vv}^2 + 2 \cos ^2 \s \sin^2 (y(v))-
\lambda_{\s,i}\, ,\quad\mbox{for every } i\geqslant 0.
\]
\noindent
Since $0 \leqslant 2 \cos^2 \s \sin^2 (y(v)) \leqslant 2\cos^2
\s=2-2\sin^2\s$, Lemma \ref{eigenvalue} gives us
\begin{equation}
  \label{potential}
P_{\s,i}:=2 \cos^2\s\, \sin^2(y(v))-\lambda_{\s,i}\leqslant 2- i^2.
\end{equation}
This fact allows us to prove the following lemma, which assures that
${\cal L}_\s$ is injective when restricted to the set of functions
that are, in the variable $u$, even and $L^2$-orthogonal to $e_{\s,
  0}$ and $e_{\s, 1}$.

\begin{lemma}
\label{injectivity} Given $v_0< v_1$, let $w$ be a solution of
${\cal L}_{\s}w=0$ on $I_\s\times[v_0, v_1]$ such that
\begin{itemize}
  \item[(i)] $w(\cdot\,,v_0)= w(\cdot\,,v_1)=0$
 \item [(ii)] $\int_{0}^{U_\s} w(u,v) e_{\s,i}(u)\, du =0$ , for
   every $v\in[v_0, v_1]$ and $i\in\{0,1\} $
 \end{itemize}
 Then $w=0$.
\end{lemma}

\begin{proof}
  By {\it(ii)}, $w=\sum_{i\geqslant 2} w_i(v)\,e_{\s,i}(u).$
Since the potential $P_{\s,i}$
  of the operator $L_{\s ,i}$ is negative for every $i\geqslant 2$
  (see~(\ref{potential})) and the operator $L_{\s,i}$ is
  elliptic, the maximum principle holds.  We can then conclude the
  Lemma \ref{injectivity} from {\it (i)}.
\end{proof}

To study the mapping properties of the Jacobi operator ${\cal L}_\s$,
we need to give a description of the Jacobi fields associated to
$M_{\s,\a,0}$, which are defined as the solutions of ${\cal L}_\s
v=0.$

Since $M_{\s,\a,0}$ is invariant by the reflection symmetry across the
plane $\{x_2=0\}$, there are only four independent Jacobi fields:

\begin{itemize}
\item Two Jacobi fields induced by vertical translations and by
  horizontal translations in the $x_1$-direction.  These Jacobi fields
  are clearly periodic and hence bounded.
\item A third Jacobi field generated by the 1-parameter group of
  dilations, which is not bounded (it grows linearly).
\item The last Jacobi field can be obtained by considering the
  1-parameter family of minimal surfaces induced by changing the
  parameter $\s$. This Jacobi field is not periodic and grows
  linearly.
\end{itemize}
The Jacobi fields induced by translation along the $x_3$-axis and by
dilatation, are solutions of ${\cal L}_\sigma u=0$ that are collinear
to the eigenfunction $e_{\sigma,0}$. The Jacobi fields induced by the
horizontal translation and by the variation of the parameter $\sigma$,
are collinear to $e_{\sigma,1}$.

The Jacobi fields of $M_{\s,0,\b},$ which is invariant by the 
reflection symmetry across the plane $\{x_1=0\}$ are the 
same as  $M_{\s,\a,0},$ with the exception of that one
induced by horizontal translations in the $x_1$-direction
which is to be replaced by the field induced
by horizontal translations in the $x_2$-direction. 

The following proposition states that, if the parameter $\mu$ is
appropriately chosen, there exists a right inverse for ${\cal L}_\s$
whose norm is uniformly bounded.
   
\begin{proposition}
  \label{inverse} Given $\mu\in(-2,-1),$ there exists
  $\s_0\in(0,\pi/2)$ so that, for every $\s \in (0, \s_0)$ and
  $v_0\in\R$, there exists an operator
  \[
  \begin{array}{cccc}
    G_{\s,v_0}: & {\cal C}^{0,\a}_{\mu}(I_\s\times [v_0,+ \infty)) &
    \longrightarrow & {\cal C}^{2,\a}_{\mu}(I_\s\times [v_0,+ \infty))\\
    \end{array}
  \]
  such that for all $f\in {\cal C}^{0,\a}_{\mu}(I_\s\times [v_0,+
  \infty))$, the function $w:= G_{\s,v_0}(f)$ solves
  \begin{equation}
  \left\{\begin{array}{l}
  {\cal L}_\s\,w=f  \quad \mbox{, on } I_\s\times[v_0,+\infty)\\
  w\in\mbox{{\rm Span}}\{e_{\s,0},e_{\s,1}\} \mbox{, on }  I_\s\times\{v_0\}.
   \end{array}\right.
  \end{equation}
    Moreover  
    $||w||_{{\cal C}^{2,\a}_{\mu}}\leqslant c \,||f||_{{\cal
        C}^{0,\a}_{\mu}}$, for some constant $c>0$ which does not
    depend on $\s \in (0, \s_0)$ nor $v_0 \in \R.$
\end{proposition}
\begin{proof}
Every $f\in{\cal C}^{0,\a}_{\mu}(I_\s\times [v_0,+\infty))$ can be
decomposed as
\[
f= f_0 \, e_{\s, 0} + f_1 \, e_{\s,1}+\bar f,
\]
where $\bar f (\cdot, v)$ is $L^2$-orthogonal to $e_{\s, 0}$
and to $e_{\s, 1}$, for each $v\in[v_0,+\infty)$.\\

\noindent{\bf Step 1.} Firstly, let's prove Proposition \ref{inverse}
for functions $ \bar f\in{\cal C}^{0,\a}_{\mu}(I_\s\times [v_0,+
\infty))$ that are $L^2$-orthogonal to $\{e_{\s, 0},e_{\s, 1}\}$.  As
a consequence of the Lemma \ref{injectivity}, ${\cal L}_\s$ is
injective when acting on the space of these functions.  Hence, 
Fredholm alternative assures there exists, for each $v_1 > v_0 +1$, a
unique $\bar w\in{\cal C}^{2,\a}_\mu$, with $\bar w(\cdot , v)$
$L^2$-orthogonal to $e_{\s, 0}, e_{\s, 1}$ satisfying:
  \begin{equation}\label{eqwf}
    \left\{\begin{array}{l}
        {\cal L}_\s\, \bar w = \bar f \quad \mbox{ on } I_\s \times [v_0,v_1],\\
          \bar w( \cdot , v_0)= \bar w( \cdot ,v_1)=0.
    \end{array}\right.
  \end{equation}

  \begin{assertion}
  \label{assertion.step1}
  There exist $c\in\R$ and $\s_0\in(0,\pi/2)$ such that, for every
  $\s\in(0,\s_0)$, $v_0\in\R$, $v_1>v_0+1$ and $\bar f\in{\cal
    C}^{0,\a}_{\mu}(I_\s\times [v_0,v_1])$, there exists $\bar
  w\in{\cal C}^{2,\a}_{\mu}(I_\s\times [v_0,v_1])$ $L^2$-orthogonal to
  $\{e_{\s,0},e_{\s,1} \}$ satisfying (\ref{eqwf}) and
 \begin{equation}
\label{estimate.assertion}
\sup_{I_\sigma \times [v_0,v_1]}  e^{-\mu v}|\bar w|\leqslant 
c\,\sup_{I_\sigma \times [v_0,v_1]}  e^{-\mu v} |\bar f|.
\end{equation}
\end{assertion}

Suppose by contradiction that Assertion \ref{assertion.step1} is false.
    Then, for every $n \in \N$, there exists $\s_n\in(0,1/n)$,
    $v_{1,n}>v_{0,n}+1$ and $\bar f_n,\bar w_n$ satisfying (\ref{eqwf}) (for
    $\s_n,v_{0,n},v_{1,n}$ instead of $\s,v_0,v_1$) such that

$$\sup_{I_{\s_n} \times [v_{0,n},v_{1,n}]}
  e^{-\mu v} \, |\bar f_n|=1,\qquad \mbox{and} $$

$$A_n:=\sup_{I_{\s_n} \times [v_{0,n},v_{1,n}]}
  e^{-\mu v} \, |\bar w_n| \to+\infty \quad \hbox{as} \quad n\to\infty.$$

\noindent
Since $I_{\s_n} \times [v_{0,n},v_{1,n}]$ is a compact set,
$A_n$
is achieved at a point $(u_n,v_n)\in I_{\s_n}
\times [v_{0,n}, v_{1,n}]$.

After passing to a subsequence, the intervals
$I_n=[v_{0,n}-v_n,v_{1,n}-v_n]$ converge to a set $I_\infty$.
Elliptic estimates 
imply that
\[
\sup_{I_{\s_n} \times [v_{0,n},v_{0,n}+\frac 1 2]} e^{-\mu v} \,
|\nabla \bar w_n| \leqslant c(\sup_{I_{\s_n} \times
  [v_{0,n},v_{0,n}+1]} e^{-\mu v} \, |\bar f_n|+ \sup_{I_{\s_n} \times
  [v_{0,n},v_{0,n}+1]} e^{-\mu v} \, |\bar w_n|) \leqslant c (1+A_n).
\]
From this estimate for the gradient of $\bar w_n$ near $v=v_{0,n}$ it
follows that $v_n$ cannot be too close to $v_{0,n}$, where $\bar w_n$
vanishes.  More precisely, $v_{0,n}-v_n$ remains bounded away from
$0$, and then converges to some $\bar v_0 \in [-\infty,0)$. Using
similar arguments, it is possible to show that $\nabla \bar w_n$ is
bounded near $v_{1,n}$ and consequently $v_{1,n}-v_n$ converges to
some $\bar v_1 \in (0,\infty]$.  Then 
we can conclude that $I_\infty =
[\bar v_0, \bar v_1].$

We define
\[
\tilde w_n(u,v):= \frac{e^{-\mu v_n}} {A_n}
\bar w_n(u,v+ v_n),\quad (u,v)\in I_{\s_n}\times I_n .
\]
We observe that
\[
|\tilde w_n(u,v)|\leqslant e^{\mu\,v}\,\frac{e^{-\mu
    (v+ v_n)}|\bar w_n(u,v+ v_n)|}{A_n}\leqslant e^{\mu\,v}
\]
\begin{equation}
\mbox{and}\quad  \sup_{I_{\s_n} \times I_n} e^{-\mu v} \, |\tilde w_n|=1.
\label{equationcontradiction.n}
\end{equation}
Using the above estimate for $\nabla \bar w_n$ we obtain
$$
|\nabla \tilde w_n|\leqslant c\frac{1+A_n}{A_n}\,e^{\mu\,v}< 2c\,e^{\mu\,v}.
$$

Since the sequences $\{\tilde w_n\}_n$ and $\{\nabla \tilde w_n\}_n$
are uniformly bounded, Ascoli-Arzel\`a's theorem assures that, if
$n\to+\infty$ a subsequence of $\{\tilde w_n\}_n$ converges on compact
sets of $I_{0}\times I_\infty$ to a function $\tilde w_\infty$ which
vanishes on $I_0\times\partial I_\infty$, when $\partial I_\infty \neq
\emptyset$, and such that $\tilde w_\infty(\cdot,v)$ is $L^2$-orthogonal
to $\{e_{0,0}, e_{0,1}\}$ for each $v\in I_\infty$.  Moreover,
\begin{equation}
\sup_{I_{0} \times I_\infty} e^{-\mu v} \, |\tilde w_\infty|=1.
\label{equationcontradiction}
\end{equation}

Since $\s_n \to 0$ as $n \to \infty$, we can conclude that $\tilde
w_\infty$ satisfies
\[
\left\{\begin{array}{l}{\cal L}_0 \tilde w_\infty=0\\
   \tilde w_\infty|_{I_0 \times\partial I_\infty}=0\ \mbox{(if }\partial I_\infty \neq \emptyset\mbox{)}
  \end{array}\right.
\]

If $I_\infty$ is bounded, the maximum principle allows us to conclude
that $\tilde w_\infty=0$ on $I_0 \times I_\infty$, which
contradicts~\eqref{equationcontradiction}.  Hence $I_\infty$ is an
unbounded interval.

Recall ${\cal L}_0$ is given in terms of the $(x,y)$-variables. The
equation ${\cal L}_0 \tilde w_\infty=0$ becomes
 $$\partial^2_{xx}\tilde w_\infty+\sin y\,
\partial_{y}\left(\sin y \,\partial_{y}\tilde w_\infty \right)+ 
2\sin^2 y\,\tilde w_\infty=0.$$
Now we consider the eigenfunctions decomposition of $\tilde w_\infty$,
$$\tilde w_\infty(x,y)= \sum_{j\geqslant 2} a_{j}(y)\, \cos(j\,x).$$ 
Each coefficient $a_{j}$, $j \geqslant 2$, must satisfy the associated
Legendre differential equation (see Appendix C) 
$$\sin y\, \partial_{y}\left(\sin y \,\partial_{y}a_{j} \right)
-j^2a_{j}+2 \sin^2 y \,a_{j} =0.$$ We obtain that $a_j(y)$ equals the
associated Legendre functions of second kind, i.e.
$a_{j}(y)=Q^j_1(\cos y),$ $j \geqslant 2.$

We obtain
from (\ref{change.coordinates}) that
\[
u(x) \to x \quad \hbox{and} \quad v(y) \to \frac 1 2 \ln \left|\tan
  \frac{y}{2}\right|\quad\mbox{as } \s \to 0.
\]
In particular, define $y(v)=2 \arctan(e^{2v})$ for $\s=0$.  Therefore,
\begin{equation}
\label{change.coordinates.-1}
 \cos y(v)=\frac{1-e^{4v}}{1+e^{4v}} 
 \qquad \mbox{and}
\end{equation}
\[
\tilde w_\infty(u,v)= \sum_{j\geqslant 2}
Q^j_1\left(\frac{1-e^{4v}}{1+e^{4v}}\right)\cos(j\,u).
\]
It is possible
to show that  $|a_j|$ tend to $+\infty$ as the 
function $e^{2j|v|}.$  
 Since the interval $I_\infty$ is unbounded, we reach a
contradiction with~\eqref{equationcontradiction}.  This proves
Assertion \ref{assertion.step1}.\\

Let $c\in\R$ and $\s_0$ satisfy Assertion \ref{assertion.step1} and
fix $\s\in(0,\s_0)$, $v_0\in\R$ and $\bar f\in{\cal
  C}^{0,\a}_{\mu}(I_\s\times [v_0,v_1])$. Then, for every $v_1>v_0+1$,
there exists a function $\bar w$ which is $L^2$-orthogonal to
$\{e_{\s,0},e_{\s,1} \}$ and satisfies (\ref{eqwf}) and
(\ref{estimate.assertion}).  Let's take the limit as $v_1\to\infty.$
Clearly,
$$
e^{-\mu\,v}|\bar w|\leqslant\|\bar w\|_{{\cal C}^{0,\a}_\mu(I_\s\times
  [v_0,v_1])}\leqslant c\, \|\bar f\|_{{\cal C}^{0,\a}_\mu(I_\s\times
  [v_0,v_1])}.
$$
And using Sch\"auder estimates we get
 $$ e^{-\mu\,v}|\nabla \bar w|\leqslant  
 \|\bar w\|_{{\cal C}^{2,\a}_\mu(I_\s\times [v_0,v_1])}\leqslant
 c_1 \left(\|\bar f \|_{{\cal C}^{0,\a}_\mu(I_\s\times [v_0,v_1])}+  
 \|\bar w\|_{{\cal C}^0_\mu(I_\s\times [v_0,v_1])}\right)
 \leqslant c_2 \|\bar f \|_{{\cal C}^{0,\a}_\mu(I_\s\times [v_0,v_1])} .
   $$

\noindent
Hence Ascoli-Arzel\`a's theorem assures that a subsequence of
$\{\bar w_{v_1}\}_{v_1>v_0+1}$ converges to a function $\bar w$
defined on $I_\s\times[v_0,+\infty)$, which 
 satisfies $$\sup_{I_\sigma \times [v_0,+\infty)}  e^{-\mu\,v}|\bar w|
\leqslant  c\, \sup_{I_\sigma \times [v_0,+\infty)}  e^{-\mu\,v} 
|\bar f|. $$ Using again elliptic estimates we can conclude
that $\bar w$ satisfies the statement 
of Proposition \ref{inverse}. The uniqueness of the solution
follows from Lemma \ref{injectivity}.\\

\noindent {\bf Step 2} Let's now consider $f\in{\cal
C}^{0,\a}_{\mu}(I_\s\times [v_0,+ \infty))$ in ${\rm Span} \{e_{\s,
0},e_{\s, 1}\}$, i.e.
$$f(u,v)= f _0(v) \, e _{\s, 0}(u)   +
f _1(v) \, e _{\s, 1}(u) .$$ We extend the functions $f_0(v),f_1(v)$
for $v \leqslant v_0$ by $f_0(v_0), f_1(v_0)$, respectively.  Given
$v_1>v_0+1$, consider the problem
\begin{equation}
\label{eqwf2} \left\{\begin{array}{l}
    L_{\s,j}w_j
    = f _j,\quad \ v\in(-\infty, v_1]\\
    w _j(v_1)= \partial_v w _j (v_1)=0.
  \end{array}\right.
\end{equation}
Peano theorem assures the existence and the uniqueness of the solution $w _j.$
Our aim consists of proving the following result.

\begin{assertion}\label{assertion.step2}
$\sup_{(-\infty,v_{1}]} e^{-\mu v } |w_{j}| \leqslant 
c \sup_{(-\infty,v_{1}]} e^{-\mu v } |f_{j}| $
for some constant $c$ which does not depend on $v_1.$
\end{assertion}
\noindent
Suppose by contradiction that, for every $n\in\N$, there exists
$\s_n\in(0,1/n)$, $v_{1,n}>v_{0,n}+1$ and $f_{j,n},w_{j,n}$
satisfying~(\ref{eqwf2}) such that

$$\sup_{(-\infty,v_{1,n}]}
e^{-\mu v} \, |\bar f_{j,n}|=1, $$

$$A_n:=\sup_{ (-\infty,v_{1,n}]}
e^{-\mu v} \, |w_{j,n}| \to+\infty \quad \hbox{as} \quad n\to\infty.$$

\noindent The solution $w_{j,n}$ of the previous equation is a linear
combination of the two solutions of the homogeneous problem
$L_{\s_n,j}\, w =0.$ They grow at most linearly at $\infty$ (recall
that the Jacobi fields have this rate of growth).  Hence the supremum
$A_n$ is achieved at a point $v_n \in (-\infty, v_{1,n}].$ We define
on $I_n:=(-\infty,v_{1,n}-v_n]$ the function
\[
\tilde w_{j,n} (v):= \frac{1}{A_n} \, e^{- \mu v_n}\, w_{j,n}
(v_n + v).
\]

As in Step 1, one shows that the sequence
$\{v_{1,n}-v_n\}_n$ remains bounded away from $0$ and, after passing
to a subsequence, it converges to $\bar v_1 \in (0, +\infty]$ and
$\{\tilde w_{j,n}\}_n$ converges on compact subsets of $
I_\infty=(-\infty,\bar v_1]$ to a nontrivial function $\tilde w_j$
such that
\begin{equation}
\sup_{v \in I_\infty} e^{-\mu v} \, |\tilde w_j|=1
\label{equationcontradiction2}
\end{equation}
and $\tilde w_j ( \bar v_1)=\partial_v w_j ( \bar v_1)= 0$ , if $\bar
v_1 < +\infty$.  The function $\tilde w_j$ is the solution of a second
order ordinary differential equation, given, in terms of the
$(x,y)$-variables, by
\begin{equation}
\label{equationlimit2}
\sin y\, \partial_{y}\left(\sin y \,\partial_{y}\tilde w_j \right)
-j^2\,\tilde w_j+2 \sin^2 y \,\tilde w_j  =0.
\end{equation}
If $\bar v_1 < +\infty$ then $\tilde w_j=0$ and this is a
contradiction with \eqref{equationcontradiction2}. In the case $\bar
v_1 = +\infty$ we will try to reach a contradiction determining the
solution \eqref{equationlimit2}. This is again the associated Legendre
differential equation (see Appendix C).
The solutions of
the equation (\ref{equationlimit2}) are a linear combination of the
associated Legendre functions of first kind, $P_1^j(\cos y),$ and
second kind, $Q_1^j(\cos y),$ with $j=0,1.$ It holds that $P_1^0(\cos
y)=\cos y$ and $P_1^1(\cos y)=-\sin y$.  We change of coordinates to
express $\tilde w_j$ in terms of the $(u,v)$-variables. It is possible
to show that as $v \to \pm \infty$ then $|Q_1^1(\cos y(v))|$ and
$|Q_1^0(\cos y(v))|$ tend to $\infty$ respectively as $e^{2|v|} $ and
$|v|.$ We can conclude that functions $\tilde w_1$ and $\tilde w_0$ do
not satisfy the equation (\ref{equationcontradiction2}) with $\mu \in
(-2, -1),$ a contradiction. \\

Therefore, $\sup_{(-\infty, v_1]} e^{- \mu v} \, |w_j| \leqslant c\,
\sup_{(-\infty, v_1]} e^{- \mu v} \, |f_j|$. Taking $v_1\to+\infty$,
we get a solution of $L_{\s,j} \, w_j = f_j$ defined in $ [v_0,
+\infty)$ which satisfies
\[
\sup_{ [v_0, +\infty)} e^{-\mu v} \, |w_j| \leqslant c \,
\sup_{ [v_0, +\infty)} e^{-\mu v} \, |f_j|.
\]
Elliptic estimates allow us to obtain the wanted estimates for the
derivatives.  To prove the uniqueness of solution, it is sufficient to
observe that no solution of ${\cal L}_\s v =0$ that is collinear with
$e_{\s,0}$ and $e_{\s,1}$ decays exponentially at $\infty$. This fact
follows from the behaviour of the Jacobi fields.
\end{proof}

%% file: KMR-FAMILY.tex
\section {A family of minimal surfaces close to $\widetilde M_{\s,0,\beta}$ 
and $\widetilde M_{\s,\a,0}$}
\label{family.M.abt}
The aim of this section is to find a family of minimal surfaces 
near to a translated and dilated copy of $\widetilde M_{\s,0,\b}$ 
and $\widetilde M_{\s,\a,0}$ with given Dirichlet data on the
boundary.
We start recalling that in subsection \ref{Msab.graph}
we observed that a translated and dilated copies of $\widetilde M_{\s,\a,0}$ 
and $\widetilde M_{\s,0,\b}$ 
can be expressed as the graphs over the 
$\{x_3=0\}$ plane respectively of the functions 
\begin{equation}
\label{graphannulus1}
\tilde U_{\gamma,\xi}(r,\t)=-(1+\g)\ln(2r)+r\kappa_{1} \cos\t-
\frac{1+\g}{r}\xi_1\cos\t+d+{\mathcal O} (\e)
\end{equation}
\begin{equation}
\label{graphannulus2}
\tilde U_{\gamma,\xi}(r,\t)=-(1+\g)\ln(2r)-r \kappa_{2}\sin\t -
\frac{1+\g}{r}\xi_2\sin\t +d+{\mathcal O} (\e)
\end{equation}
where 
$d=\xi_3+(1+\g) \ln (1+\g),$
$\kappa_1,\kappa_2,\xi_1,\xi_2,\xi_3 ,\g\in\R$  
small enough, $\kappa_1=b_1+b_2,$ $\kappa_2=b_1-b_2,$ 
$b_1=\sin \frac{\a-\b} 2,$ 
$b_2=\sin \frac{\a+\b} 2,$ 
and $r$ belongs to a neighbourhood of $r_\e=\frac{1}{2\sqrt{\e}}.$\\

%
%
%

\noindent
We denote by $Z$ the immersion of the surface $\widetilde M_{\s,\a,\b}.$
The following proposition, whose proof is contained in 
section \ref{appendixb},  
states that the linearized of the 
mean curvature operator is the Lam\'e operator introduced
in section \ref{subsecLame}.
\begin{proposition}
\label{operator.lame} The surface parameterized by $Z_{f}:= Z +f \, N$
is minimal if and only if the function $f$ is a solution of
$${\cal L }_\s f =Q_\s(f).$$
\noindent where ${\cal L}_\s : = \partial^2_{uu} +
\partial_{vv}^2 +   2\sin^2\s\, \cos^2 (x(u))+ 
2\cos^2\s\, \sin^2 (y(v))$ is the Lam\'e operator and $Q_\s$ is a
nonlinear operator which satisfies
\begin{equation}
\label{proprieta.Q.lame}
\| Q_\s(f_2) - Q_\s(f_1)\|_{{\mathcal C}^{0, \alpha} (I_\s \times
[v,v+1])} \leqslant c \, \sup_{i=1,2} \| f_i\|_{{\mathcal C}^{2, \alpha}
(I_\s \times [v,v+1])} \,  \| f_2 -f_1\|_{{\mathcal C}^{2, \alpha}
(I_\s \times [v,v+1])}
\end{equation}
for all $f_1, f_2$ such that $\| f_i\|_{{\mathcal C}^{2, \alpha}
(I_\s \times [v,v+1])} \leqslant 1$. Here the constant $c >0$ does not
depend on $v \in {\mathbb R}$, nor on $\s \in (0,\frac \pi 2)$.
\end{proposition}

\noindent
As a consequence of the dilation of factor $1+\g$ of the surface 
the minimal surface equation becomes
\begin{equation}
\label{operator.THL}
 {\cal L}_\s \, w =\frac{1}{1+\g} Q_\s \left((1+\g)w \right).
\end{equation}

\noindent
We now truncate the surfaces $\widetilde M_{\s,\a,0}$ and 
$\widetilde M_{\s,0,\b}$ at the graph of the
curve $r = \frac 1 {2  \sqrt {\e}}$ of the function
\eqref{graphannulus1} and \eqref{graphannulus2}
 with, respectively, $\b=0$ and $\a=0$
and we consider only the upper half  of these
surfaces which we   call $M_1$ and $M_2.$\\

\noindent
We are interested in minimal normal graphs over these surfaces
 which are  asymptotic to them. 
The normal graph of the function $w$ over $M_1,M_2$
is minimal, if and only if $w$ is a solution of (\ref{operator.THL}).\\

\noindent
In subsection \ref{Msab.graph} 
we have proved lemma 
\ref{lemX_3}, which allows us to parametrize 
the surfaces $M_1,M_2$ in a neighbourhood 
of $r_\e$ as the graph of the functions given above.
The surfaces $M_1$ and $M_2$ are parametrized on
$I_\s \times [v_\e,V_\s]$ where $v_\e$ denote
value of the variable  $v$ corresponding
to the value $r_\e$ of the variable $r.$
In section \ref{sec:KMR}
we showed 
that $v_\e=-\frac 1 2 \ln \e+{\mathcal O}(1).$\\

\noindent
It is important to remark that the boundary  of these surfaces 
does not correspond to the curve $v = v_\e.$ 
We therefore modify the above parametrization so
that for $v \in [v_\e-\ln 2,v_\e +\ln 2]$
the image of the functions \eqref{graphannulus1}
and \eqref{graphannulus2}
corresponds to the horizontal curve $v = v_\e$.
Furthermore we would like that the normal vector field relative to 
$M_1,M_2$ is vertical near the boundary of this
surface. This can be achieved by modifying the normal vector field
into a transverse vector field $\tilde N$ which agrees with the
normal vector field $N$ for all $v \geqslant v_\e + \ln 4$ and with
the vector $e_3$ for all $v \in [v_\e , v_\e + \ln 2]$.\\

\noindent
Now, we consider a graph over this surface for some function $u$,
using the modified vector field $\tilde N$. This graph will be
minimal if and only if the function $u$ is a solution of 
a nonlinear elliptic equation related to 
(\ref{operator.THL}). To get the new equation, we
take
into account the effects of the change of parameterization and the
change in the vector field $N$ into $\tilde N.$ The
new minimal surface equation is
\begin{equation}
\label{operator.tilde} {\cal L}_\s  \, w   =  \tilde L_\e \, w + \tilde Q_\s \left(w \right).\\
\end{equation}
\noindent
Here $\tilde Q_\s$ enjoys the same properties as $Q_\s,$ since it is
obtained by a slight perturbation from it.
The operator $\tilde L_\e$ is a linear second order operator whose
coefficients are supported in $[ v_\e , v_\e + \ln 4] \times S^1
$ and are bounded by  a constant multiplied for $\e^{1/2}$, in
${\mathcal C}^\infty$ topology, where partial derivatives are
computed with respect to the vector fields $\partial_u$ and
$\partial_v$.\\

\noindent
As a fact, if we
take into account the effect of the change of the normal vector
field, we would obtain, applying the result of Appendix B of \cite{HP}, a
similar formula where the coefficients of the corresponding
operator $\tilde L_\e$ are bounded by a constant multiplied for
$\e$  since
$$ \tilde N_\e \cdot N_\e  = 1 + {\mathcal O} (\e) $$
for $t \in [t_\e , t_\e + \ln 2]$. Instead, if we take into
account the effect of the change in the parameterization, we would
obtain a similar formula where the coefficients of the
corresponding operator $\tilde L_\e$ are bounded by a constant
multiplied for $\e^{1/2}.$ 
The estimate of the coefficients of
$\tilde L_\e$ follows from these considerations.\\

\noindent
In the following of this section we will give 
the detailed proof of the existence of a family of a minimal graph 
about $M_1$ and asymptotic to it. The proof relative to the case where
the surface $M_1$ is replaced by $M_2$ can be obtained easily from
the previous one. We recall that the surfaces $M_1$ and $M_1$ are
respectively invariant with respect to two different mirror symmetries. 
A normal graph of the function $w=w(u,v)$ about $M_1$ (respectively about $M_2$) 
inherits the same property of symmetry if $w$ is an even 
(respectively odd) function with respect to the variable $u.$ 
Then once the proof of the result concerning $M_1$ will be completed,
 to obtain the result about $M_2$ it will be sufficient to replace
 even functions with odd functions and to repeat the 
 same arguments.\\

\noindent
Now, assume that we are given a function $\varphi \in {\mathcal
C}^{2, \alpha} (I_\s)$ which is even with respect to $u$,
$L^2$-orthogonal to $e_{\s,0},e_{\s,1}$ and such that
\begin{equation}
\label{stima.varphi.KMR}
 \| \varphi \|_{{\mathcal C}^{2,\a}} \leqslant k \e.
 \end{equation}
\noindent
We define
$$ w_\varphi  (\cdot, \cdot) : = \bar {\mathcal H}_{v_\e,\varphi}
 ( \cdot,\cdot ),$$
where $v_\e=- \frac 1 2 
\ln \e +{\mathcal O}(1)$ and $\bar {\mathcal H}$ is introduced
in proposition
\ref{otheroperator}.

\noindent
In order to solve the equation (\ref{operator.tilde}), we choose
$\mu \in (-2, -1)$ 
and look for $u$ of the form $u = w_\varphi +g$
where $g \in {\mathcal C}^{2, \alpha}_\mu (I_\s  \times [v_\e , +\infty)
).$ 
Using  
Proposition \ref{inverse}, we can rephrase this problem as a fixed
point problem
\begin{equation}
\label{fixedpointequation} g =  S (\varphi,g)
\end{equation}
where the nonlinear mapping $S$ which depends on $\e$ and
$\varphi$ is defined by
$$ S (\varphi,g) : = {G}_{\e,  v_\e} \left( \tilde L_\e (w_\varphi +g) -  
{\cal L}_\s \, w_{\varphi}  +
 \tilde Q_\e \left( w_\varphi + g \right) \right).$$
where the operator $G_{\e,  v_\e}$ is defined in   Proposition
\ref{inverse}. To prove the existence of a fixed point for
(\ref{fixedpointequation}) we need the following

\begin{lemma}
\label{Sproperties.sigma} Let $\mu \in (-2,-1).$ 
There exist some constants $c_k >0$ and $\e_k >0$, such that
\begin{equation}
\label{estimateS0} 
\|S (\varphi,0) \|_{{\mathcal C}^{2,\a}_\mu 
(I_\s  \times [v_\e , +\infty))} \leqslant c_k \e^{3/2+\mu/2}
\end{equation}
and, for all $\e \in (0, \e_k)$
$$ \|S (\varphi,g_2)-S (\varphi,g_1)\|_{{\mathcal C}^{2,\a}_\mu
(I_\s  \times [v_\e , +\infty)
)} 
\leqslant  \frac{1}{2} \,
\|g_2-g_1 \|_
{{\mathcal C}^{2,\a}_\mu(I_\s  \times [v_\e , +\infty)
)},$$
\begin{equation}
\label{Terza.proprieta.S}
 \|S (\varphi_2,g)-S (\varphi_1,g)\|_{{\mathcal C}^{2,\a}_\mu
(I_\s  \times [v_\e , +\infty))} 
\leqslant  c\e^{\frac{1}{2}+\mu/2} \,
\|\varphi_2-\varphi_1\|_{{\mathcal C}^{2,\a}}
\end{equation}
 for all $g,g_1,g_2 \in
{\mathcal C}^{2,\a}_{\mu}(I_\s \times [v_\e , +\infty))$ such that
$\| g_i \|_{{\mathcal C}^{2,\a}_\mu(I_\s \times [v_\e , +\infty))} 
\leqslant  \, 2c_k \e^{3/2+\mu/2}$ 
for all boundary data 
$\varphi,\varphi_1,\varphi_2 \in {\mathcal C}^{2, \alpha} (S^1)$ satisfying 
\eqref{stima.varphi.KMR}.
\end{lemma}

\noindent{\bf Proof.} We know from   proposition \ref{inverse}
that $\| G_{\e, v_\e}(f) \|_{{\mathcal C}^{2,\a}_\mu} \leqslant c
\|f\|_{{\mathcal C}^{0,\a}_\mu},$ then
$$\| S(\varphi,0)\|_{{\mathcal C}^{2,\a}_\mu(I_\s  \times [v_\e , +\infty)
)}\leqslant 
 c \|  \tilde L_\e (w_\varphi ) -  {\cal L}_\s \, w_{\varphi}  +
  \tilde Q_\s \left( w_\varphi \right)  \|_{{\mathcal C}^{0,\a}_\mu(I_\s  \times [v_\e , +\infty)
)}\leqslant$$

$$\leqslant  c \left( \|  \tilde L_\e (w_\varphi )\|
_{{\mathcal C}^{0,\a}_\mu(I_\s  \times [v_\e , \infty)
)} +
\| {\cal L}_\s \, w_{\varphi}  \|_{{\mathcal C}^{0,\a}
_\mu(I_\s  \times [v_\e , \infty)
)}+
\|  \tilde Q_\s \left( w_\varphi \right)  \|_{{\mathcal
C}^{0,\a}_\mu(I_\s  \times [v_\e , \infty)
)}\right).$$

\noindent So we need to find the estimates for the three terms above.

\noindent We recall that $|\varphi |_{2,\a}\leqslant k \e$. For all 
$\mu \in (-2,-1),$ 
thanks to  proposition \ref{otheroperator} we
know that
\begin{equation}
\label{estimateyy+1} 
|w_\varphi|_{2,\a;[v,v+1]} \leqslant e^{ -2(v-v_\e)}|\varphi |_{2,\a}
\end{equation}

\noindent We know
that
$$\| w_\varphi \|_{{\mathcal C}^{2,\a}_\mu}= 
\sup_{v \in [v_\e, +\infty)}e^{- \mu v}
| w_\varphi |  _{2, \a ;[v,v+1] }\leqslant    \sup_{v \in [v_\e, \infty)}e^{-
\mu v} e^{ -2 (v-v_\e)} |\varphi |_{2,\a}\leqslant$$
 $$\leqslant e^{-\mu v_\e}|\varphi |_{2,\a} \leqslant c_k \e^{1+\mu/2}.$$
From this inequality and  from the estimates of the coefficients
of $\tilde  L_\e$,  it follows that
$$\| \tilde L_\e (w_\varphi ) \| _{{\mathcal C}^{0,\a}_\mu} \leqslant
 c  \e^{1/2} \| w_\varphi \|_{{\mathcal C}^{0,\a}_\mu}\leqslant c_k 
 \e^{3/2+\mu/2}.$$
As for ${\cal L}_\s ,$  being $w_\varphi$ an harmonic function we obtain
following equality 
$$ {\cal L}_\s  \, w_\varphi = 2k \, w_\varphi ,$$ where 
$k(u,v)=\sin^2\s\, \cos^2 (x(u))+ \cos^2\s\, \sin^2 (y(v))\leqslant c\e, $
because we have set $\a+\b+\sigma \leqslant \e$ (see lemma \ref{lemX_3})
 and $y_\e \leqslant y(v) \leqslant \pi,$ with $y_\e=\pi-a_\e,$
 where $a_\e={\mathcal O}(\sqrt \e).$
Therefore, we conclude that
$$ \|{\cal L}_\s \, w_\varphi \| _{{\mathcal C}^{0,\a}_\mu 
(I_\s  \times [v_\e , \infty))} 
\leqslant 2c\e\| w_\varphi \|_{{\mathcal C}^{0,\a}_\mu 
(I_\s  \times [v_\e , \infty))}\leqslant
  c_k  \e^{2+\mu/2}.$$
The last term is  estimated by
$$ \| \tilde Q_\s \left( w _\varphi \right) \|_{{\mathcal C}^{0,\a}_\mu 
(I_\s  \times [v_\e , \infty))}
\leqslant c_k  \e^{2+\mu/2}.$$
Putting together these estimates we get the first result. 
The details of other estimates are left to the reader.
\hfill \qed
\begin{theorem}
\label{fixedpoint} Let be $B:=\{g \in {\cal C}^{2,\a}_{\mu}(I_\s \times [v_\e,+\infty))\, ;\,
||g||_{{\cal C}^{2,\a}_{\mu}} \leqslant 2c_k \e^{3/2+\mu/2}\}.$ Then the nonlinear
mapping $S$ defined above has a unique fixed point $g$ in $B.$
\end{theorem}

\noindent{\bf Proof.} The previous lemma shows that, if $\e$ is
chosen small enough, the nonlinear mapping  $S$ is a contraction
mapping  from the ball B
of radius $2 c_k \e^{3/2+\mu/2}$ in ${\mathcal C}^{2,\a}_{\mu} (I_\s 
\times [v_\e , V_\s])$ into itself. This value comes from
the estimate of the norm of $S(\varphi,0).$ Consequently thanks to the
Schauder theorem, $S$ has a unique fixed point $v$ in this ball.
\hfill \qed \\

%

\noindent
This argument provides a minimal surface which is close to  $M_1$ 
(respectively $M_2$)
 and has one boundary. This surface, 
 denoted by $S_{t,\g,\xi_1,d_t}(\varphi),$  is, close to its
boundary, a vertical graph over the annulus $B_{2r_\e} -
B_{r_\e}$ whose parametrization is given, for $\b=0$ by
\begin{equation}
\label{espansioni.KMR.bzero}
\bar U_{t,1} (r, \theta ) = -(1+\g)\ln(2r)+r\kappa_1\cos \t-
\frac{1+\g }{r}\xi_1\cos\t +d_t+\bar {\cal H}_{v_\e,\varphi}
 (\ln 2r-v_\e, \theta) + V_t(r,\theta ).
\end{equation}
and, for $\a=0,$ by
\begin{equation}
\label{espansioni.KMR.azero}
\bar U_{t,2} (r, \theta ) = -(1+\g)\ln(2r)-r\kappa_2\sin\t-
\frac{1+\g }{r} \xi_2\sin\t +d_t+\bar{\cal H}_{v_\e,\varphi} 
(\ln 2r-v_\e, \theta) + V_t(r,\theta ).
\end{equation}
where $  v_\e  = - \frac 1 2 \, \ln \e +{\mathcal O}(1).$
The boundary of the surface corresponds to $r_\e = \frac{1}{2 \sqrt \e}.$
The function $V_t$ depends non linearly on $\e,\phi.$
It satisfies $ \|V_t(\e,\phi_i)(r_\e \cdot,\cdot) \|
_{{\mathcal C}^{2,\a}  (\bar B_2- B_{1})}  
\leqslant c \e$ and 
\begin{equation}
\label{stima.contrazione.KMR}
 \|\bar V_t(\e,\phi)(r_\e \cdot,\cdot)-\bar V_t(\e,\phi')(r_\e \cdot,\cdot)
  \|_{{\mathcal C}^{2,\a} 
(\bar B_2- B_{1})} 
\leqslant c \e^{1/2} \| \phi - \phi' \|_{{\mathcal C}^{2,\a}}.
\end{equation}
This estimate follows from
$$ \|\bar V_t(\e,\phi)(r_\e \cdot,\cdot)-\bar V_t(\e,\phi')(r_\e \cdot,\cdot) \|
_{{\mathcal C}^{2,\a} 
(\bar B_2- B_{1})}\leqslant e^{\mu v_\e}
\| S(\phi,\bar V_t)-S(\phi',\bar V_t)\|_{{\mathcal C}^{2,\a}_{\mu}
 (I_\sigma \times [v_\e,+\infty) )} $$
 and the estimate  \eqref{Terza.proprieta.S}.

%% file: gluing.tex
\section{The matching of Cauchy data}
\label{gluing}
In this section we shall complete the proof of theorems
\ref{th1}, \ref{th2} and \ref{th3}.\\

\subsection{The proof of theorem \ref{th2}}
\label{proof.theorem1}
The proof of theorem \ref{th2} is articulated in two distinct
parts: the proof of the existence of the family ${\cal K}_1$ and
of the existence of the family ${\cal K}_2.$\\

\noindent
We start proving the existence of the family ${\cal K}_2.$ 
The proof is based on an analytical gluing procedure. 
The surfaces in family ${\cal K}_2$ are symmetric
with respect to the vertical plane $\{x_2=0\}.$
So all of the surfaces involved in the following 
proof must enjoy the same mirror symmetry property.
We will show how to glue  a  compact
piece of a Costa-Hoffman-Meeks type surface with bent catenoidal 
end to two halves of the KMR example $\widetilde M_{\s,\a,0}$ along
 the upper and lower boundaries and to a horizontal periodic flat annulus 
 with a removed disk along the   middle boundary.
 All of the surfaces just mentioned have the desired property 
 of symmetry as well as the surfaces obtained 
 by them by  a slight perturbation. We will recall below the 
 necessary results we proved in previous sections. \\

\noindent
Using the result of section \ref{family.M.abt},  we can obtain 
a minimal surface that is a perturbation 
of half the KMR example $\widetilde M_{\s,\a,0}$ and is asymptotic to it. 
This surface,
denoted by $S_{t,\lambda_t,\xi_t,d_t}(\varphi_t), $
can be parameterized over the annular neighbourhood 
$B_{2r_\e} -B_{r_\e},$ 
of its boundary as the vertical graph of
(see  \eqref{espansioni.KMR.bzero})
$$\bar U_{t}(r,\t)=-
(1+\lambda_t) \ln (2r)+r\kappa_{t}\cos \t 
-\frac{(1+\lambda_t)}{r} \xi_{t}\cos \t + d_t +
\bar{\mathcal H}_{v_\e,\varphi_t} (\ln 2r-v_\e,\theta ) + 
{\cal O}_{C^{2,\a}_b}(\e).
$$
This surface will be glued to the Costa-Hoffman-Meeks example
along its upper boundary. 
The surface that will be glued along the lower boundary,
denoted by $S_{b,\lambda_b,\xi_b,d_b}(\varphi_b), $
can be parameterized in the annulus $B_{2r_\e} -B_{r_\e},$ 
as the vertical graph of
$$\bar U_b(r, \theta)=(1+\lambda_b)\ln (2r) +r\kappa_b \cos \t
-\frac{(1+\lambda_b)}{r} {\xi_b} \cos \t+d_b +\bar{\mathcal H}_{v_\e,\varphi_b} 
(\ln 2r-v_\e,\theta ) + {\mathcal O}_{C^{2,\a}_b}(\e).
$$
\noindent 
We remark that the functions $\varphi_t,\varphi_b \in C^{2,\a}(S^1)$
are even and orthogonal in the $L^2$-sense to the constant function
and to $\t \to \cos \t.$\\

\noindent
Using the result of section \ref{strip},  we can construct a
minimal graph $S_{m}(\varphi_m)$ 
close to a horizontal periodic flat annulus.
It can be parameterized, in the neighbourhood $B_{2r_\e} -B_{r_\e}$ 
of its boundary, as the vertical graph of
$$\bar U_m(r, \theta)=
\tilde H_{r_\e,\varphi_m}(r,\t) +{\mathcal O}_{C^{2,\a}_b}(\e),
$$
where $\tilde H_{r_\e,\varphi_m}$ is the harmonic extension of 
$\varphi_m \in C^{2,\a}(S^1)$ which is an even function orthogonal
in the $L^2$-sense to the constant function. \\

\noindent
Thanks to the result of   section \ref{family.costa.xi}, we can
construct a minimal surface $\bar M_{k,\e}^T(\e/2,\Psi),$ with
 $\Psi=(\psi_t,\psi_b,\psi_m),$ which is close to   a truncated 
  genus $k$ Costa-Hoffman-Meeks surface and
has three boundaries. 
The functions $\psi_t,\psi_b,\psi_m \in C^{2,\a}(S^1)$ 
are even. Moreover $\psi_t,\psi_b$ are $L^2$-orthogonal
to the constant function and to $\t \to \cos \t$
and $\psi_m$ is orthogonal to the constant function.\\

\noindent  
The surface $\bar M_{k,\e}^T(\e/2,\Psi)$ is, close to its
upper and lower boundary, a vertical graph over the annulus 
 $B_{r_\e} -B_{r_\e/2},$ 
whose parametrization is, respectively, given by
$$
  U_{t} (r, \theta ) = \sigma_{t} - \ln (2r) -\frac{\e} 2 r\cos \t +
H_{\psi_t} (s_\e-\ln 2r, \theta) + {\mathcal O}_{C^{2,\a}_b}(\e),
$$
$$
  U_b (r, \theta ) = - \sigma_{{b} } + \ln (2r) -\frac{\e} 2 r\cos \t +
H_{\psi_b} (s_\e-\ln 2r, \theta) + {\mathcal O}_{C^{2,\a}_b}(\e) ,
$$
where $  s_\e  = - \frac{1}{2}\, \ln \e .$
Nearby the middle boundary the surface  is a vertical graph 
whose parametrization is 
$$
  U_m (r, \theta ) = \tilde H_{\rho_\e,\psi_m} 
\left(\frac{1}{r} ,\theta \right) +
 {\mathcal O}_{C^{2,\a}_b}(\e) .
$$
\noindent
The functions ${\mathcal O}_{C^{2,\a}_b}(\e)$ replace the functions 
$V_t, V_b, V_m,\bar  V_t, \bar V_b, \bar V_m$ that appear at the end of sections
 \ref{strip}, 
 \ref{family.costa.xi} and \ref{family.M.abt}. They
depend nonlinearly on the different parameters and boundary data functions
but they are bounded by a constant 
 times $\e$ 
in $C^{2,\a}_b$ topology, where partial derivatives are taken with respect 
to the vector fields $r\partial_r$ and $\partial_\t.$
We assume that the parameters and the boundary functions are chosen
so that 
\begin{equation}
\label{assumption.gluing.1} 
 |\lambda_t |+|\lambda_b|+|-\lambda_t \ln\left({\e^{-1/2}}\right)
+\eta_t |+
|  \lambda_b \ln\left({\e^{-1/2}} \right)+\eta_b|+
\e^{-1/2}/2 |\kappa_t+\e/2|+ \e^{-1/2}/2|\kappa_b+\e/2|+
\end{equation}
$$2\e^{1/2} (|(1+\lambda_t)\xi_t|+|(1+\lambda_b)\xi_b| )+ 
\|\varphi_t \|_{C^{2,\a}(S^1)} +
\|\varphi_b \|_{C^{2,\a}(S^1)}+ 
\|\varphi_m\|_{C^{2,\a}(S^1)}+$$
$$+\|\psi_t \|_{C^{2,\a}(S^1)}+
\|\psi_b \|_{C^{2,\a}(S^1)}+
\|\psi_m\|_{C^{2,\a}(S^1)}\leqslant  k \e,$$
where $\eta_t=d_t- \sigma_{t},$ $\eta_b=d_b+ \sigma_{b}.$
where the constant $k> 0$ is fixed large enough. 
It remains to show that, for all $\e$ small enough, it is possible to choose
the parameters and boundary functions in such a way that the surface
$$
S_{t,\lambda_t,\xi_t,d_t}(\varphi_t)\cup S_{b,\lambda_b,\xi_b,d_b}(\varphi_b) 
\cup S_{m}(\varphi_m) \cup   M^T_{k}(\e/2,\Psi)$$
is a $C^1$ surface across the boundaries of the different summands. Regularity
theory will then ensure that this surface is in fact smooth and by construction
it has the desired properties. This will therefore complete the proof of
the existence of the family of examples denoted by ${\cal K}_1$.\\

\noindent
It is necessary to fulfill the following system of equations
$$ \left\{
\begin{array}{c}
U_t(r_\e,\cdot) = \bar U_t(r_\e,\cdot)\\
U_b(r_\e,\cdot) = \bar U_b(r_\e,\cdot)\\
U_m(r_\e,\cdot) = \bar U_m(r_\e,\cdot)\\
\partial_r U_b(r_\e,\cdot) =\partial_r \bar U_b(r_\e,\cdot)\\
\partial_r U_t(r_\e,\cdot) =\partial_r \bar U_t(r_\e,\cdot)\\
\partial_r U_m (r_\e,\cdot)=\partial_r \bar U_m(r_\e,\cdot)\\
\end{array}
\right.
$$
on $S^1.$ The first three equations lead to the system
\begin{equation}
\label{Primo.sistema.xi}
\left\{
\begin{array}{c}
 -\lambda_t \ln\left({2r_\e}\right)+\eta_t-
 \left((1+\lambda_t) \frac{\xi_t}{r_\e} \right)  \cos \t+
 r_\e(\kappa_t+\frac \e 2 )\cos \t+
\varphi_t- \psi_t= {\mathcal O}_{C^{2,\a}_b}(\e)\\
\lambda_b \ln\left({2r_\e}\right)+\eta_b-
\left( (1+\lambda_b)\frac{\xi_b}{r_\e} \right)  \cos \t+
r_\e(\kappa_b+\frac \e 2 )\cos \t+
\varphi_b-\psi_b=
{\mathcal O}_{C^{2,\a}_b}(\e)\\
\varphi_m-\psi_m=
{\mathcal O}_{C^{2,\a}_b}(\e).\\
\end{array}
\right.
\end{equation}
The last three equations give the system
\begin{equation}
\label{Secondo.sistema.xi}
\left\{
\begin{array}{c}
 -\lambda_t  +\left( (1+\lambda_t)\frac{\xi_t}{r _\e} \right)  \cos \t
 +r_\e(\kappa_t+\frac \e 2) \cos \t-
\partial_\t(\varphi_t+\psi_t)=
{\mathcal O}_{C^{1,\a}_b}(\e)\\
\lambda_b +\left( (1+\lambda_b)\frac{\xi_b}{r _\e} \right)  \cos \t
+r_\e(\kappa_b+\frac \e 2) \cos \t-
\partial_\t(\varphi_b+\psi_b)=
{\mathcal O}_{C^{1,\a}_b}(\e)\\
 \partial_\t(\varphi_m+\psi_m)=
{\mathcal O}_{C^{1,\a}_b}(\e).\\
\end{array}
\right.
\end{equation}
\noindent
To obtain this system we applied the results of lemma
\ref{proprieta.funzione.armonica.1} and
 \ref{proprieta.funzione.armonica.2}.
Here, the functions ${\mathcal O}_{C^{l,\a}}(\e)$ in the above
 expansions depend nonlinearly on the different parameters and 
 boundary data functions but they are bounded by a constant 
 times $\e$ in $C^{l,\a}$ topology.
The projection of the first two equations of each system over the 
$L^2$-orthogonal complement of ${\rm Span}\{1, \cos \t \}$
and the remaining two equations gives the system
\begin{equation}
\label{sistema.proiettato.xi}
\left\{
\begin{array}{c}
\varphi_t- \psi_t={\mathcal O}_{C^{2,\a}_b}(\e)\\
\varphi_b- \psi_b={\mathcal O}_{C^{2,\a}_b}(\e)\\
\varphi_m- \psi_m={\mathcal O}_{C^{2,\a}_b}(\e)\\
\partial_\t\varphi_t+ \partial_\t\psi_t={\mathcal O}_{C^{1,\a}_b}(\e)\\
\partial_\t\varphi_b+\partial_\t\psi_b={\mathcal O}_{C^{1,\a}_b}(\e)\\
\partial_\t\varphi_m+ \partial_\t\psi_m={\mathcal O}_{C^{1,\a}_b}(\e).\\
\end{array}
\right.
\end{equation}

\begin{lemma}
The operator $h$ defined by
$$\begin{array}{c}
C^{2,\a}(S^1) \to C^{1,\a}(S^1)\\
 \varphi \to \partial_\t \varphi
\end{array}
$$
acting on functions that are orthogonal to the constant function
in the $L^2$-sense and are even, is invertible.
\end{lemma}
\noindent
{\bf Proof (\cite{FP}).} We observe that if we decompose
$\varphi =\sum_{j\geqslant 1} \varphi_j \cos( j\t),$
then
$$h(\varphi) =- \sum_{j\geqslant 1} j\varphi_j \sin( j\t),$$
 is clearly invertible from $H^1(S^1)$ into $L^2(S^1).$ 
Now elliptic regularity theory implies that this is also the case 
when this operator is defined between H\"older spaces. \hfill \qed\\

\noindent
Using this result, the system \eqref{sistema.proiettato.xi}
can be rewritten as 
\begin{equation}
\label{sistema.compatto.xi}
(\varphi_t,\varphi_b,\varphi_m,\psi_t,\psi_b, 
\psi_m) ={\mathcal O}_{C^{2,\a}} (\e).
\end{equation}
Recall that the right hand side depends nonlinearly on 
$\varphi_t,\varphi_b,\varphi_m,\psi_t,\psi_b, 
\psi_m$ and also
on the parameters $\lambda_t,\lambda_b,\eta_t,\eta_b, \xi_t,\xi_b,
\kappa_t,\kappa_b.$
We look at this equation as a fixed point problem and fix $k$ large enough. 
Thanks to estimates \eqref{stima.contrazione.KMR}, 
 \eqref{stima.contrazione.strip},\eqref{stima.contrazione.scherk},
\eqref{stima.contrazione.costa.xi.1} and \eqref{stima.contrazione.costa.xi.2}, 
we can use a fixed point theorem for contracting
mappings in the ball of radius $k\e$
in $(C^{2,\a}(S^1))^6$ to obtain, for all $\e$ small enough, a solution 
$(\varphi_t,\varphi_b,\varphi_m, \psi_t, \psi_b, \psi_m)$
of \eqref{sistema.compatto.xi}. This solution being obtained a fixed 
point for contraction mapping and the right hand side of \eqref{sistema.compatto.xi} 
being continuous with respect to all data, we
see that this fixed point $(\varphi_t,\varphi_b,\varphi_m, \psi_t, \psi_b, \psi_m)$
depends  continuously (and in fact smoothly) on the parameters $\lambda_t,\lambda_b,\eta_t,\eta_b, \xi_t,\xi_b,\kappa_t,\kappa_b.$
Inserting the founded solution  into \eqref{Primo.sistema.xi} and 
\eqref{Secondo.sistema.xi}, we see that it
remains to solve a system of the form
\begin{equation}
\label{sistema.finale.xi}
\left\{
\begin{array}{c}
 -\lambda_t \ln({2r_\e})+\eta_t+\left(- (1+\lambda_t)\frac{\xi_t}{r_\e} +
 r_\e(\kappa_t+\frac \e 2) \right)  \cos \t= 
{\mathcal O} (\e),\\
\lambda_b \ln({2r_\e})+\eta_b+\left( -(1+\lambda_b)\frac{\xi_b}{r_\e}
+r_\e(\kappa_b+\frac \e 2) 
 \right)  \cos \t=
{\mathcal O} (\e),\\
- \lambda_t  +\left( (1+\lambda_t) \frac{\xi_t}{r _\e} 
 +r_\e(\kappa_t+\frac \e 2)\right)\cos \t=
{\mathcal O} (\e),\\
 \lambda_b  +\left((1+\lambda_b) \frac{\xi_b}{r _\e} 
+r_\e(\kappa_b+\frac \e 2)\right)\cos \t=
{\mathcal O} (\e).\\
\end{array}
\right.
\end{equation}
where this time, the right hand sides only depend nonlinearly on 
$\lambda_t,\lambda_b,\eta_t,\eta_b, \xi_t,\xi_b,\kappa_t,\kappa_b.$
There are  eight equations that are obtained by projecting this system
over the constant function and the function $\t \to \cos \t.$ 
If we set 
$$(  \bar \eta_t, \bar  \eta_b) =(  -\lambda_t \ln({2r_\e})+\eta_t ,
  \lambda_b \ln({2r_\e})+\eta_b),$$
$$  (  \bar \xi_t, \bar  \xi_b) = r_\e^{-1} ((1+\lambda_t)\xi_t, 
(1+\lambda_b)\xi_b ) ,\quad
 (  \bar \kappa_t, \bar  \kappa_b) = r_\e (\kappa_t, \kappa_b ) ,$$
\noindent
the previous system can
be rewritten as
\begin{equation}
\label{sistema.coefficienti.compatto.xi}
(\lambda_t,\lambda_b,\bar\xi_t, \bar\xi_b,
\bar \eta_t, \bar \eta_b, \bar \kappa_t, \bar  \kappa_b )={\mathcal O}  (\e).
\end{equation}
This time, provided $k$ has been fixed large enough, we can use Leray-Sch\"auder
fixed point theorem in the ball of radius $k \e$ in $\R^{8}$ to solve 
\eqref{sistema.coefficienti.compatto.xi}, for all $\e$ small enough. 
This provides a set of parameters and a set 
of boundary data such that \eqref{Primo.sistema.xi} and \eqref{Secondo.sistema.xi}
hold. Equivalently we have proven the existence of a solution of systems 
\eqref{Primo.sistema.xi}  and \eqref{Secondo.sistema.xi}. So  the proof of the 
first part of theorem \ref{th2} is complete.\\

\noindent
The proof of the second part of  theorem \ref{th2} 
uses the same arguments seen above. So we will omit  most of the details.
Our aim is showing the existence of the 
 family of surfaces denoted by ${\cal K}_1.$
 The surfaces in the family ${\cal K}_1$ are symmetric 
 with respect to the vertical plane $\{x_1=0\}.$
It is important to observe that in the proof, 
the KMR example which we deal with, is obtained by 
 slight perturbation  of $\widetilde M_{\s,0,\b}.$
The symmetry properties of this surface differ 
by the ones of the surface close to $\widetilde M_{\s,\a,0}$ involved 
in the gluing procedure described above.
In particular $\widetilde M_{\s,0,\b}$ is 
symmetric with respect to the vertical plane 
$\{x_1=0\}.$ 
We recall that the Costa-Hoffman-Meeks type surface involved
in the first gluing procedure enjoys a mirror symmetry with respect
to the vertical plane $\{x_2=0\}.$ Then it is not appropriate for the gluing 
with a KMR example of the type described above.
To obtain the Costa-Hoffman-Meeks type surface we need,
it is sufficient to rotate counterclockwise by $\pi/2$ with 
respect the vertical axis $x_3$ the  Costa-Hoffman-Meeks
surface with bent catenoidal ends described in
section \ref{family.costa.xi}.
 The surface obtained in this way enjoys the mirror symmetry
 with respect to the vertical plane $\{x_1=0\}.$ 
In the parametrizations of the top and bottom ends 
 the cosinus function  is replaced by the sinus 
function, that is:

$$
  U_{t} (r, \theta ) = \sigma_{t} -\ln (2r) -\frac{\e} 2 r\sin \t +
H_{\psi_t} (s_\e-\ln (2r), \theta) + {\mathcal O}_{C^{2,\a}_b}(\e),
$$
$$
  U_b (r, \theta ) = - \sigma_{{b} } + \ln (2r) -\frac{\e} 2 r\sin \t +
H_{\psi_b} (s_\e-\ln(2r) , \theta) + {\mathcal O}_{C^{2,\a}_b}(\e) ,
$$
where $  s_\e  = - \frac{1}{2}\, \ln \e $ and 
$(r,\t)\in B_{r_\e}-B_{r_\e/2}.$
As for the planar middle end, the form of its parametrization
remains unchanged. We refer to the first part of the proof 
for its expression.
Another important remark to do, concerns
the properties of the Dirichlet boundary data $\psi_t,\psi_b,\psi_m.$
If before it was requested that they were even functions and orthogonal
(only $\psi_t,\psi_b$) to the constant function and to $\t \to \cos \t$ 
to preserve the mirror 
symmetry property with respect to the vertical plane $\{x_1=0\},$ now as 
consequence of the above observations they must be odd functions and 
$\psi_t,\psi_b$ must be orthogonal to the constant function and 
to $\t \to \sin \t.$ It is clear that all the results showed in section \ref{family.costa.xi} continue to hold. \\



\noindent
Now we give the expressions of the parametrizations of the surface $S_{t,\lambda_t,\xi_t,d_t}(\varphi_t), $
the minimal surface obtained by perturbation from  
 the KMR example
$\widetilde M_{\s,0,\b}$ and  asymptotic to it. 
This surface  can be parameterized in the neighbourhood 
$B_{2r_\e} -B_{r_\e},$ as the vertical graph of
$$\bar U_{t}=
-(1+\lambda_t) \ln (2r)+r\kappa_{t}\sin \t 
-\frac{(1+\lambda_t)}{r} \xi_{t}\sin \t + d_t +
\bar{\mathcal H}_{v_\e,\varphi_t} (\ln 2r-v_\e,\theta ) + 
{\cal O}_{C^{2,\a}_b}(\e).
$$

\noindent
The parametrization of the surface that we will glue to the 
Costa-Hoffman-Meeks type surface along its lower boundary 
 is given by 
$$\bar U_b(r, \theta)=
(1+\lambda_b)\ln (2r) 
-{\xi_b}\frac{(1+\lambda_b)}{r} \sin \t+d_b +
\bar {\mathcal H}_{v_\e,\varphi_b} (\ln 2r-v_\e,\theta ) +
{\mathcal O}_{C^{2,\a}_b}(\e),
$$
where $(r,\t)\in B_{2r_\e} -B_{r_\e}.$ \\

\noindent
To prove the theorem it is necessary to show the existence of a 
solution of the following 
system of equations
$$ 
\left\{
\begin{array}{c}
U_t(r_\e,\cdot) = \bar U_t(r_\e,\cdot)\\
U_b(r_\e,\cdot) = \bar U_b(r_\e,\cdot)\\
U_m(r_\e,\cdot) = \bar U_m(r_\e,\cdot)\\
\partial_r U_b(r_\e,\cdot) =\partial_r \bar U_b(r_\e,\cdot)\\
\partial_r U_t(r_\e,\cdot) =\partial_r \bar U_t(r_\e,\cdot)\\
\partial_r U_m(r_\e,\cdot) =\partial_r \bar U_m(r_\e,\cdot)\\
\end{array}
\right.
$$
on $S^1,$ under the  assumption \eqref{assumption.gluing.1}
for the parameters and the boundary functions.  
It is clear that the proof of the existence of a solution
of this system is based on the same arguments seen before. 
We  remark that the role played before by the functions
$\t \to \cos (i \t),$ now is played by the  functions 
$\t \to \sin (i \t).$ 
That completes the proof 
of theorem \ref{th2}.

\subsection{The proof of theorem \ref{th1}}
\label{proof.theorem.2Scherk+CHM}
To proof the theorem \ref{th1}
we will glue a compact piece of the surface $M^T_k(\xi),$ with $\xi=0,$
described in section \ref{family.costa.xi} to
two halves  of a   Scherk type surface along
the upper and lower boundary and to a  horizontal periodic flat 
annulus along the middle boundary. The construction 
 of these surfaces is showed in section \ref{strip}. 
In particular we showed the existence of a minimal graph
close to half a Scherk type example whose ends have asymptotic
directions given by $\cos \t_1 e_1 + \sin \t_1 e_3$
and  $-\cos \t_2 e_1 + \sin \t_2 e_3.$ 
These surfaces
in the neighbourhood $B_{2r_\e}-B_{r_\e}$ 
of the boundary,  admit the following parametrization 
$$\bar U_t=d_t-\ln (2r) +
 \tilde H_{r_\e,\varphi_t}(r,\t)+{\cal O}_{C^{2,\a}_b}(\e),$$ 
 $$\bar U_b=d_b+ \ln (2r) +
 \tilde H_{r_\e,\varphi_b}(r,\t)+{\cal O}_{C^{2,\a}_b}(\e),$$ 
where the Dirichlet boundary data $\varphi_i \in  C^{2,\a}(S^1),$
 for $i=t,b,$ are requested to be even and orthogonal to 
 the constant function,
$w_{\varphi_i}$ denotes their harmonic extensions.
 The other surfaces involved in the gluing procedure 
 have been described in the previous subsection.\\

\noindent 
The proof is similar to the one 
given for theorem \ref{KMR+CHM}, so we will give only the essential
details.  
Actually  to prove the theorem it is necessary to show the existence of a 
solution of the following 
system of equations
$$ \left\{
\begin{array}{c}
U_t(r_\e,\cdot) = \bar U_t(r_\e,\cdot)\\
U_b(r_\e,\cdot) = \bar U_b(r_\e,\cdot)\\
U_m(r_\e,\cdot) = \bar U_m(r_\e,\cdot)\\
\partial_r U_b(r_\e,\cdot) =\partial_r \bar U_b(r_\e,\cdot)\\
\partial_r U_t(r_\e,\cdot) =\partial_r \bar U_t(r_\e,\cdot)\\
\partial_r U_m(r_\e,\cdot) =\partial_r \bar U_m(r_\e,\cdot)\\
\end{array} \right.
$$
on $S^1,$ under a  assumption similar to \eqref{assumption.gluing.1}.
We refer to subsection \ref{proof.theorem1} for the expressions
of $U_t,U_b,U_m,\bar U_m.$ It is necessary to point out
that in this subsection we consider the more symmetric example ($\xi=0$)
in the family $(M^T_k(\xi))_\xi$. So it is necessary to replace
$\e/2$ by $0$ in 
the expressions of the functions $U_t$ and $U_b$ of the 
top and bottom ends.\\

\noindent
We want to remark  the boundary data for the surfaces 
we are going to glue together do not satisfy the same hypotheses
of orthogonality. All of these functions are orthogonal to the
constant function, but only $\psi_t,\psi_b$  are orthogonal
to $\t \to \cos \t$ too.
 The functions denoted
by ${\mathcal O}_{{\mathcal C}_b^{2,\a}}(\e)$ that appear in the 
expressions of $\bar U_i$ and $U_i,$ with $i=t,b,m,$
have a Fourier series decomposition containing
a term collinear to $\cos \t$ only if the corresponding
boundary data is assumed to be orthogonal only to the constant
function. Furthermore the fact that $\xi=0$ (the catenoidal ends
are not bent) implies that functions which parametrize the top and 
bottom end of $M_k^T(0)$ are orthogonal to $\cos \t.$ 
In other terms, in difference with the Scherk type surfaces,
 we are not able to prescribe the coefficients
in front of the eigenfunction $\cos \t$ for the catenoidal ends 
of $M_k^T(0),$  because in this more symmetric setting are 
obliged to vanish. \\

\noindent
 The first three equations lead to the system
\begin{equation} 
\label{Primo.sistema.xi3}
\left\{
\begin{array}{c}
 \eta_t + \varphi_t- \psi_t= 
 {\mathcal   O}_{C^{2,\a}_b}(\e)\\
\eta_b  +\varphi_b-\psi_b=
{\mathcal O}_{C^{2,\a}_b}(\e)\\
\varphi_m-\psi_m={\mathcal O}_{C^{2,\a}_b}(\e),\\
\end{array}
\right.
\end{equation}
where $\eta_t=d_t- \sigma_{t},$ 
$\eta_b=d_b+ \sigma_{b}.$
The last three equations give the system
\begin{equation}
\label{Secondo.sistema.xi3}
\left\{
\begin{array}{c}
\partial_\t(\varphi_t+\psi_t)={\mathcal O}_{C^{1,\a}_b}(\e)\\
\partial_\t(\varphi_b+\psi_b)={\mathcal O}_{C^{1,\a}_b}(\e)\\
 \partial_\t(\varphi_m+\psi_m)={\mathcal O}_{C^{1,\a}_b}(\e).\\
\end{array}
\right.
\end{equation}
Now to complete the proof it is sufficient to use 
the same arguments of subsection \ref{proof.theorem1}.

\subsection{The proof of theorem \ref{th3}} 
To prove this theorem it is necessary to treat separately
the case $k=0$ and $k \geqslant 1.$ \\

\noindent
{\bf The case k=0.} 
We will glue half a Scherk example with half 
a KMR example with $\a=\b=0.$  
We observe that this surface
is symmetric with respect  to the $\{x_1=0\}$ and $\{x_2=0\}$ planes.
The  Scherk example is symmetric with respect
to the $\{x_2=0\}$ plane.
 To preserve  this  property of symmetry in the surface
 obtained by the gluing procedure, we will consider  
the perturbation of $\tilde M_{\s,0,0}$ which enjoys the same
mirror symmetry. That is the surface denoted by 
$S_{t,\lambda_t,\xi_t,d_t}(\varphi)$ 
with $\lambda_t=\xi_t=0$ and $d_t=d.$
It can be parametrized in the annulus 
$B_{2r_\e}-B_{r_\e}$ as the vertical  graph  of
$$\bar U(r,\t)=
 -\ln (2r)+ \bar d 
+ \bar{\mathcal H}_{v_\e,\varphi} (\ln 2r-v_\e,\theta ) + 
{\cal O}_{C^{2,\a}_b}(\e).$$ 
The   Scherk example is parametrized as the vertical
graph of
$$ U (r,\t)=-\ln(2r) +d+\tilde H_{r_\e,\psi}(r,\t)+{\cal O}_{C^{2,\a}_b}(\e).$$
As for the Dirichlet boundary data, we assume
$\varphi$ to be an even function orthogonal to the constant 
function and to $\t \to \cos \t,$ and $\psi$ to be 
an even function orthogonal to the constant function.\\

\noindent
 To prove the theorem in the case $k=0,$ it is necessary 
 to show the existence of a solution of the following 
system of equations
$$ \left\{
\begin{array}{c}
U(r_\e,\cdot) = \bar U(r_\e,\cdot)\\
\partial_r U(r_\e,\cdot) =\partial_r \bar U(r_\e,\cdot)\\
\end{array} \right.
$$
on $S^1,$ under appropriate assumptions on the norms of the Dirichlet
boundary data and the parameters $\xi,$ $d,$ $\bar d.$
\noindent
 These equations lead to the system
\begin{equation} 
\label{Primo.sistema.KMRSCHERK}
\left\{
\begin{array}{c}
 \eta 
+  \varphi- \psi= {\mathcal O}_{C^{2,\a}_b}(\e)\\
  \partial_\t(\varphi+\psi)=
 {\mathcal O}_{C^{1,\a}_b}(\e).\\
\end{array}
\right.
\end{equation}
where $\eta= \bar d-d.$  
Now to complete the proof it is sufficient to 
use the same arguments of subsection \ref{proof.theorem1}.
The details are omitted.\\

\noindent
{\bf The case $k \geqslant 1.$}
The proof of the theorem in this case is similar 
the proof of theorem \ref{th1}. 
In fact three of the surfaces we are going to glue are 
a compact piece of the Costa-Hoffman-Meeks example
$M_k,$  half a   Scherk type example
and a horizontal periodic flat annulus 
as in the proof of theorem \ref{th1}. 
The fourth surface is half a KMR example, of the same type 
of  the proof of theorem for $k=0.$ 
The surfaces are parametrized as vertical graph
over $B_{2r_\e}-B_{r_\e}$ of the following functions:
$$\bar U_b (r,\t)=\ln(2r) +d_b +
\tilde H_{r_\e,\varphi_b}(r,\t)+{\cal O}_{C^{2,\a}_b}(\e),$$
for the  Scherk type example,
$$ \bar U_m (r,\t)=\tilde H_{r_\e,\varphi_m}(r,\t)+{\cal O}_{C^{2,\a}_b}(\e),$$
 for the horizontal periodic flat annulus,
 $$\bar U_t(r,\t)=-
 \ln (2r)+  d_t +
 \bar{\mathcal H}_{v_\e,\varphi_t} (\ln 2r-v_\e,\theta ) + 
{\cal O}_{C^{2,\a}_b}(\e),$$
for the KMR example,
 $$
  U_{t} (r, \theta ) = \sigma_{t} - \ln (2r)  +
H_{\psi_t} (s_\e-\ln (2r), \theta) + {\mathcal O}_{C^{2,\a}_b}(\e),
$$
$$
  U_b (r, \theta ) = - \sigma_{{b} } + \ln (2r) +
H_{\psi_b} (s_\e-\ln(2r) , \theta) + {\mathcal O}_{C^{2,\a}_b}(\e) ,
$$
 $$ U_m (r,\t)=\tilde H_{\rho_\e,\varphi_m}(\frac 1 r,\t)+
 {\cal O}_{C^{2,\a}_b}(\e)$$
 for the compact piece of the Costa-Hoffmam-Meeks example.
 The Dirichlet boundary data are requested to be even functions.
Functions $\psi_t,$ $\psi_b$ are orthogonal to the constant
function and to $\t \to \cos \t.$ Functions $\psi_m,$ $\varphi_t,$
$\varphi_b,$ $\varphi_m$ are orthogonal only to the constant function.
In this case the system of equations to solve is:
\begin{equation} 
\label{Primo.sistema.SCHERK+CHM+KMR}
\left\{
\begin{array}{c}
\eta_t 
+ \varphi_t- \psi_t= 
 {\mathcal O}_{C^{2,\a}_b}(\e)\\
\eta_b 
 + \varphi_b-\psi_b=
 {\mathcal O}_{C^{2,\a}_b}(\e)\\
\varphi_m-\psi_m=  {\mathcal O}_{C^{2,\a}_b}(\e)\\
\partial_\t(\varphi_t+\psi_t)
={\mathcal O}_{C^{1,\a}_b}(\e)\\
\partial_\t(\varphi_b+\psi_b)=    {\mathcal O}_{C^{1,\a}_b}(\e)\\
\partial_\t(\varphi_m+\psi_m)=    {\mathcal O}_{C^{1,\a}_b}(\e),\\
\end{array}
\right.
\end{equation}
where $\eta_t=d_t - \s_t,$ $\eta_b=d_b+\s_b.$ 
 The details are left to the reader.

%% file: appendixA.tex
\section{Appendix A}
\label{appendixa.xi}
\begin{definition}
Given $\ell \in {\mathbb N}$, $\alpha \in (0,1)$ and $\nu \in {\mathbb R}$, 
the space ${\mathcal C}^{\ell , \alpha}_{\nu} (B_{\rho_0}(0))$ is defined to be the space 
of functions in ${\mathcal C}^{\ell , \alpha}_{loc} (B_{\rho_0}(0))$ for which the 
following norm is finite

$$ \|\rho^{-\nu} \, w \|_{{\mathcal C}^{\ell , \alpha}  (B_{\rho_0}(0))}.$$
\end{definition}
\noindent
Now we can state the following result.

\begin{proposition}
\label{poisson.piano.interno.xi}
There exists an operator
$$ \tilde H : C^{2, \alpha}(S ^1) \longrightarrow
C^{2, \alpha}_{0} (S^1 \times 
 [\bar \rho,+\infty)),$$
such that for
each even function $\varphi(\t) \in  C^{2, \alpha}(S ^1 )$, which is
$L^2$-orthogonal to the constant function, 
 then $w_\varphi =  {\tilde H}_{\bar \rho,\varphi}$ solves
$$ \left\{ \begin{array}{lll}
\delta  w_\varphi  =0 & \hbox{ on }& S ^1 \times [\bar \rho,+\infty)  \\
w_\varphi =\varphi &\hbox{ on }& S ^1 \times \{\bar \rho \}.
\end{array}\right.$$
Moreover,
\begin{equation}
\label{estimate.piano.interno.xi}
 ||{\tilde H}_{\bar \rho,\varphi} ||_{C^{ 2,\alpha}_{-1}(S ^1 \times [\bar \rho,+\infty)) } 
\leqslant c  \,|| \varphi ||_{C^{2,\alpha}(S^1)},
\end{equation}
for some constant  $c>0$.
\end{proposition}
\begin{remark}
Following the arguments of the proof below, 
it is possible to state a similar proposition but with the hypothesis 
$\varphi$  odd. 
\end{remark}
\noindent
{\bf Proof.}
We consider the decomposition of the function $\varphi$ with
respect to the basis $\{\cos(i\t) \}$, that is
$$\varphi= \sum_{i=1}^{\infty}\varphi_i \cos(i\t).$$
Then the solution $w_\varphi$ is given by
$$w_\varphi(\rho,\t) =\sum_{i=1}^{\infty}\left( \frac{\bar \rho}{\rho}\right)^{i}
\varphi _i \cos(i\t).$$
\noindent
Since $ \frac{\bar \rho}{\rho} \leqslant 1,$ then  $ \left( \frac{\bar \rho}{\rho}\right)^{i}\leqslant \left( \frac{\bar \rho}{\rho}\right),$
\noindent
we can conclude that $  |w(r,\t)|\leqslant c  \rho^{-1} |\varphi(\t)| $ and then
$||w_{\varphi}||_{C^{ 2,\alpha}_{-1}}\leqslant c  ||\varphi||_{C^{ 2,\alpha}}. $
\hfill \qed

\noindent
Now we give the statement of an useful result whose proof is contained in \cite{FP}.
\begin{proposition}
\label{poisson.catenoide.troncone.xi} 
There exists an operator
$$
 H : {\mathcal C}^{2, \a}(S ^1) \longrightarrow {\mathcal
C}^{2, \a}_{-2} ([0,+ \infty) \times S^1),$$
such that for all $\varphi \in  {\mathcal C}^{2, \a}(S ^1 )$, even function 
and orthogonal to $1$ and $\cos \t,$ in the $L^2$-sense,
the function $w = H_{\varphi}$ solves
$$
\left\{ \begin{array}{rllllll} (\partial^{2}_{s} + \partial_\theta^{2} )
\, w
& =& 0 & \mbox{ in }& [0,+ \infty)  \times S^1  \\[3mm]
w & = & \varphi & \mbox{ on }& \{0 \} \times S^1
\end{array}\right.$$
Moreover
$$\| H_{\varphi}\|_{{\mathcal C}^{ 2, \a}_{ -2} } \leqslant c  \,
\| \varphi \|_{{\mathcal C}^{2, \a}},$$
for some constant  $c>0$.
\end{proposition}

\begin{proposition}
\label{otheroperator}
 There exists an
operator
$$\bar {\mathcal H}_{v_0} : C^{2, \alpha}(S ^1) \longrightarrow
C^{2, \alpha}_{\mu} (S^1 \times [v_0,+ \infty)),$$ $\mu \in(-2,-1),$ such that for
every function $\varphi(v) \in  C^{2, \alpha}(S ^1 )$, which is
$L^2$-orthogonal to $e_{0,i}(u)$ with $i=0,1$ and even,
 the function $w_\varphi =\bar {\mathcal H}_{v_0}(\varphi)$ solves
$$ \left\{ \begin{array}{lll}
\partial^2_{uu} w_\varphi +\partial^{2}_{vv} w_\varphi  =0 & \hbox{ on }& S ^1 \times
[v_0,+ \infty)  \\
w_\varphi =\varphi &\hbox{ on }& S ^1 \times \{v_0\}.
\end{array}\right.$$
Moreover,
\begin{equation}
\label{estimate.xi} ||\bar {\mathcal H}_{v_0}(\varphi) ||_{C^{ 2,\alpha}_{
\mu}(S ^1 \times [v_0,+\infty]) } \leqslant c  \,|| \varphi
||_{C^{2,\alpha}(S^1)},
\end{equation}
for some constant  $c>0$.
\end{proposition}
\noindent {\bf Proof.}
We consider the decomposition of the function $\varphi$ with
respect to the basis $\{e_{0,i}(u) \}$, that is
$$\varphi= \sum_{i=2}^{\infty}\varphi _i e_{0,i}(u).$$
Then the solution $w_\varphi$ is given by
$$w_\varphi(u,v)=\sum_{i=2}^{\infty}e^{-i (v-v_0)}\varphi_i e_{0,i}(u).$$
We recall that $\mu \in (-2,-1)$ so we have
$-i \leqslant \mu$ from which it follows 
$|w_\varphi|_{2,\a;[v,v+1]} \leqslant e^{ \mu (v-v_0)} |\varphi
|_{2,\a}$ and
$$\| w_\varphi \|_{{\mathcal C}^{2,\a}_\mu}= \sup_{v \in [v_0, \infty]}e^{- \mu v}
| w_\varphi  |  _{2, \a ;[v,v+1] }\leqslant    \sup_{v \in [v_0, \infty]}e^{-
\mu v} e^{ \mu (v-v_0)} |\varphi |_{2,\a}\leqslant e^{-\mu
v_0}|\varphi |_{2,\a}.$$
\hfill \qed

\begin{lemma}
\label{proprieta.funzione.armonica.1}
Let $u(r,\t)$ be the harmonic extension 
defined on $[r_0,+\infty) \times S^1$ of the even funcion 
$\varphi \in {\mathcal C}^{2,\a}(S^1)$ and such that $u(r_0,\t)=\varphi(\t).$
Then  $$\partial_\t u(r,\t-\pi/2)_{|r=r_0}= -r_0  \partial_r u(r,\t)_{|r=r_0}.$$
\end{lemma}
\noindent
{\bf Proof.}
If $\varphi(\t)=\sum_{i \geqslant 0} \varphi_i \cos (i \t),$ then
the function $u$ is given by
$$u(r,\t)=\sum_{i\geqslant 0} \varphi_i 
\left( \frac r {r_0} \right)^i \cos (i \t). $$
Then 
$$\partial_r u(r,\t)= \sum_{i\geqslant 1} \varphi_i 
\left( \frac r {r_0} \right)^i \frac {i \cos (i \t)} r$$
and
$$\partial_\t u(r,\t)= -\sum_{i\geqslant 0} \varphi_i
\left( \frac r {r_0} \right)^i  {i \sin (i \t)} .$$
Consequently 
$$\partial_\t u(r,\t-\pi/2)=
 -\sum_{i\geqslant 0} \varphi_i \left( \frac r {r_0} \right)^i  {i \cos (i \t)}$$
 from which lemma follows easily.
 \hfill \qed
 
 \begin{lemma}
\label{proprieta.funzione.armonica.2}
Let $u(r,\t)$ be the harmonic extension 
defined on $[0,r_0] \times S^1$ of the even funcion 
$\varphi \in {\mathcal C}^{2,\a}(S^1)$ and such that $u(r_0,\t)=\varphi(\t).$
Then  $$\partial_\t u(r,\t-\pi/2)_{|r=r_0}= r_0 \partial_r u(r,\t)_{|r=r_0}.$$
\end{lemma}
\noindent
{\bf Proof.}
If $\varphi(\t)=\sum_{i \geqslant 0} \varphi_i \cos (i \t),$ then
the function $u$ is given by
$$u(r,\t)=\sum_{i\geqslant 0} \varphi_i 
\left( \frac {r_0} r \right)^i \cos (i \t). $$
Then 
$$\partial_r u(r,\t)= -\sum_{i\geqslant 1} \varphi_i 
\left( \frac {r_0} r \right)^i \frac {i \cos (i \t)} r$$
and
$$\partial_\t u(r,\t)= -\sum_{i\geqslant 0} \varphi_i
\left( \frac {r_0} r \right)^i  {i \sin (i \t)} .$$
Consequently 
$$\partial_\t u(r,\t-\pi/2)=
 -\sum_{i\geqslant 0} \varphi_i \left( \frac {r_0} r \right)^i  {i \cos (i \t)}$$
 from which lemma follows easily.
 \hfill \qed

%% file: APPENDIXB.tex
\section{Appendix B}
\label{appendixb}
\noindent{\bf Proof of proposition \ref{operator.lame}.} 
Let $Z$ be the immersion of the surface $\widetilde M_{\s,\a,\b}$
and $N$ its normal vector. We want to find the differential 
equation to which a function $f$ must satisfy such that
the surface parametrized by $Z_f=Z+fN$ is minimal.
In section \ref{subsecLame} we parametrized 
the surface $\widetilde M_{\s,\a,\b}$ on the cylinder $\esf^1 \times \R.$
We introduced the map ${\bf z}(x,y):\esf^1\times [0,\pi[ \to \bar{\mathbb C}$ 
where $x,y$ denote the sphero-conal coordinates. 
We  start working with the conformal variables $p,q$  defined to be as the 
real and the imaginary part of  ${\bf z}.$ 
It holds that
$$ |Z_p|^2=|Z_q|^2=\Lambda, \quad |N_p|^2=|N_q|^2=-K \Lambda,$$
$$ \langle N_p, N\rangle= \langle N_q ,N \rangle=0, \quad   \langle Z_p ,Z_q\rangle=0, \quad 
 \langle N_p, N_q\rangle=0,$$
$$\langle N_q ,Z_q\rangle=- \langle N_p, Z_p\rangle, \quad \langle N_q ,Z_p\rangle=\langle N_p, Z_q \rangle,$$
so $$\langle N_p ,Z_p \rangle= |N_p||Z_p| \cos \gamma_1= \sqrt{-K}\Lambda \cos \gamma_1,$$
$$ \langle N_p ,Z_q \rangle=|N_p||Z_q| \cos \gamma_2=\sqrt{-K}\Lambda \cos \gamma_2.$$
Here $K$ denotes the Gauss curvature, $Z_p,Z_q$ and $N_p,N_q$ 
denote the partial derivatives of the vectors $Z$ and $N,$
 $\gamma_1$ is the angle between the vectors $N_p$ and $Z_p$,
$\gamma_2$ is the angle between the vectors $N_p$ and $Z_q.$\\

\noindent
The proof of proposition \ref{operator.lame} is articulated in some lemmas. 
 We denote by $E_f,F_f,G_f $ the coefficiens of the 
second fundamental form for the surface parametrized by $Z_f$.  
The following lemma gives the expression of the area energy functional.
\begin{lemma}
$$ A (f) := \int \sqrt{E_fG_f - F_f^2} \,
dp\, dq, $$
with
$$E_fG_f - F_f^2=\Lambda^2 + \Lambda (f_p^2   + f_q ^2) + 2 K \Lambda^2 f^2 + 2f ( f_q^2 - f_p^2  ) \sqrt{-K} \Lambda \cos \gamma_1 $$
$$ - 4 f f_p f_q \sqrt{-K} \Lambda  \cos \gamma_2 - K \Lambda f^2 (f_p^2 + f_q^2 ) + f^4 K^2 \Lambda ^ 2.$$
\end{lemma}

\noindent{\bf Proof.}
The coefficients of the second fundamental form are:
$$E_f= |\partial_p Z_f|^2=|Z_p|^2+f_p^2+f^2 |N_p|^2 + 2 f \langle N_p, Z_p\rangle,$$
$$ G_f=|\partial_q Z_f|^2 =|Z_q|^2+f_q^2+f^2 |N_q|^2 + 2 f \langle N_q, Z_q\rangle,$$
$$F_f=|\partial_p Z_f \cdot \partial_q Z_f| =f_p f_q + f (\langle Z_p, N_q \rangle+ 
\langle Z_q, N_p \rangle).$$
Then 
$$E_f G_f =   |Z_p|^2|Z_q|^2 +f_p^2  |Z_q|^2 + f_q^2 |Z_p|^2 +f^2 ( | N_q |^2 | Z_p |^2 + | N_p |^2 | Z_q |^2 )  +$$
$$ f^2 ( f_p^2 |N_q|^2 + f_q^2 |N_p|^2)+ f^4 |N_p|^2 |N_q|^2 + 
4 f^2 (\langle N_pZ_p \rangle)(\langle N_q Z_q \rangle)+
2f(f_p^2 \langle N_q,Z_q \rangle+f_q^2\langle N_p, Z_p \rangle)+$$ 
$$ f_p^2 f_q^2 +2f (\langle N_q ,Z_q \rangle|Z_p|^2 + \langle  N_p ,Z_p \rangle |Z_q|^2) +
 2f^3 (\langle N_q ,Z_q \rangle |Z_p|^2 +\langle  N_p, Z_p \rangle |Z_q|^2).$$
Since $\langle N_q, Z_q \rangle +\langle N_p, Z_p \rangle =0$ and 
 $|Z_p|^2=|Z_q|^2$ we can conclude that 
the last two terms of the previous expression are zero. 
Since $\langle N_q ,Z_p \rangle= \langle N_p, Z_q \rangle$ we have
$$ F_f=f_p f_q + 2f \langle N_p ,Z_q \rangle.$$
Then
$$ F_f^2 = f_p^2 f_q ^2 +  4 f^2 (\langle N_p, Z_q\rangle)^2 +
 4 f f_p f_q \langle N_p, Z_q\rangle.$$
So the expression of $E_fG_f-F_f^2$ is:
$$|Z_p|^2|Z_q|^2 +f_p^2  |Z_q|^2 + f_q^2 |Z_p|^2 +f^2 ( | N_q |^2 | Z_p |^2 + | N_p |^2 | Z_q |^2 )  +$$
$$ f^2 ( f_p^2 |N_q|^2 + f_q^2 |N_p|^2)+ f^4 |N_p|^2 |N_q|^2 + 4 f ^2 (\langle N_p,Z_p)
(\langle N_q, Z_q \rangle)+2f(f_p^2 \langle N_q, Z_q\rangle+f_q^2 \langle N_p, Z_p \rangle) $$
$$ -  4 f^2 ( \langle N_p, Z_q \rangle)^2 - 4 f f_p f_q \langle N_p, Z_q\rangle.$$
Ordering the terms we get:
$$|Z_p|^2|Z_q|^2 +f_p^2  |Z_q|^2 + f_q^2 |Z_p|^2 +f^2 ( | N_q |^2 | Z_p |^2 + | N_p |^2 | Z_q |^2 )   -  4 f^2 (\langle  N_p ,Z_q \rangle)^2$$
$$ + 4 f ^2 (\langle N_p,Z_p\rangle)(\langle N_q ,Z_q\rangle)  +2f(f_p^2 \langle N_q,Z_q\rangle+
f_q^2 \langle N_p, Z_p\rangle)  - 4 f f_p f_q \langle N_p ,Z_q\rangle+ $$
$$ +f^2 ( f_p^2 |N_q|^2 + f_q^2 |N_p|^2)+ f^4 |N_p|^2 |N_q|^2.$$
\noindent
The expression of $E_f G_f - F_f^2$ becomes:
$$ \Lambda^2 + \Lambda (f_p^2   + f_q ^2) - 2 K \Lambda^2 f^2 + 4 f^2 K \Lambda^2 \left( \cos^2 \gamma_1 + \cos^2 \gamma_2 \right)+$$
$$+ 2f ( f_q^2 - f_p^2  ) \sqrt{-K} \Lambda \cos \gamma_1  - 4 f f_p f_q \sqrt{-K} \Lambda  \cos \gamma_2 - K \Lambda f^2 (f_p^2 + f_q^2 ) + f^4 K^2 \Lambda ^ 2.$$
\noindent
Using the relations $\langle N_q, Z_p\rangle=\langle N_p,Z_q\rangle$ and 
$\langle N_q ,Z_q\rangle=- \langle N_p ,Z_p\rangle,$ it
is possible to understand that the relative positions of these
vectors are such that $\gamma_2 = \frac {\pi} 2 \pm \gamma_1   .$
So $\cos^2 \gamma_2= \cos^2 (\frac {\pi} 2 \pm \gamma_1    )= \sin^2 \gamma_1$ and $ \cos^2 \gamma_1 + \cos^2 \gamma_2= 1.$ Then
we can write:
$$ \Lambda^2 + \Lambda (f_p^2   + f_q ^2) + 2 K \Lambda^2 f^2 + 2f ( f_q^2 - f_p^2  ) \sqrt{-K} \Lambda \cos \gamma_1 $$
$$ - 4 f f_p f_q \sqrt{-K} \Lambda  \cos \gamma_2 - K \Lambda f^2 (f_p^2 + f_q^2 ) + f^4 K^2 \Lambda ^ 2.$$ \hfill \qed

\noindent
The next lemma completes the proof of the proposition \ref{operator.lame}.

\begin{lemma}
The surface whose immersion is given by $Z+fN,$
is minimal  if and only if $f$ satisfies 
$${\cal L }_\s f +Q_\s(f)=0,$$
where ${\cal L }_\s$ is the Lam\'e operator and 
 $Q_\s$ is a second order differential operator
 which satisfies
$$ \| Q_\s(f_2) - Q_\s(f_1)\|_{{\mathcal C}^{0, \alpha} (I_\s \times
[v,v+1])} \leqslant c \, \sup_{i=1,2} \| f_i\|_{{\mathcal C}^{2, \alpha}
(I_\s \times [v,v+1])} \,  \| f_2 -f_1\|_{{\mathcal C}^{2, \alpha}
(I_\s \times [v,v+1])} .$$
\end{lemma}
\noindent{\bf Proof.} The surface parameterized by $Z_f=Z+fN$ is minimal
if and only the first variation of $A(f)$ is $0$. That is 
$$ 2 \, D A_{|f} (g) = \int  \frac{1}{\sqrt{(E_fG_f - F_f^2)}_{|f=0}} 
D_f (E_fG_f - F_f^2 ) \, (g) \,dp \, dq.$$
Thanks to the previous lemma it holds that
$$ \frac 1 {\sqrt{(E_fG_f-F_f^2)}_{|f=0}} D_f(E_fG_f-F_f^2)(g) = \frac 1 {\Lambda} \left( 2 \Lambda  ( f_p g_p + f_q g_q ) + 4 K {\Lambda}^2 fg+\right.$$
$$ +2 \sqrt {-K} \Lambda \cos \gamma_1 \left[ 2 f f_q g_q + g f_q^2-  2 f f_p g_p - g f_p^2 \right ] +$$
$$ -4 \sqrt {-K} \Lambda \cos \gamma_2 \left[  f f_q g_p +  f g_q f_p + g f_p f_q \right ] +$$
$$ \left. -2K \Lambda \left[  f g f_p^2 +  f_p g_p f^2 +f g f_q^2 + f_q g_q f^2 \right ] + 4 K^2 \Lambda ^2 f^3 g  \right ).$$
\noindent
Reordering the summands, we have:
$$ \frac 1 {\sqrt{(E_fG_f-F_f^2)}_{|f=0}} D_f(E_fG_f-F_f^2)(g) = 2 \left( f_p g_p + f_q g_q  + 2 K {\Lambda} fg+\right.$$
$$ + \sqrt {-K}\cos \gamma_1 \left[ 2 f(f_q g_q-f_p g_p)+ g (f_q^2-  f_p^2) \right ] +$$
$$ -2 \sqrt {-K} \cos \gamma_2 \left[  f( f_q g_p +   g_q f_p) + g f_p f_q \right ] +$$
$$ \left. -K  \left[  f g(f_p^2 +f_q^2) +f^2(f_p g_p +f_q g_q) \right ] + 2 K^2 \Lambda  f^3 g  \right ).$$
\noindent
In the next computation we can skip the factor 2 in front of the
last expression.

$$ f_{p}g_p  + f_{q}g_q  + 2 K \Lambda fg + Q_1(f,f_p,f_q)g - Q_2(f,f_p,f_q)g_p - Q_3(f,f_p,f_q) g_q =0,$$
where $$ Q_1(f,f_p,f_q)= -(f_p^2-  f_q^2)\sqrt {-K}  \cos \gamma_1
-  2 f_p f_q  \sqrt {-K} \cos \gamma_2 - K f (f_p^2 +f_q^2) + 2
K^2 \Lambda  f^3,$$
$$ Q_2(f,f_p,f_q)= 2ff_p \sqrt {-K} \cos \gamma_1 +  2 f f_q  \sqrt {-K} \cos \gamma_2 + K f^2 f_p,$$
$$ Q_3(f,f_p,f_q)= -2ff_q \sqrt {-K} \cos \gamma_1 +  2 f f_p  \sqrt {-K} \cos \gamma_2 + K f^2 f_q.$$
\noindent
An integration by parts and a change of sign give us the equation:
$$\left( f_{pp}  + f_{qq}  - 2 K \Lambda f - Q_1(f,f_p,f_q) +\right.$$
$$\left. + P_2(f,f_p,f_q,f_{pp}, f_{pq},f_{qq}) + P_3(f,f_p,f_q,f_{pp}, f_{pq},f_{qq})\right) g =0,$$
where  $$P_2(f,f_p,f_q,f_{pp}, f_{pq},f_{qq})=\partial_p
Q_2(f,f_p,f_q) $$ and $$P_3(f,f_p,f_q,f_{pp}, f_{pq},f_{qq})=\partial_q  Q_3(f,f_p,f_q).$$

\noindent 
That is  $$P_2(f,f_p,f_q,f_{pp}, f_{pq},f_{qq})=2(f_p^2 + f
f_{pp}) \sqrt {-K} \cos \gamma_1+ 2( f_pf_q + f f_{pq} ) \sqrt
{-K} \cos \gamma_2+$$
$$+K(2f f_p^2+f^2f_{pp})+2f(f_p (\sqrt {-K} \cos \gamma_1)_p +f_q (\sqrt {-K} \cos \gamma_2)_p)+f^2 f_p K_p $$
and
$$P_3(f,f_p,f_q,f_{pp}, f_{pq},f_{qq})=-2(f_q^2 + f f_{qq}) \sqrt {-K} \cos \gamma_1+ 2( f_pf_q + f f_{pq} ) \sqrt {-K} \cos \gamma_2+$$
$$+K(2f f_q^2+f^2f_{qq})+2f(-f_q (\sqrt {-K} \cos \gamma_1)_q +f_p (\sqrt {-K} \cos \gamma_2)_q)+f^2 f_q K_q.$$

\noindent Now we change the variables, passing from the 
$(p,q)$ variables to the  $(u,v)$ variables. Then we want 
to understand how above differential equation 
changes.
 We recall that $p$ and $q$ 
are the real and imaginary part of the variable ${\bf z}$
of which we know the expression in terms of the
spheroconal coordinates $x,y$ (see \eqref{espressione.z}).
It is known that the metric $\bar g$ induced on a surface whose immersion $Z$ 
is given by the
Weierstrass representation on a domain of the complex ${\bf z}$-plane, can be
expressed in terms of the metric $d\bar s^2=dp^2+dq^2,$
by $\bar g=\Lambda (dp^2+dq^2),$ where $\Lambda=|Z_p|^2=|Z_q|^2.$ 
The Laplace-Beltrami operators
written with respect to the metrics $d \bar s^2$ and $\bar g$
are related by: 
$$\Delta_{d \bar s^2}=\frac 1 \Lambda \Delta_{\bar g}.$$  
That is they differ by the conformal factor $1/\Lambda.$
In section \ref{operator.lame} we observed that 
 the  conformal factor related to the change of coordinates 
 $(x,y) \to (u,v)$ is $-K/k.$
So the conformal factor due to 
the change of coordinates $(p,q) \to (u,v)$ is obtained by multiplication
of the conformal factors described above.   
Summarizing it holds that $$ f_{pp}+f_{qq}=\frac{-K \Lambda}{k}( f_{uu}+f_{vv}).$$
So we can write
$$\frac{-K \Lambda}{k} \left(f_{uu}+f_{vv}\right) + 2 \left( -K \Lambda \right) f +
R_1+R_2+R_3=0,$$ where  
$$ R_1(f,f_u,f_v)=  -\frac{-K \Lambda}{k} \left [ -( f_u^2-f_v^2) \sqrt{-K} \cos \gamma_1 - 2 f_u f_v \sqrt{-K} \cos \gamma_2 - Kf ( f_u^2+f_v^2) \right]  - 2 K^2 \Lambda f^3 $$
$$=\frac{-K \Lambda}{k} \left[ ( f_u^2-f_v^2) \sqrt{-K} \cos \gamma_1  +2 f_u f_v \sqrt{-K} \cos \gamma_2 + Kf ( f_u^2+f_v^2)
-2 K k f^3\right ] =$$
$$ \frac{-K \Lambda}{k}\bar P_1(f,f_u,f_v),$$
$$ R_2(f,f_u,f_v, f_{uu}, f_{uv},f_{vv}) = \frac{-K \Lambda}{k} P_2(f,f_u,f_v, f_{uu}, f_{uv},f_{vv})$$
and
$$ R_3(f,f_u,f_v, f_{uu}, f_{uv},f_{vv}) = \frac{-K \Lambda}{k} P_3(f,f_u,f_v, f_{uu}, f_{uv},f_{vv}).$$
We can write, simplifying the notation:
 $$ \frac{-K \Lambda}{k} \left[ f_{uu}+f_{vv} + 2
k(u,v) f + \bar P_1(f)+P_2(f)+P_3(f)  \right]=0.$$ We can recognize the
Lam\'e operator, $${\cal L}_\s f =  f_{uu}+f_{vv} + 2(\sin^2 \s
\cos^2 x(u) + \cos^2 \s \sin^2 y(v)) f, $$ then, if we set 
$Q_\s=\bar P_1(f)  +P_2(f)+ P_3(f),$  the equation can
be written 
 $${\cal L }_\s f +Q_\s(f)=0.$$
To show the estimate about $Q_\s$ is sufficient to show that 
all its coefficients are bounded.
In particular we will show that the Gauss curvature $K$ and its derivatives
$K_u, K_v$ are bounded.  
We start observing that 
$\frac{-K}{k(x(u),y(v))}$ is bounded.
It is well known that the Gauss curvature 
has the following expression in terms of the Weierstrass data $g,dh$: 
$$K= -16 \left(|g|+\frac{1}{|g|} \right)^{-4} \left| \frac{dg}{g}\right|^2
 \left|dh\right|^{-2}$$
We recall that $dh=\frac{\mu\,dz}{\sqrt{(z^2+\l^2)(z^2+\l^{-2})}}.$ 
Since $|z^2+\lambda^2 | |z^2+\lambda^{-2}|$ and 
$k(x,y)= \sin^2 \sigma \, \cos^2 x(u) + \cos^2 \sigma \, \sin^2 y(v )$ 
have the same zeroes,
that is the points $D,D',D'',D'''$ given by 
\eqref{eqbranchings1},  then $-K/k$ is bounded as well as 
its derivatives.\\
 
\noindent
We can give an estimate of the derivatives of $K$ and $\sqrt{-K}.$
We can write $\sqrt{-K}= \sqrt k\sqrt { \frac {-K} k}.$ 
From the observations made above it follows that it is sufficient  
to study the derivatives of $\sqrt k$ to show 
that the derivatives of $\sqrt{-K}$ are bounded.\\

\noindent
We recall that
$$l(x)=\sqrt{1-\sin^2\s \sin^2x} \quad m(y)=\sqrt{1-\cos^2 \s \cos^2y}.$$
From the expression of $k,$ using (\ref{change.coordinates}) it is easy to get:
$$\frac {\partial} {\partial u} \sqrt k=-\frac{\sin^2 \s \sin 2x(u)} 
{2\sqrt k} l(x(u)),$$
$$\frac {\partial} {\partial v} \sqrt k=\frac{\cos^2 \s \sin 2y(v)}
 {2\sqrt k} m(y(v)).$$
Then $$ \left|  \frac {\partial} {\partial u} \sqrt k   \right|= 
\frac {\sin^2 \s |\sin 2x(u)| l(x(u))} {2\sqrt{\sin^2 \s \cos^2 x(u)
+ \cos^2 \s \sin^2 y(v)} } \leqslant \frac{\sin^2 \s |\sin
2x(u)|}{2\sin \s |\cos x(u)|}\leqslant  \sin \s,$$

$$ \left|  \frac {\partial} {\partial v} \sqrt k \right|= 
\frac {\cos^2 \s |\sin 2y(v)| m(y(v))}{2\sqrt{\sin^2 \s \cos^2 x(u) +
\cos^2 \s \sin^2 y(v)} } \leqslant \frac{\cos^2 \s |\sin 2y(v)|}{2\cos
\s |\sin y(v)|}\leqslant \cos \s.$$ 
That is the derivatives of $\sqrt k$ (and consequently the 
ones of $\sqrt{-K}$)  are bounded. That completes the proof.\\
\hfill \qed\\

%% file: APPENDIXC.tex
\section{Appendix C}
\label{appendixc}
The differential equation 
\begin{equation}
\label{caso.l=1}
\sin y\, \partial_{y}\left(\sin y \,\partial_{y} f\right)
-j^2 f +2 \sin^2 y \,f =0
\end{equation}
 is a particular case ($l=1$) of the associated 
Legendre differential equation, that  is given by
$$\sin y\, \partial_{y}\left(\sin y \,\partial_{y} f\right)
-j^2 f +l(l+1) \sin^2 y \,f =0,$$
where $l,j \in \N.$
The family of the solutions of equation \eqref{caso.l=1} (see \cite{AS})
is $$c_1 P_l^j (\cos y)+c_2 Q_l^j(\cos y),$$ for $l=1,$
where  $P_l^j(t)$ and $Q_l^j(t)$ are respectively the associated
 Legendre functions of first and second kind.
 If $l=1$ these functions are defined as follows:

$$P_1^j(t)=\left\{
\begin{array}{cc}
t & \hbox{if} \quad j = 0\\
-\sqrt{1-t^2} & \hbox{if} \quad j = 1\\
0, & \hbox{if} \quad j \geqslant 2,
\end{array}
\right.
$$

$$Q_1^j(t)=(-1)^j \sqrt{(1-t^2)^j} \, \frac{d^jQ_1^0(t)}{dt^j}\quad
\mbox{ and }\ Q_1^0(t)=\frac{t}{2}\ln
\left(\frac{1+t}{1-t}\right)-1 .$$